\documentclass{amsart}
\pdfoutput=1
\usepackage{graphicx}    
\usepackage{multicol}    
\usepackage{subcaption}
\usepackage{enumitem}
\usepackage{url}
\usepackage{hyperref}

\usepackage{verbatim}
\graphicspath{{figures/}}

\renewcommand{\PackageWarningNoLine}[2]{}
\usepackage{amsmath}
\usepackage{amssymb}
\usepackage{amsthm}
\usepackage{amsfonts}
\usepackage{mathrsfs}
\usepackage{esint}
\usepackage{colonequals}                
\usepackage{color}
\usepackage{tikz}
\usepackage{cite}
\hypersetup{
    pdftitle = {A Variational Approach to Particles in Lipid Membranes},
    pdfauthor = {Charles M. Elliott, Carsten Gr\"aser, Graham Hobbs, Ralf Kornhuber, Maren-Wanda Wolf},
    pdfkeywords = {Elliptic pdes, variational methods, constrained minimization, finite elements}
    }

\newcommand{\N}{\mathbb{N}}

\newcommand{\R}{\mathbb{R}}
\newcommand{\Rinfty}{\R \cup \{\infty\}}

\providecommand{\op}[1]{\operatorname{#1}}

\newcommand{\cal}{\mathcal}
\newcommand{\cE}{\cal E}
\newcommand{\cJ}{\cal J}
\newcommand{\cJJ}{\tilde{\cal J}}
\newcommand{\aJ}{\tilde{a}}
\newcommand{\VJ}{\tilde{V}}

\newcommand{\cV}{\cal V}
\newcommand{\cG}{\cal G}
\newcommand{\cX}{\cal X}

\newcommand{\cM}{\cal M}
\newcommand{\cW}{\cal W}
\newcommand{\cB}{\cal B}

\newcommand{\cI}{\cal I}
\newcommand{\REL}{\mathrel{\widehat{=}}}
\newcommand{\indicator}{{\mathbf 1}}

\newcommand{\XX}{X_0}

\newcommand{\bfN}{\mathbf{N}}
\newcommand{\bfGamma}{\mathbf{\Gamma}}
\newcommand{\bfgamma}{\boldsymbol{\gamma}}

\newcommand{\soft}{\text{soft}}
\newcommand{\hard}{\text{hard}}

\newcommand{\ds}{ds}
\newcommand{\wc}{{\mkern 2mu\cdot\mkern 2mu}}

\newcommand{\HEllipse}{{\tikz[baseline=-.8ex]\filldraw(0,0) ellipse (0.8ex and 0.3ex);}}
\newcommand{\VEllipse}{{\tikz[baseline=-.8ex]\filldraw(0,0) ellipse (0.3ex and 0.8ex);}}

\providecommand{\argmin}[1]{\underset{#1}{\op{arg\,min\,}}}
\providecommand{\inlineargmin}[1]{{\op{arg\,min\,}_{#1}}}

\providecommand{\st}{\mathrel\vert}
\providecommand{\bigst}{\bigm\vert}

\providecommand{\Bigst}{\Bigm\vert}

\newtheorem{theorem}{Theorem}[section]
\newtheorem{lemma}{Lemma}[section]
\newtheorem{corollary}{Corollary}[section]
\newtheorem{proposition}{Proposition}[section]

\newtheorem{problem}{Problem}[section]
\newtheorem{remark}{Remark}[section]

\numberwithin{equation}{section} 
\numberwithin{figure}{section}

\makeatletter
\def\ifempty#1{\def\@temp{#1}\ifx\@temp\@empty}
\makeatother

\newcounter{assumptionCounter}

\newcommand{\assitem}[1][]{%
    \ifempty{#1}%
        \renewcommand{\theassumptionCounter}{A\arabic{assumptionCounter}}%
        \refstepcounter{assumptionCounter}
    \else%
        \renewcommand{\theassumptionCounter}{#1}%
        \addtocounter{assumptionCounter}{-1}
        \refstepcounter{assumptionCounter}
    \fi%
    \item[(\theassumptionCounter)]
}

\newenvironment{assumptions}{
    \begin{list}{}
    {
        \setlength{\leftmargin}{13mm}
        \setlength{\itemindent}{0mm}
        \setlength{\labelwidth}{13mm}
    }
}{\end{list}}

\newcounter{todocounter}
\newlength{\todowidthinner}
\usepackage{xifthen}
\usepackage{marginnote}
\DeclareRobustCommand{\MyChange}[3][\empty]{%
  {\color{#2}#3}
  \ifthenelse{\isempty{#1}}{}
  {%
    \addtocounter{todocounter}{1}%
    \ifmmode%
        {\color{#2}\text{$^{\framebox{\arabic{todocounter}}}$}}%
    \else%
        {\color{#2}\text{$^{\arabic{todocounter}}$}}%
    \fi%
    \marginpar{\textcolor{#2}{$^{\arabic{todocounter}}$\textnormal{#1}}}%
  }%
}
\definecolor{lightgray}{rgb}{0.7,0.7,0.7}
\usepackage[normalem]{ulem}

\definecolor{cecol}{rgb}{0,0.7,0}

\definecolor{ghcol}{rgb}{0.7,0.7,0}

\definecolor{mwcol}{rgb}{0,0.7,0.7}

\definecolor{revisioncol}{rgb}{1,0,0}

\begin{document}
\title
[Particles in Lipid Membranes]
{A Variational Approach to Particles in Lipid Membranes}

\author[Elliott]{Charles M. Elliott}
\address{Charles M. Elliott\\
Mathematics Institute\\
Zeeman Building\\
University of Warwick
Coventry CV4 7AL\\
United Kingdom}
\email{C.M.Elliott@warwick.ac.uk}

\author[Gr\"aser]{Carsten Gr\"aser}
\address{Carsten Gr\"aser\\
Freie Universit\"at Berlin\\
Institut f\"ur Mathematik\\
Arnimallee~6\\
D - 14195~Berlin\\
Germany}
\email{graeser@math.fu-berlin.de}

\author[Hobbs]{Graham Hobbs}
\address{Graham Hobbs\\
Mathematics Institute\\
Zeeman Building\\
University of Warwick
Coventry CV4 7AL\\
United Kingdom}
\email{g.hobbs@warwick.ac.uk.}

\author[Kornhuber]{Ralf Kornhuber}
\address{Ralf Kornhuber\\
Freie Universit\"at Berlin\\
Institut f\"ur Mathematik\\
Arnimallee~6\\
D - 14195~Berlin\\
Germany}
\email{kornhuber@math.fu-berlin.de}

\author[Wolf]{Maren-Wanda Wolf}
\address{Maren-Wanda Wolf\\
Freie Universit\"at Berlin\\
Institut f\"ur Mathematik\\
Arnimallee~6\\
D - 14195~Berlin\\
Germany}
\email{mawolf@math.fu-berlin.de}

\begin{abstract}
A variety of models for the membrane-mediated interaction of particles in lipid membranes,
mostly well-established in theoretical physics, is reviewed from a mathematical perspective.
We provide mathematically consistent formulations in a variational framework,
relate apparently different modelling approaches in terms of successive approximation,
and investigate existence and uniqueness.
Numerical computations illustrate that the new variational formulations are directly accessible to effective numerical methods.
\end{abstract}

\thanks{This work is partially supported by Freie Universit\"at Berlin
via the associated project {AP}01 of CRC 1114 
funded by Deutsche Forschungsgemeinschaft (DFG). G.H.  was supported by the UK Engineering
and Physical Sciences Research Council (EPSRC) Grant EP/H023364/1
within the MASDOC Centre for Doctoral Training.
The authors are grateful to Roland Netz and Thomas R.\ Weikl 
for various hints and stimulating discussions.}
\maketitle

\tableofcontents

\section{Introduction}\label{sec:intro}

The interplay of proteins and curvature of lipid bilayers is well-known to regulate cell
morphology and a variety of cellular functions, such as trafficking or signal detection.
Here, interplay not only means that proteins can induce curvature 
by shaping and remodelling the membrane, but also that the membrane 
curvature plays an active role in creating functional membrane domains 
and organizing membrane proteins, including their conformation dynamics, \cite{Lip91, McMahonGallupNature2005, SimVot15}. 
Moreover, microscopic causes, such as hydrophobic mismatch of proteins and amphiphilic lipids, 
may have macroscopic effects, such as budding or fission. 
For example, the membrane remodelling during endocytosis is a consequence
of the interplay between the elastic membrane and concerted actions of highly
specialized membrane proteins that can both sense and create membrane curvature. In the
case of  clathrin-mediated endocytosis more than forty different proteins are involved,
many of which are only transiently recruited to the 
plasma membrane~\cite{MeineckeBoucrotCamdereHonMittalandMcMahon2013}.

A popular particle-based approach to the simulation of proteins in lipid membranes are 
\emph{coarse-grained molecular models} in which the membrane constituents, i.e., 
lipids and proteins, are represented by short chains of beads, typically consisting of 3 -- 10 beads.
Often the coarse-grained models are simulated using dissipative particle dynamics (DPD), with
pairwise dissipative and random interactions between the beads that locally conserve
momentum and yield the correct hydrodynamics. DPD simulations can reach time and
length scales beyond those available by traditional molecular dynamics and thus
allow for studying, e.g., cluster formation, budding or anomalous diffusion of
transmembrane proteins \cite{RIH07, SGW08, SL07}. Moreover, recent coarse-grained models incorporate
structural and mechanical properties from experiments and all-atom bilayer simulations
and claim to be `semiquantitatively consistent' with experiments even on larger time
and length scales~\cite{brannigan2008model,laradji2011coarse,SaundersVoth2014,WD10}.

On the macroscopic side of the model hierarchy, there are pure \emph{continuum models}
based on the fundamental Canham-Helfrich (CH) model of lipid membranes~\cite{Can70,Hel73}.
The associated Canham-Helfrich (CH) energy is obtained by expansion 
of the bending energy with respect to the principal curvatures up to second order.
It is noteworthy that the CH energy is related to the much studied
Willmore functional in  differential geometry~\cite{KuwSch04-a,Wil93}. 
The CH model describes equilibrium
and close-to-equilibrium properties of biological membranes, such as the entropic
repulsion of fluctuating membranes or the equilibrium shapes of lipid membranes.
It has been modified so as to include the effect of different lipid
components  or large-scale protein assemblies. 
Lipids and proteins are then represented by areal concentrations. Their
impact on the membrane morphology is modelled by concentration-dependent bending
rigidities, phase-dependent spontaneous curvature, and a line energy associated with
the phase boundaries~\cite{Lip92,JulLip96}. 
The Euler-Lagrange equations associated with CH energy functionals usually are
fourth-order, nonlinear (geometric) partial differential equations.
These equations are typically coupled with nonlinear partial differential equations describing the evolution of
lipids and proteins on the surface. The numerical solution of these coupled systems is
an emerging field of current research 
(for an overview see, e.g., \cite{DziEllHui11,GarNieRoe08, DziEll13-a, DecDziEll05}). 
In numerical computations, binary and multicomponent models of membranes were able
to reproduce the typical equilibrium shapes of vesicles, such as dumbbells, discocytes
or starfishes (see, e.g., \cite{BarGarNur08-c, Du11, EllSti13, EllSti10,HuWeiLip11} and the references
cited therein). 

%
%

\emph{Hybrid models} finally intend to bridge the gap between  molecular dynamics based models
that are expensive and have certain limitations in terms of the accessible time and
length scales, and continuum models with continuous densities of proteins that do not incorporate the effect of small numbers of discrete particles. 
The basic idea is to treat proteins as discrete rigid particles 
coupled to the continuous membrane by suitable interface conditions.
Aiming at moderate length scales, 
the  continuous membrane  is often described  by a linearised CH model (Monge gauge).
We emphasize that only mechanical curvature-induced interaction of particles 
and no chemical forces are considered in this context. 
However, additional chemical forces between particles could be incorporated by appropriate additional potentials and, under suitable conditions on these potentials, our mathematical analysis can be easily extended to this case.
While the history of hybrid models can be traced back
to the early nineties and meanwhile became a well-established field of research in 
theoretical physics 
(see, e.g.~\cite{DomFou99,DomFou02,GouBruPin93,HelfrichJakobsson90,Huang86,KimNeuOst98,MarMis02,Net97,ParLub96,WeiKozHel98,YolDes12, YolHauDes14, RauRowTur15} only to mention a few of  a multitude of references)
the mathematical and numerical analysis of hybrid models is still in its infancy.

This paper is devoted to a mathematical consistent formulation of  a variety of existing and some new models
which reveals a kind of hierarchy of  different modelling approaches 
 in terms of successive approximation and is directly accessible to effective numerical methods.
We start with, in a sense, most detailed hybrid models of proteins in lipid membranes that are
based on the coupling of rigid particles of finite size to a linearised CH membrane by suitable
conditions on the membrane displacement and its normal derivative (angle condition) along the
boundaries of the particles~\cite{GouBruPin93,HelfrichJakobsson90,Huang86,KimNeuOst98,ParLub96,WeiKozHel98}.
The boundary values are determined by the specific interaction at the interface
between the hydrophobic belt of transmembrane proteins and the surrounding lipid bilayer.
We derive an equivalent reformulation of the  corresponding fourth order boundary value problem 
by incorporating the boundary conditions  in terms of constraints along the particle contour lines (curve constraints)
in a similar way as in fictitious domain methods~\cite{GlowinskiPanPeriaux94}. 
Variation of height and tilt is represented by additional degrees of freedom.  
This approach and most of our theoretical results also extend to rotation of particles
which, however, will be considered elsewhere~\cite{C04-wolfDiss15}.
For fixed location of particles we prove existence and uniqueness of minimizers of the CH energy 
under curve constraints, together with convergence of a penalty formulation for vanishing penalty parameter.
A gradient flow approach to varying locations of particles with soft-wall constraints is discussed only briefly
as more detailed investigations are the subject of current research.

Approximating the  conditions along the particle boundaries by their (constant) mean value,
we introduce a novel class of finite-size particle models. The resulting averaged hybrid models
no longer  impose any conditions on the contour of the membrane along the particle boundaries and
no longer provide a representation of tilt.   
By Green's formula, the angle conditions along the particle boundaries  now take the form of 
averaged curvature conditions over the area of particles.
In addition to existence and uniqueness for corresponding curvature-constrained 
minimization problems and its penalized counterparts,
we also prove existence of global minimizers for varying locations of particles with hard- and soft-wall constraints.

It is sometimes convenient, both for analytical and numerical purposes,
e.g., when proceeding to larger copy numbers or to larger spatial scales, 
to approximate  averages of  curvature  over the whole area of finite-size particles
by  weighted point values of curvature, e.g., in the barycenters.
The resulting point-like hybrid models have quite a history in theoretical physics,
see, e.g.~\cite{BarFou03,DomFou99,DomFou02,KimNeuOst98,MarMis02,NajAtzBro09,Net97, WeiDes13}
and also \cite{ParLub96, GouBruPin93}.
However, as  point curvature constraints are not well-defined for functions with only
second order weak derivatives, the resulting problem is not well-posed from a mathematical point of view.
While this issue is often addressed by suitable truncation of Fourier expansions in physics literature (see, e.g., \cite{DomFou02}),
we consider an extension of the linearised CH energy by regularizing third and fourth order terms (cf.~e.g.,~\cite{BarFou03}).
From a physical point of view these additional terms could be justified in terms of higher moments 
in the expansion of the membrane bending energy~\cite{C04-Mitov1978}.
Such kind of  extended CH energy then allows for existence and uniqueness results for point-like particles
that are in complete analogy to averaged hybrid models.
We also provide a representation of global minimizers in terms of suitable Green's functions.
For unbounded domains,  we thus recover a representation of solutions that was
first suggested by Bartolo and Fournier~\cite{BarFou03}.

For particles interacting with the cytoskeleton, we consider  two different classes of models,
describing  the interaction with  the membrane by point values 
or by point forces~\cite{EvaTurSen03, GovGop06,SenTur04}.
As point values of functions with second-order weak derivatives are well-defined, 
the classical linearised CH membrane energy is used in both modelling approaches.
We prove existence and uniqueness for prescribed point values at fixed locations
and derive representations of the solutions in terms of suitable Green's functions.
In addition to  existence of global minimizers for varying locations, we prove
that solutions for $N$ particles are entirely determined by the two-particle
problem considering only the particles with the largest and smallest prescribed point values.
Again, we derive  representations of the solutions of these problems in terms of Green's functions.
Similar results are presented for prescribed point forces.
Apart from existence and uniqueness results for fixed location of particles,
we derive a representation of global minimizers in terms of Green's functions and
prove clustering of point forces to one or two clustering points depending on the 
sign of point forces. In this way, computation of global minimizers with $N$ point forces can
be reduced to a problem with at most $N=2$ point forces.

The paper concludes with some numerical experiments.
First, we compute the interaction potential for two circular or elliptical particles
as described by the finite-size hybrid model with coupling conditions on the 
membrane displacement and its normal derivative at the particle contour lines 
and compare the results with those obtained from a related point-like model. 
As a second scenario, we numerically investigate the interaction potential of two point forces and, in particular, 
illustrate  how clustering depends on the ratio of bending rigidity and membrane tension.

The paper is organised as follows. 
After a short introduction to the biophysical background in Section~\ref{sec:PARTLIP},
Section~\ref{sec:MATHMEM} summarizes
the mathematical description of lipid membranes by the CH energy.
Section~\ref{sec:rpfs} is devoted to  finite-size hybrid models performing the coupling
by  conditions on the membrane displacement and its normal derivative  (angle condition) 
at the particle boundaries while Section~\ref{sec:AVCU}  concentrates on averaged hybrid models 
arising from approximations of these boundary conditions by mean values of curvature.
Hybrid models with point-like particles are introduced by approximating those mean values 
by point curvature constraints and are analysed in Section~\ref{sec:pcc}. 
Models for particles interacting with the cytoskeleton
by point values and point forces are considered in Section~\ref{sec:pvc}.
The paper concludes with some numerical experiments collected in Section~\ref{sec:NUMEX}.
The theoretical findings in Section~\ref{sec:rpfs} -- \ref{sec:pvc} are based on an abstract variational framework 
that is presented in the Appendix together with  some required regularity results on Green's functions.

\section{Particles in lipid membranes} \label{sec:PARTLIP}

A biomembrane is a thin layer that surrounds biological cells and cellular organelles  acting as a barrier between the cell and its surroundings.
It consists of a lipid bilayer with embedded and attached proteins. The lipid bilayer is composed of phospholipids, 
made up of a phosphate group and a diglyceride. The resulting form is that the molecules have a head, 
the phosphate group, and a tail comprised of the two fatty acid hydrocarbon chains that make up the diglyceride. 
The heads are hydrophilic and the tails hydrophobic.  As a consequence, when 
placed in water these molecules form structures 
in which the heads point outwards and the tails inwards. 
There are a number of such formations, 
but the one of interest here is the bilayer sheet. This is when the phospholipids line up so that the heads form 
two distinct layers, sandwiching the tails between them as depicted in Figure \ref{fig:bilayer} (taken from \cite{McMahonGallupNature2005}).

\begin{figure}[ht]
\centering
\includegraphics[width=0.59\textwidth]{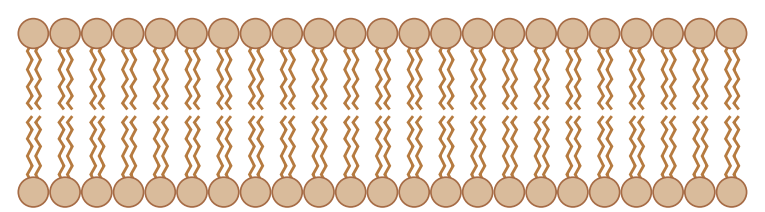}
\caption{Phospholipid bilayer sheet}
\label{fig:bilayer}. 
\end{figure}

The biological membrane is composed of different types of lipid molecules, which may differ in their head groups, in 
the length of their hydrocarbon chains, or in the number of unsaturated bonds within these chains, which is associated with 
different conformations (phases) leading to order (dense packing) or disorder (light packing).
Since the different lipid types and phases tend to separate,
 the composition of such a multicomponent membrane can become laterally inhomogeneous.
 
Another crucial source of inhomogeneity are proteins that are
responsible for the diverse 
functions performed by biological cells, which range from the transport of specific molecules across the membrane
to the reception of chemical signals from the extracellular environment. 
Transmembrane or embedded proteins,  at least partially, penetrate the membrane  
while  peripheral proteins, (for example  BAR domains  \cite{McMahonGallupNature2005}),  are attached to it by electrostatic interaction 
or other weak binding forces.
A pictorial description of a transmembrane  protein  is provided in the left picture of  
Figure~\ref{fig:membrane_deformation} (taken from \cite{McMahonGallupNature2005}).
Interaction of embedded proteins with the membrane is characterised by the shape of their hydrophobic belt
whose width might be different from the membrane thickness and might
vary along the particle contour line.
In contrast, peripheral proteins usually impose their shape and thus curvature 
to the lipid membrane like a scaffold.
Both embedded and peripheral proteins  may  not only tilt and  move up and down with the membrane,
but also move laterally and eventually cluster according  to
mechanical forces induced by membrane curvature.

\begin{figure}
\begin{subfigure}[bt]{0.45\linewidth}
\centering
\includegraphics[width=0.75\textwidth]{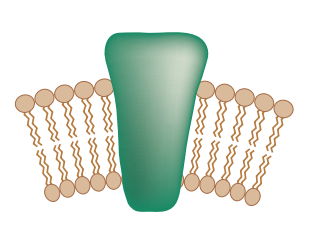}
\vspace{0.62cm}
\caption{transmembrane protein}
\label{fig:Inclusion}
\end{subfigure}
\begin{subfigure}[bt]{0.48\linewidth}
\centering
\includegraphics[width=0.6\textwidth]{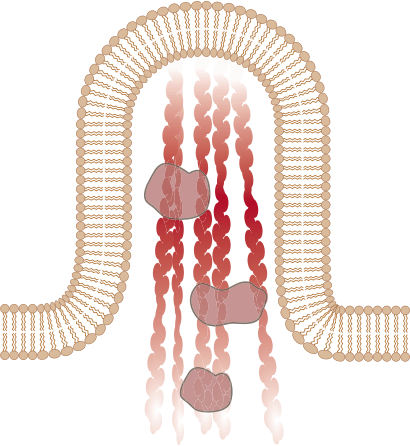}
\caption{filaments anchored in the cytoskeleton}
\label{fig:Filaments}
\end{subfigure}
\caption{Protein-induced  membrane deformation} \label{fig:membrane_deformation}
\end{figure}

Cytoskeletal protein assemblies  such as actin filaments or bundles of actin filaments 
locally impose membrane curvature  by applying a point force or constraining the 
membrane to take a fixed height. 
A cartoon picture is given in Figure~\ref{fig:Filaments} 
(taken from \cite{McMahonGallupNature2005}). 
Branching, bundling, and treadmilling of actin filaments is, e.g., responsible for the formation of filopodia which are  slim, finger-like projections 
of a cell membrane that are formed by the bundling of actin filaments attached 
to the cell cytoskeleton  at one end  and pushing against the cell membrane 
at the other~\cite{MattilaLappalainen08}. 
For further information on the mechanical interaction of particles in lipid bilayers
we refer, e.g.,  to the overviews \cite{McMahonGallupNature2005,RosVir07}  and the literature cited therein.
  
\section{Mathematical description of lipid membranes} \label{sec:MATHMEM}

\subsection{Canham-Helfrich free energy}
As the width of a bilayer ($10^{-9}m$) is much smaller than its lateral extension, 
it is natural to model the membrane as a smooth, two-dimensional
hypersurface $\mathcal M$
embedded in the three-dimensional Euclidean space $\mathbb R^{3}$.  
Note that this simplification neglects transversal 
stretching  and transversal shearing  as possible elastic deformations. 
Since the fluidity of the membrane excludes lateral shearing, 
the deformations of the membrane are caused by lateral stretching and bending.

The mathematical study of biomembranes principally concerns the minimization of the energy functional 
describing the energy associated to displacements of $\mathcal M$.
Fundamental to the  macroscopic approach to modelling biomembranes is the 
Canham-Helfrich (CH) model \cite{Can70,Hel73, Eva74}
which is based on the expansion of the bending energy 
with respect to the principal curvatures up to second order. It describes
equilibrium and close-to-equilibrium properties of 
biological membranes. 
The fundamental object of the CH model is the elastic bending energy ${\mathcal J_{CH}}$ defined by
\begin{equation} \label{eq:CANHAMHELFRICH}
    \cJ_{CH}({\mathcal M})
    =\int_{{\mathcal M}}
    {\textstyle \frac{1}{2}}\kappa (H-c_0)^2  
    + \kappa_G K \; d{\mathcal M}.
\end{equation}
Here $H$ and $K$ stand for mean and
Gaussian curvature of the membrane ${\mathcal M}\subset{\mathbb R}^{3}$ while $\kappa$
and $\kappa_G$ are the corresponding bending rigidities 
and $d{\mathcal M}$ is the
surface element of ${\mathcal M}$. 
The additional parameter $c_{0}$ is called
\emph{spontaneous curvature} and accounts for a possible asymmetry between the outer and inner layers  in the otherwise flat
reference configuration, e.g., due to different lipid compositions in the layers.
The related energy
\begin{equation}\label{CHSurfEnergy}
    \cJ_{CHS}({\mathcal M})=
    \int_{{\mathcal M}}{\textstyle \frac{1}{2}}\kappa (H-c_0)^2 
    + \kappa_G K +\sigma \; d{\mathcal M}
\end{equation}  
supplements the bending energy with a surface energy 
$\int _{\mathcal M} \sigma\; d{\mathcal M}$ that is
associated with membrane tension $\sigma \geq 0$.
Here, the surface energy penalises area change and thus accounts for the 
incompressibility constraint of the fluid membrane in the lateral direction.
These energies, depending on the type of problem, may be 
augmented with reduced volume or bilayer area
difference constraints~\cite{Sei97}.
The mathematical derivation of Canham-Helfrich-type 
models from molecular descriptions of lipid bilayers 
by $\Gamma$-convergence techniques
has been started only recently ~\cite{BloPel04,PelRoe08}.

Note that for a general membrane $\mathcal M$ the Gaussian curvature term
$\int_{\mathcal M}\kappa_G K d{\mathcal M}$ gives a non-constant contribution to $\cJ_{CHS}(\mathcal M)$.
However, assuming that $\mathcal M$ is closed or that the geodesic curvature
along $\partial\mathcal M$---composed by the membranes outer boundary
and possible particle contour lines---is fixed, this term becomes
a topological invariant by the Gauss-Bonnet theorem.
Hence, it can be ignored 
when computing equilibrium membrane shapes
minimizing $\cJ_{CHS}(\mathcal M)$.
For a detailed discussion of geodesic curvature along particle contour lines
we refer to~\cite{C04-wolfDiss15}.

\subsection{Monge gauge} \label{subsec:MONGE}
We will now outline a geometrically linearised approximation
of the Canham-Helfrich energy  $\cJ_{CHS}$ defined in \eqref{CHSurfEnergy}.
For simplicity, let us assume that  spontaneous curvature $c_{0}$ is zero, 
so that we have
\begin{equation}\label{energy}
\cJ_{CHS}(\mathcal M)
=\int_{\mathcal M} \textstyle \frac{1}{2}\kappa H^2 
+ \kappa_G K + \sigma \; d{\mathcal M}.
\end{equation}
In the Monge gauge, one assumes that the surface is nearly flat, so that   
the membrane surface can be parametrized as a graph
\begin{align}
    \label{eq:monge}
    {\mathcal M}=\{(x_{1},x_{2},u(x_{1},x_{2}))\st(x_1,x_2)\in \Omega\}
\end{align} 
over a two-dimensional reference domain $\Omega\subset \R^2$.
Then, the mean curvature $H$ and the Gau\ss\, 
curvature $K$ of the membrane $\mathcal M$ are given by
\begin{equation*}
    H = -\nabla \cdot \frac{\nabla u}{(1+|\nabla u|^{2})^{1/2}}, \qquad
    K = \left(\frac{\partial^{2}u}{\partial x_1^{2}}\frac{\partial^{2}u}{\partial x_2^{2}} 
    - \left(\frac{\partial^{2}u}{\partial x_1\partial x_2}\right)^2\right) \left /  {(1+|\nabla u|^{2})^{1/2}}\right. .
\end{equation*} 

A common approach to derive an approximate model is to assume 
that the displacement of the membrane from the $x-y$ plane produced by the particles is small, i.e. $|\nabla u|\ll 1$. 
In this case, it is sufficient to consider the geometric linearisation
\begin{equation} \label{eq:GEOLIN}
(1+|\nabla u|^{2})^{1/2}\; \rightsquigarrow \; 1 +{\textstyle \frac{1}{2}}|\nabla u|^{2},\quad
    H \;\rightsquigarrow \; -\Delta u,\quad
    K\; \rightsquigarrow \; \frac{\partial^{2}u}{\partial x_1^{2}}\frac{\partial^{2}u}{\partial x_2^{2}} - \left(\frac{\partial^{2}u}{\partial x_1\partial x_2}\right)^2,
\end{equation}
which models perturbations from a flat surface.
Inserting  the geometric linearisation \eqref{eq:GEOLIN} of mean curvature, Gauss curvature, and the surface element $d{\mathcal M}= (1+|\nabla u|^{2})^{1/2}dx$ 
into \eqref{energy} yields, up to a constant term, the quadratic energy 
\begin{equation}\label{energy_monge}
    \cJ(u) = \int_{\Omega} \textstyle \frac{1 }{2}\kappa(\Delta u)^2 
    + \kappa_G \left(\frac{\partial^{2}u}{\partial x_1^{2}}\frac{\partial^{2}u}{\partial x_2^{2}} - 
    \left(\frac{\partial^{2}u}{\partial x_1\partial x_2}\right)^2\right)  + \frac{1 }{2} \sigma |\nabla u|^2 \; dx.
\end{equation}
Ignoring Gaussian curvature in light of the Gauss-Bonnet theorem discussed above,
a quadratic approximation of the energy
$\cJ_{CHS}$ from \eqref{energy} finally takes the form
\begin{equation}\label{eq:monge_energy}
    \cJ(u)= 
    \int_\Omega {\textstyle \frac{1}{2}} \kappa(\Delta u)^2  + {\textstyle \frac{1}{2}} \sigma|\nabla u|^{2} \; dx .
\end{equation}
We emphasize that  Gaussian curvature 
may still vary across the membrane allowing for saddle-like deformations 
(cf., e.g., the numerical results in Subsection~\ref{subsec:FBCC}).
Observe that minimization of $\cJ$ is leading to fourth order plate-like equations.


\subsection{Boundary conditions and coercivity}\label{sec:BC}
We consider the Canham-Helfrich free energy $\cJ(u)$ in the Monge gauge  \eqref{eq:monge_energy} with
membrane displacement $u$ defined on a reference domain $\Omega \subset \R^2$.
From now on, we assume that  $\Omega\subset \R^2$ is a bounded, convex domain 
with a piecewise smooth Lipschitz boundary $\partial \Omega$, e.g., a square.
Only in Section~\ref{subsec:UNBD}, we shall consider $\Omega =\R^2$.
Deformations $u$ of the membrane
are taken  from a closed subspace  $V\subset H^2(\Omega)$ 
satisfying suitable boundary conditions. 
We consider  the following three cases
\begin{equation} \label{eq:VDEF}
    V= H^2_0(\Omega),\qquad V= H^2(\Omega) \cap H^1_0(\Omega), \qquad V= H^2_{p,0}(\Omega),
\end{equation}
often referred to as Dirichlet, Navier,  and periodic boundary conditions with zero mean, respectively.
Note that 
\[
   H^2_{p,0}(\Omega)=\overline{\{v|_\Omega \st v \in C^\infty(\mathbb{R}^2) \text{ is } \Omega \text{-periodic and } \textstyle \int_{ \partial \Omega}v \; ds=0\}}.
\]
is only defined for a rectangular domain $\Omega$.
The space $V$ is equipped with the canonical  norm 
$\|\wc\|_{2}= \|\wc\|_{H^2(\Omega)}$  in $H^2(\Omega)$.
Throughout the following, we assume that $\kappa >0$ as well $\sigma \geq 0$ for
all three choices of $V$.

\begin{lemma}\label{lem:GLOBAL_COERCIVITY}
    The bilinear form 
    \begin{equation*}
    a(v,w) = \int_{\Omega} \kappa \Delta v \Delta w\; dx + \sigma \nabla v \cdot \nabla w \; dx, \qquad v,w \in V,
    \end{equation*}
    associated with the energy functional $\cJ$ is continuous and coercive on $V$.
\end{lemma}
\begin{proof}
    While continuity of $a(\wc,\wc)$ is obvious, we refer to \cite{Graeser2015} for a proof of coercivity.
\end{proof}

Let  $\Omega' \subset \Omega$  be any subset with sufficiently smooth boundary.
It is convenient to introduce the energy functional 
\begin{equation}\label{eq:monge_energy_subset}
    \cJ_{\Omega'}(u)= \int_{\Omega'} {\textstyle \frac{1}{2}} \kappa(\Delta u)^2  + {\textstyle \frac{1}{2}} \sigma|\nabla u|^{2}\; dx
\end{equation}
with the associated  bilinear form
\begin{equation}\label{eq:bilinearform}
    a_{\Omega'}(v,w) = \int_{\Omega'} \kappa \Delta v \Delta w \; dx + \sigma \nabla v \cdot \nabla w \; dx.
\end{equation}
Furthermore for any subset $\Omega' \subset \Omega$ we define
\begin{align*}
    V_{\Omega'} = \bigl\{v|_{\Omega'} \bigst v \in V\bigr\}
\end{align*}
with the associated norm $\|\wc\|_{2,\Omega'} = \|\wc\|_{H^2(\Omega')}$.
For notational convenience, we set $\cJ=\cJ_{\Omega}$ and
$a(\wc,\wc)=a_{\Omega}(\wc,\wc)$ in the sequel.


\section{Curve constraints}\label{sec:rpfs}

\subsection{Particles interacting with lipid membranes}
Transmembrane proteins are 
interacting with the membrane curvature by the shape of 
the hydrophobic belt of single molecules
or by oligomerisation, i.e., macromolecular clustering~\cite{McMahonGallupNature2005}. Other particles, such as
peripheral proteins, may be attached to the membrane surface, acting as active
or passive scaffolds~\cite{McMahonGallupNature2005},
others may be partially wrapped due to adhesion 
energy~\cite{BahramiEtAl2014,KoltoverEtAl1999}.
All these phenomena can be captured by the same type of mathematical model
that treats each particle as a finite-sized, rigid body $B$ 
that interacts with the membrane by 
suitable conditions on the common interface~$\Gamma$.

\subsection{Boundary values and curve constraints}  \label{sec:BVP}
We assume that the membrane occupies a graph over a subset
\[
    \Omega_B = \Omega \setminus \bigcup_{i=1}^N \overline{B_i}
\]
of the given reference domain $\Omega\subset \R^2$ with
non-empty, open subsets $B_i\subset \Omega$  representing $i=1,\dots,N$ 
particles with diameter  $2r_i$ and centres of mass $X_i$ included in the membrane.
For example, the particles $B_i$ could be (but do not have to be) 
circles with radii $r_i$ and midpoints $X_i$.
We assume that the particles have $C^3$-boundaries  $\Gamma_i=\partial B_i$. 
We also assume that the particles do not overlap 
and do not touch the boundary $\partial \Omega$ of $\Omega$
in the sense that $B_i\cap B_j= \overline{B_i} \cap \partial \Omega= \emptyset$
for all $i\neq j=1, \dots N$.
In the sequel, we will consider varying the height and location of  particles.
We will also consider linearized variation of tilt
(by neglecting  the corresponding variation of $B_i$),
but exclude rotation for simplicity.

We assume that the deformation $u$ defined on $\Omega_B$
satisfies (artificial) boundary conditions on 
$\partial \Omega$ as specified implicitly in the preceding section 
(via the function spaces in which $u$ lies)
and additional  boundary conditions on the 
particle boundaries $\Gamma_i$ that determine the membrane-particle coupling.
For a (a) transmembrane protein, (b) scaffold, and (c) partly wrapped particle
this is illustrated in Figure~\ref{C04-fig:INTERACTION1}.
For a transmembrane protein, the particle boundary $\Gamma$ of $B$ (dashed line) 
is the projection of the protein-membrane contact line $\gamma\subset \R^3$ 
(only shown as cut with the mid-surface of the protein)  to $\R^2$,
the height $h$ is the  (spatially varying) distance from $\gamma$ to $\R^2$, 
and $-s$ describes the (spatially varying) angle 
between the bilayer mid-plane and the horizontal outward normal of $B$.
Similarly, for a peripheral protein, 
$B$ stands for the projection of the particle-membrane contact area,
while this set is augmented by part of the membrane thickness for partly wrapped particles.

\begin{figure}[ht]
  \begin{subfigure}[b]{0.31\linewidth}
    \centering
    \includegraphics[scale=0.18]{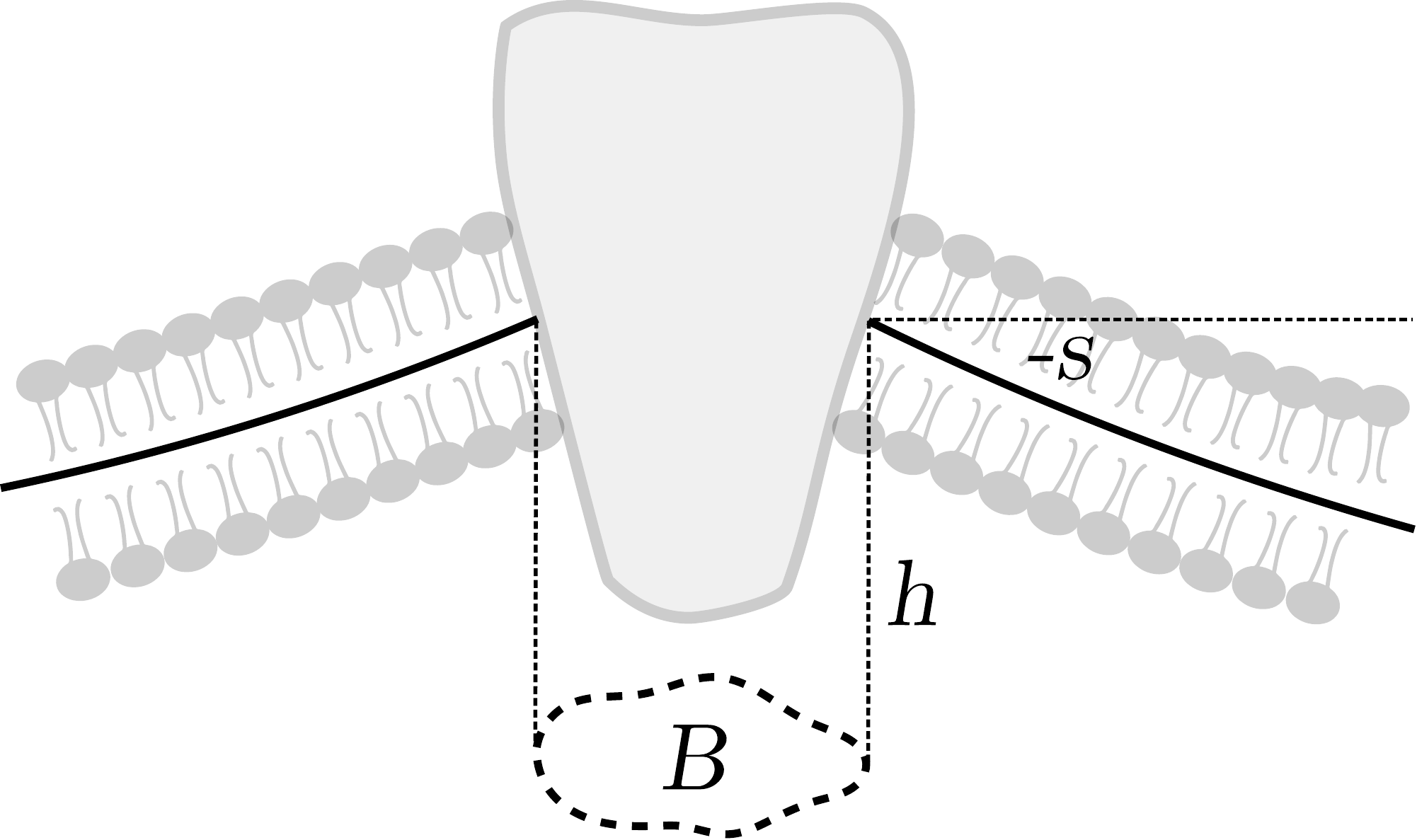}
       \caption{transmembrane protein} \label{fig:inclusion}
  \end{subfigure}
  \begin{subfigure}[b]{0.31\linewidth}
    \centering
    \includegraphics[scale=0.18]{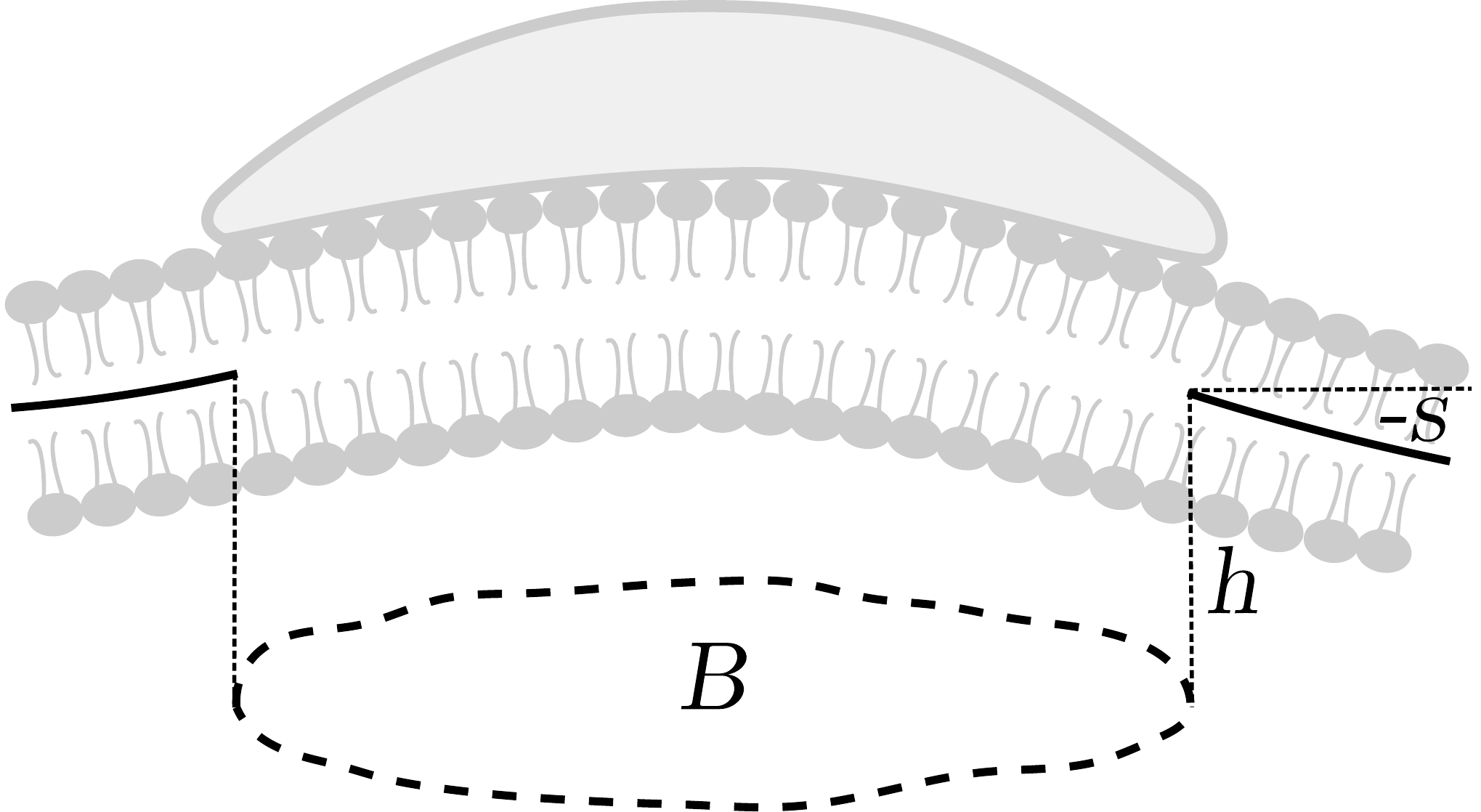}
       \caption{scaffold}
     \end{subfigure}
  \begin{subfigure}[b]{0.31\linewidth}
    \centering
    \includegraphics[scale=0.18]{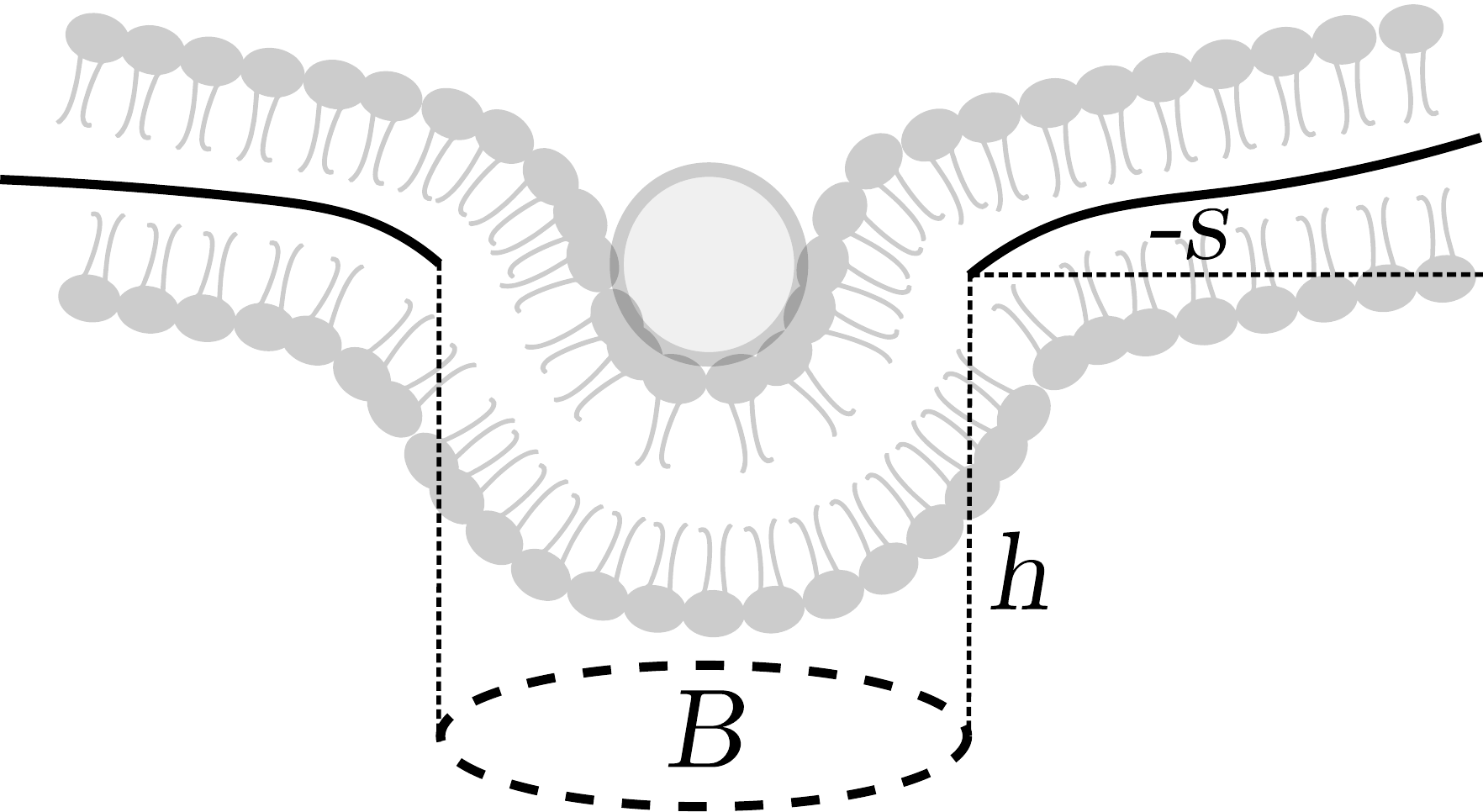}
       \caption{partly wrapped particle}
      \end{subfigure}
  \caption{Coupling of a membrane with different types of particles}
   \label{C04-fig:INTERACTION1}
\end{figure}


Assuming that height, tilt, and location of particles are fixed,
these  additional essential (Dirichlet)  boundary conditions are given by
\begin{equation} \label{eq:INTBC}
    u=h_i \in H^{\frac{3}{2}}(\Gamma_i),\quad \quad \textstyle  \frac{\partial}{\partial n}u = s_i \in H^{\frac{1}{2}}(\Gamma_i) \quad \text{ on }\Gamma_i,
\qquad i=1,\dots,N.
\end{equation} 
Here, and throughout the paper, $n$ denotes the outward normal to $\Omega_{B}$
(while the inward normal was more convenient in Figure~\ref{C04-fig:INTERACTION1}).
Furthermore,
\begin{equation} \label{eq:BNDDATA}
    h_i =h_i^0(\wc -X_i), \qquad s_i=s_i^0(\wc - X_i),\quad  i=1,\dots,N,
\end{equation}
holds with given reference data $h_i^0\in H^{\frac{3}{2}}(\Gamma_i^0)$, 
$s_i^0 \in H^{\frac{1}{2}}(\Gamma_i^0)$ on the boundaries  
$\Gamma_i^0=\Gamma_i-X_i$ of reference particles  $B_i^0=B_i-X_i$ 
with location $X_i^0=0$. 
We assume that the reference particles $B_i^0\subset \R^2$ are open sets
with $C^3$-boundaries, $X_i^0=0 \in B_i^0$, and thus $X_i \in B_i$.
The essential boundary conditions \eqref{eq:INTBC} 
are incorporated in the affine subspace 
\begin{equation} \label{eq:AFFINE_BC_SUBSPACE}
    V_{\Omega_B}^{h,s} =
    \left\{v \in V_{\Omega_B} \st ~v= h_{i}, \; \tfrac{\partial v}{\partial n}= s_{i}
    \text{ on } \Gamma_i,\; i=1,\dots,N \right\}\subset  V_{\Omega_B}.
\end{equation}
Homogeneous data $h_i=s_i=0$ provide the corresponding  linear space $V_{\Omega_B}^{0,0}$.
It is  convenient to introduce the 
vector-valued functions $h=(h_i) $ and $s=(s_i)$.

The deformation $u$ of the membrane with fixed height, tilt, and 
location of  particles is  the solution of  the following minimization problem.

\begin{problem}[Boundary values] \ \\%
    \label{p:BVP}%
    Find $u\in V_{\Omega_B}^{h,s}$ minimizing the energy $\cJ_{\Omega_B}$
    on $V_{\Omega_B}^{h,s}$.
\end{problem}

The minimization problem~\ref{p:BVP} is well-known to be equivalent 
to find $u\in V_{\Omega_B}^{h,s}$ such that
\begin{equation} \label{eq:VBVP}
    a_{\Omega_B}(u, v)= 0 \qquad \forall v \in V_{\Omega_B}^{0,0}.
\end{equation}

In order to show existence and uniqueness of a solution to Problem~\ref{p:BVP},
we now derive its reformulation in terms of 
functions defined on the \emph{whole domain} $\Omega$.

To this end, it is convenient to  introduce the particle trace operators $T_i$, 
\begin{align}\label{def:TRACEOP}
    V \ni v \to  T_iv = \Bigl(v|_{\Gamma_i}, \textstyle \frac{\partial v}{\partial n}|_{\Gamma_i}\Bigr)
    \in \cX_i :=H^{\frac32}(\Gamma_i) \times H^{\frac12}(\Gamma_i),    \qquad i=1,\dots,  N.
\end{align}

\begin{lemma}\label{lem:SINGLE_TRACE}
    The trace operator $T_i$ is a continuous, linear, and
    surjective map onto $\cX_i$ for each $i=1,\dots, N$.
\end{lemma} 

\begin{proof} Let $i=1,\dots, N$ be fixed and
    recall that  the boundary $\Gamma_i$ of $B_i$ is $C^3$. 
    As $\overline{B}_i \cap \partial \Omega = \emptyset$, a compactness argument shows
    that $\op{dist}(B_i,\partial \Omega) >0$. 
    Hence, there is a subset $\tilde{\Omega} \subset \Omega$
    with  a $C^3$-boundary $\partial \tilde{\Omega}$
    and $\overline{B}_i \cap \partial \tilde{\Omega} = \emptyset$.
    In particular,  $\tilde{\Omega}_i = \tilde{\Omega}\setminus \overline{B}_i$
    then has a $C^3$-boundary
    and classical trace theorems \cite[Theorems~8.7 and~8.8]{Wloka1987} 
    imply that the trace operator
    \begin{align*}
        \tilde{T}_i : H^2(\tilde{\Omega}_i) \to
            H^{\frac32}(\partial\tilde{\Omega}_i) \times 
             H^{\frac12}(\partial\tilde{\Omega}_i),\qquad
            \tilde{T}_i v = \Bigl(v|_{\partial\tilde{\Omega}_i},
            \textstyle \frac{\partial v}{\partial n}|_{\partial\tilde{\Omega}_i}\Bigr)
    \end{align*} 
    is linear, continuous, and surjective.
    As $\partial \tilde{\Omega}_i = \Gamma_i \cup \partial \tilde{\Omega}$,
    this provides the assertion.
\end{proof} 
Note that denseness of $C^\infty(\overline{\Omega})$ in $H^2(\Omega)$
implies that $T_i$ remains the same if constructed
from the inside of $B_i$ instead of the outside, i.e.,
as trace operator for $H^2(B_i)$ instead of $H^2(\tilde{\Omega}_i)$.

\begin{lemma}\label{lem:PRODUCT_TRACE}
    The trace operator $T:V \to \cX = \prod_{i=1}^N \cX_i$ 
    defined by $(Tv)_i = T_iv$ is linear and continuous.
    If $\overline{B}_i \cap \overline{B}_j = \emptyset$ for $i\neq j$,
    then $T$ is surjective.
\end{lemma}

\begin{proof}
    Lemma~\ref{lem:SINGLE_TRACE} provides linearity and continuity of $T$.
    If $\overline{B}_i \cap \overline{B}_j = \emptyset$ for $i\neq j$,
    then each $(v_i)\in \cX= \prod_{i=1}^N \cX_i$ 
    can be identified with a univalued function 
    $v\in H^{\frac 3 2}(\bigcup \Gamma_i)\times H^{\frac 1 2}(\bigcup \Gamma_i)$
    such that $v|_{\Gamma_i}=v_i$.
    As above, classical trace theorems  \cite[Theorems~8.7 and~8.8]{Wloka1987} 
    imply that there is a $w\in V$ such that 
    $(w|_{\bigcup \Gamma_i},\frac{\partial}{\partial n}w|_{\bigcup \Gamma_i})=v$ and thus   surjectivity of $T$.
\end{proof}

Note that the condition  
$\overline{B}_i \cap \overline{B}_j = \emptyset$ for $i\neq j$
is necessary for surjectivity of $T$.

Now, we introduce the affine subspace 
\begin{equation} \label{eq:AFFINE_CC_SUBSPACE}
    V^{h,s} = 
        \left\{v \in V \st ~v= h_{i}, \; \tfrac{\partial v}{\partial n}= s_{i}
        \text{ on } \Gamma_i,\; i=1,\dots,N \right\} \subset V.
\end{equation}
As a direct consequence of Lemma~\ref{lem:PRODUCT_TRACE},
we  get $V^{h,s} =\{v \in V \st T v = (h,s) \} \neq \emptyset$.
Homogeneous data $h_i=s_i=0$ provide the linear space $V^{0,0}\subset V$.
In this way, the  boundary conditions \eqref{eq:INTBC}
now become {\em constraints on the curves} $\Gamma_i \subset \Omega$.

\begin{problem}[Curve constraints]\ \\%
    \label{p:CC}%
    Find $u\in V^{h,s}$ minimizing the energy $\cJ$ on $V^{h,s}$.
\end{problem}

The equivalent variational formulation of  Problem~\ref{p:CC} is  to find $u\in V^{h,s}$ such that
\begin{equation} \label{eq:VCCP}
    a(u, v)= 0 \qquad \forall v \in V^{0,0}.
\end{equation}

We now  show existence and uniqueness of solutions to Problem~\ref{p:CC}. 

\begin{theorem}\label{prop:EXCC}
    Assume that $\overline{B}_i \cap \overline{B}_j = \emptyset$ for $i\neq j$.
    Then there is a unique solution of Problem~\ref{p:CC}.
\end{theorem} 
\begin{proof}
    By surjectivity, linearity, and continuity of $T$ we know that $V^{h,s}$
    is a non-empty, affine, closed, subspace of $V$. In addition, $a(\wc,\wc)$ is
    bounded and coercive on $V$ by Lemma~\ref{lem:GLOBAL_COERCIVITY}.
    Hence, the Lax-Milgram lemma 
    provides  existence and uniqueness.
\end{proof} 

We now clarify the relation of Problems~\ref{p:CC}  and \ref{p:BVP}.
\begin{lemma}\label{lem:DOMAIN_EXTENSION}
    Assume that $\overline{B}_i \cap \overline{B}_j = \emptyset$ for $i\neq j$.
    Let $v_B \in V_{\Omega_B}$ and $v_i \in H^2(B_i)$
    such that $T_i v_B = T_i v_i$ for all $i=1,\dots,N$.
    Then, the function $v$ defined by $v|_{\Omega_B}=v_B$
    and $v|_{\overline{B}_i} = v_i$, $i=1,\dots,N$, satisfies $v\in V$.
\end{lemma} 
\begin{proof}
    By definition of $V_{\Omega_B}$ there must be an extension
    $\tilde{v} \in V$ with $\tilde{v}|_{\Omega_B} = v_B|_{\Omega_B}$.
    For this extension we have $T_i \tilde{v} -T_i v = T_i \tilde{v} - T_i v_i=0$
    and thus, by \cite[Theorems~8.9]{Wloka1987},
    $\tilde{v}-v \in H_0^2(\bigcup_{i=1}^N B_i) \subset H^2(\R^2)$.
    On the other hand $(\tilde{v}-v)|_{\Omega_B}=0$ guarantees
    that $\tilde{v}-v$ takes the desired boundary condition on $\partial \Omega$
    and hence $\tilde{v}-v \in V$ and $v = \tilde{v} - (\tilde{v}-v)\in V$.
\end{proof} 

\begin{proposition} \label{prop:EQBVPCC}
    Assume that $\overline{B}_i \cap \overline{B}_j = \emptyset$ for $i\neq j$.
    Then the Problems~\ref{p:CC} and \ref{p:BVP} are equivalent in the sense
    that the restriction $u|_{\Omega_B}$ of the solution $u$ to Problem~\ref{p:CC}
    is a solution to Problem~\ref{p:BVP} whereas any solution $u_B$ of
    Problem~\ref{p:BVP} (defined on $\Omega_B \subset \Omega$) can be extended
    to the solution of Problem~\ref{p:CC} (defined on $\Omega$).
\end{proposition}
\begin{proof}
    Let $u$ be the solution to Problem~\ref{p:CC}. 
    Then $u_B = u|_{\Omega_B} \in V_{\Omega_B}^{h,s}$, 
    because $u|_{\Omega_B}\in V_{\Omega_B}$ and $u|_{\Omega_B}$
    satisfies the boundary conditions \eqref{eq:INTBC}.
    Now let $v_B \in V_{\Omega_B}^{0,0}$.
    By Lemma~\ref{lem:DOMAIN_EXTENSION}, the extension $v$ of $v_B$ to $\Omega$ 
    by zero  is contained in $V^{0,0}$. As $u$ satisfies \eqref{eq:VCCP}, we have
    $a_{\Omega_B}(u_B,v_B)=  a(u, v)=0$.
    Hence,  $u_B$ satisfies \eqref{eq:VBVP} 
    and thus solves Problem~\ref{p:BVP}.

    Now assume that $u_B$ is a solution to Problem~\ref{p:BVP}.
    For each fixed $i$, the bilinear form $a_{B_i}(\wc,\wc)$ 
    is continuous and coercive on $H^2_0(B_i)$ by  Lemma~\ref{lem:GLOBAL_COERCIVITY}.
    Hence, the Lax-Milgram lemma provides existence and uniqueness 
    of $u_i \in H^2(B_i)$ with boundary conditions $T_i u_i = (h_i, s_i)$ such that
    \begin{equation} \label{eq:VARB}
    a_{B_i}(u_i,v_i)=0\qquad \forall v_i \in  H_0^2(B_i), \quad i=1,\dots, N.  
    \end{equation}
    Exploiting Lemma~\ref{lem:DOMAIN_EXTENSION}, the extension $u$ 
    of $u_B\in V_B^{h,s}$ to $\Omega$ by  $u_i\in H^2(B_i)$, $i=1,\dots, N$,
    then satisfies $u \in V^{h,s}$.
    Furthermore, for each $v \in V^{0,0}$ we have
    $v_B = v|_{\Omega_B} \in V_{\Omega_B}^{0,0}$ and
    $v_i = v|_{B_i} \in H_0^2(B_i)$.
    As $u_B$ satisfies \eqref{eq:VBVP} and $u_i$ satisfies \eqref{eq:VARB}, this leads to
    \begin{align*}
        a_{\Omega_i}(u_B,v_B) = 0, \qquad a_{B_i}(u_i,v_i) = 0, \quad i=1,\dots, N.
    \end{align*}
    Adding  these equations, we obtain  $a(u,v) = 0$. Hence,
    $u$ is the solution of Problem~\ref{p:CC}.
\end{proof}

Existence and uniqueness of a solution to Problem~\ref{p:BVP}
is now an immediate consequence of Proposition~\ref{prop:EQBVPCC}.
\begin{theorem}\label{prop:EXBVP}
    Assume that $\overline{B}_i \cap \overline{B}_j = \emptyset$ for $i\neq j$.
    Then there is a unique solution $u\in V_{\Omega_B}^{h,s}$
    of Problem~\ref{p:BVP}.
\end{theorem}

\begin{remark}
    For each given location of particles the numerical approximation of Problem~\ref{p:BVP} requires the resolution of $\Omega_B$ 
    by a corresponding finite element mesh.
    In contrast,  Problem~\ref{p:CC} allows for a fixed mesh on
    the full domain $\Omega$ and completely separate, location-dependent 
    approximations of $\Gamma_i$ for the approximate evaluation of curve constraints.
\end{remark}

\subsection{Parametrized curve constraints} \label{sec:PCC}

We consider curve constraints of the form \eqref{eq:INTBC} 
which now  allow for varying heights 
$\gamma=(\gamma_i)\in \R^N$ of the particles $B_i$, $i=1,\dots, N$.  
More precisely, we impose the conditions
\begin{equation}\label{eq:PACU}
    u=h_i + \gamma_i \in H^{\frac{3}{2}}(\Gamma_i),\quad \quad
    \tfrac{\partial}{\partial n}u = s_i \in H^{\frac{1}{2}}(\Gamma_i) \quad \text{ on }\Gamma_i, \quad i=1,\dots, N,
\end{equation}
with additional unknowns $\gamma_i \in \R$.

\begin{problem}[Parametrized curve constraints]\ \\ %
    \label{p:CCH}%
    Find $(u,\gamma) \in V\times \R^N$
    minimizing the energy $\cJ$ subject to curve constraints \eqref{eq:PACU}.
\end{problem}

We  consider existence and uniqueness of a solution
together with an equivalent variational reformulation of Problem~\ref{p:CCH}.

\begin{lemma}\label{lem:VAR_HEIGHT_CHARACTERIZATION}
Assume that $\overline{B}_i \cap \overline{B}_j = \emptyset$ for $i\neq j$.
Then there is an $\eta_0\in V$ such that 
\begin{equation} \label{eq:ETA0}
T_i\eta_0=(h_i,s_i)\qquad \forall i=1,\dots, N,
\end{equation} 
and for each fixed $i=1,\dots,N$ there is an $\eta_i\in V$,  such that 
\begin{equation} \label{eq:ETAI}
T_i \eta_i=(1,0),\quad T_j \eta_i =0,\qquad  \forall i,j=1,\dots, N,\; j\neq i.
\end{equation}
Moreover, $v\in V$ satisfies curve constraints \eqref{eq:PACU}
with some $\gamma_i\in \R$ if and only if
\begin{equation}\label{eq:SPACOND}
    v\in \eta_0 + \op{span}\{\eta_i\st i=1,\dots,N\}+ V^{0,0}.
\end{equation}
\end{lemma}
\begin{proof} The first two  assertions follow from Lemma~\ref{lem:PRODUCT_TRACE}.
    Assume that $v\in V$ satisfies the curve constraints \eqref{eq:PACU} 
    with some $\gamma_i \in \R$.
    Then, by definition of $\eta_i$, we have  
    \[
    T_i(v -  \eta_0 - \sum_{i=1}^N\gamma_i\eta_i)=0 \qquad \forall i=1,\dots, N,
    \]
    and thus  $v\in \eta_0 + \op{span}\{\eta_i\st i=1,\dots,N\}+ V^{0,0}$. 
    The converse is obvious.
\end{proof}

\begin{proposition} \label{prop:CCHEIGHT}
    Assume that $\overline{B}_i \cap \overline{B}_j = \emptyset$ for $i\neq j$.
    Then Problem~\ref{p:CCH} admits a unique solution $(u,\gamma) \in V\times \R^N$.
    The equivalent variational formulation is to find 
    $u_0\in \eta_0+V^{0,0}$ and $\gamma= (\gamma_i)\in \R^N$
    such that $u=u_0 + \sum_{i=1}^N\gamma_i \eta_i$ satisfies
    \begin{equation} \label{eq:VCCH}
        a(u,v)= 0 \quad \forall v\in V^{0,0}\quad \text{and}\quad
        a(u,\eta_i)= 0\quad  \forall i=1,\dots, N,
    \end{equation}
    with  $\eta_i\in V$, $i=0,\dots,N$, 
    defined in Lemma~\ref{lem:VAR_HEIGHT_CHARACTERIZATION}.
\end{proposition}

\begin{proof}
    By Lemma~\ref{lem:VAR_HEIGHT_CHARACTERIZATION}, the Problem~\ref{p:CCH} is
    equivalent to minimizing $\cJ$ on  the  non-empty, affine, and closed subspace
    \begin{align*}
        \eta_0 + \op{span}\{ \eta_i\st i=1,\dots,N \}+ V^{0,0}
        =\Bigl\{ v \in V \Bigst T v \in (h,s) + \prod_{i=1}^N \op{span} \{(1,0)\} \Bigr\}
    \end{align*}
    of $V$.
    Hence, the Lax-Milgram lemma   
    provides  existence, uniqueness, and the variational formulation \eqref{eq:VCCH}.
\end{proof}

We now explicitly decouple the computation of $u$ and $\gamma$
using the orthogonal projections
\begin{align}  \label{eq:OPROP}
    P_{\Gamma_i} v = v- \frac{1}{|\Gamma_i|}\int_{\Gamma_i}v_1\;  \ds (1,0) \in \cX_i, \quad  v=(v_1,v_2)  \in \cX_i,
\end{align}
with respect to the $L^2$-scalar product
$
(\wc,\wc)_{\Gamma_i} = (\wc,\wc)_{L^2(\Gamma_i) \times L^2(\Gamma_i)} 
$.
The kernel of $P_{\Gamma_i}$ is given by
$\op{ker}P_{\Gamma_i} = \op{span}\{(1,0)\}\subset \cX_i$.
The induced projection $P=(P_{\Gamma_i}):\cX \to \cX$ 
on the product space $\cX= \prod_{i=1}^N \cX_i$ now provides a decoupling of $u$ and $\gamma$.


\begin{problem}[Projected curve constraints]\ \\ %
    \label{p:PCCH}%
    Find $u \in V$ minimizing the energy $\cJ$ on $\{ v \in V \st P T v = P (h,s) \}$.
\end{problem}

\begin{proposition} \label{prop:CCH_EQUIVALENCE}
    The pair $(u, \gamma)\in V\times\R^N$ solves Problem~\ref{p:CCH}, if and only if
    $u\in V$ solves Problem~\ref{p:PCCH}
    and $\gamma_i = (T_i u)_1-h_i  \in \R$, $i=1,\dots,N$.
\end{proposition}
\begin{proof}
    The assertion follows from a  general result stated in  Lemma~\ref{lem:A_CONST_EQUIV}.
\end{proof}

\begin{remark} \label{rem:TILTCURVE}
    The existence and uniqueness result in Proposition~\ref{prop:CCHEIGHT}
    can be extended to particles with  additionally varying  tilt angle,
    as expressed by parametrized curve constraints of the form
    \begin{equation}\label{eq:HEIGHTTILT1}
        u=h_i+ \gamma_i + \alpha_i^\top (\wc -X_i), \qquad
        \textstyle  \frac{\partial u}{\partial n}=s_i + \alpha_i^\top \frac{\partial}{\partial n}(\wc -X_i)
        \quad \text{on}\quad \Gamma_i.
    \end{equation}
    Here,  the dot stands for the Euclidean scalar product in $\R^2$
    and the two components of the  varying angles $\alpha_i=(\alpha_{i,1},\alpha_{i,2})$
    are describing tilt in the two coordinate directions.
    Now, let $\eta^1_i \in V$ and $\eta^2_i \in V$  coincide with
    $(\wc - X_i)_1$ and $(\wc - X_i)_2$ in a small neighbourhood of $\Gamma_i$.
    Then  \eqref{eq:HEIGHTTILT1} can be equivalently written as 
    \[        
     u = h_i+ \gamma_i \eta_i + \alpha_{i,1} \eta_i^1 + \alpha_{i,2} \eta_i^2, \quad
        \tfrac{\partial u}{\partial n} = s_i +  
        \tfrac{\partial}{\partial n} \Bigl(\gamma_i \eta_i + \alpha_{i,1} \eta_i^1 + \alpha_{i,2} \eta_i^2\Bigr)
        \quad  \text{on}\quad \Gamma_i
    \]    
    or, equivalently,
    \begin{align*}\label{eq:HEIGHTTILT}
        u  \in  \eta_0 + \op{span} \{\eta_i, \eta^1_i , \eta^2_i\st i=1,\dots, N \} + V^{0,0}.
    \end{align*}
    Utilizing suitable orthogonal projections, the computation of $u\in V$ and of the
    coefficients $\gamma_i,\alpha_{i,j}\in \R$ can be decoupled
    in analogy to  Problem~\ref{p:PCCH} and Proposition~\ref{prop:CCH_EQUIVALENCE}.
    We refer to \cite{C04-wolfDiss15} for details.
\end{remark}

\subsection{Varying the location of  particles}  \label{subsec:VCC}
\subsubsection{Hard and soft wall constraints} \label{sec:HARDSOFT}
For varying locations  $X=(X_i)\in \Omega^N$ of particles $B_i=B_i(X)= X_i+B_i^0$
we want to enforce that particles do not overlap, 
touch  the boundary $\partial \Omega$ or even escape 
from the 
considered domain $\Omega$. 
To this end, we constrain the locations $X=(X_i)$ to the subset
\begin{equation} \label{eq:FFSP}
    \omega=\{X \in \Omega^N \st B_i(X) \cap B_j(X) = B_i(X) \cap \partial \Omega= \emptyset 
        \;\forall i \neq j \}
\end{equation}
of all positions with non-overlapping particles contained in the domain.
Recall that the  particles $B_i(X)\subset \R^2$ are open sets
with $C^3$-boundaries and $X_i \in B_i (X)$.

\begin{lemma} \label{lem:COMPO}
    The subset $\omega \subset  \Omega^N$  is compact in $\R^{N \times 2}$.
\end{lemma}
\begin{proof}
    Denoting $B_{N+1}(X) = \R^2 \setminus \overline{\Omega}$, we have 
    $B_i(X) \cap \partial \Omega \neq \emptyset$,
    if and only if $B_i(X) \cap B_{N+1}(X) \neq \emptyset$,
    since $B_i(X)$ is open.
    Hence, $\omega$ can be written as
    \begin{align*}
        \omega=\{X \in \R^{N \times 2} \st B_i(X) \cap B_j(X) = \emptyset 
        \;\forall i,j\in\{1,\dots,N+1\}, i \neq j \}.
    \end{align*} 

    Now let $X \in \R^{N \times 2}\setminus \omega$.
    Then there are $i \neq j$ such that $B_i(X) \cap B_j(X) \neq \emptyset$.
    Hence there is $x \in \R^2$ and $\varepsilon>0$ such that
    $\cB_\varepsilon(x) \subset B_i(X) \cap B_j(X)$ for the open
    ball $\cB_\varepsilon(x)$. 
    Then we also have $\cB_{\varepsilon/2}(Y) \subset B_i(Y)\cap B_j(Y)$
    if $|X-Y| < \tfrac{\varepsilon}{2}$ and thus $Y \in \R^{N \times 2} \setminus \omega$.
    Hence, $\R^{N \times 2}\setminus \omega$ is open. Thus $\omega$ is closed.
    As $\Omega$ is bounded, $\omega \subset \Omega^N$ is also bounded and therefore compact.
\end{proof}
Later on, we will also require that particles do not touch neither each other nor the boundary.
It can be shown by elementary arguments that the set
of such particle positions is given by the interior of $\omega$ which can be characterized by
\begin{align}\label{eq:INTERIOR_OF_OMEGA}
    \begin{split}
    \op{int} \omega
        &= \{X \in \Omega^N \st \overline{B_i(X)} \cap \overline{B_j(X)} = \overline{B_i(X)} \cap \partial \Omega= \emptyset
            \;\forall i \neq j \} \\
            &= \{X \in \Omega^N \st \op{dist}(B_i(X), B_j(X))>0 \textnormal{ and } \op{dist}(B_i(X), \partial \Omega)>0
            \;\forall i \neq j \}.
    \end{split} 
\end{align} 

The constraints $X=(X_i)\in \omega$ are enforced by  additional terms in the energy.
We first consider so-called {\em soft-wall constraints}
$\cV_{\soft}=\cV_1+\cV_2$ consisting of a Lennard-Jones potential 
\begin{equation} \label{eq:LJP}
    \cV_1(X)=\sum_{\underset{i\neq j}{i,j=1}}^N \cV_{ij},\quad 
    \cV_{ij}=4 \epsilon_{ij}\left[ \left(\frac{\sigma_{ij}}{\text{dist}(B_i,B_j)}\right)^{12}-\left(\frac{\sigma_{ij}}{\text{dist}(B_i,B_j)}\right)^{6} \right],
\end{equation}
for $X\in \omega$ such that $\text{dist}(B_i,B_j) > 0$, $i\neq j$ and $\cV_1(X)=\infty$ otherwise.
This term  accounts for  the repulsion and attraction of particles. We also define 
\begin{equation} \label{eq:BCP}
    \cV_2(X)=\sum_{i=1}^N \left(\frac{\sigma_{i}}{\text{dist}(B_i,\partial \Omega)}\right)^6.
\end{equation}
for $X\in \omega$ such that $\text{dist}(B_i,\partial \Omega)>0$, $i=1,\dots,N$, 
and $\cV_2(X)=\infty$ otherwise.  This term is accounting for escaping particles. 
For circular particles we have $\text{dist}(B_i,B_j)=|X_i-X_j|-(r_i+r_j)$.
Note that the soft-wall potential $\cV_{\soft}=\cV_1 + \cV_2$ is continuously differentiable on $\op{int}\omega$.

The energy $\cV_{\soft}$  associated with soft-wall constraints can be regarded
as a penalization of {\em hard-wall constraints}
$\cV_{\hard}=\bar{\cV}_1+ \bar{\cV}_2$ with
\begin{equation} \label{eq:HWC}
    \bar{\cV}_1(X)= \left\{
        \begin{array}{cc}
            0, & B_i \cap  B_j = \emptyset  \;\forall  i,j \\[5mm]
            \infty, &\text{otherwise}
        \end{array}
        \right .
        \quad 
        \bar{\cV}_2 (X)= \left\{
            \begin{array}{cc}
                0, & B_ i \subset \bar{\Omega}\; \forall   i  \\[5mm]
                \infty , & \text{otherwise}
            \end{array}
            \right .
\end{equation}
Note that $\omega$ is the domain of $\cV_{\hard}$.

%
%

%
%
 
\subsubsection{A Gradient flow approach} \label{subsec:GRADFLOW}
With given reference particles  $B_i^0$ located at  $0$
 and data $(h_i,s_i)$ prescribed on $\Gamma_i(X)=X_i+\Gamma_i^0$
according to \eqref{eq:BNDDATA} (see Section~\ref{sec:BVP}),
we now allow for varying locations $X=(X_i)$ of particles $B_i(X)=X_i+B_i^0$, 
and  fix height and tilt, for simplicity. We impose soft wall constraints $\cV_{\soft}$
according to \eqref{eq:LJP}, \eqref{eq:BCP} .

\begin{problem}[Varying locations of particles with soft wall constraints]\ \\%
    \label{p:VLOCBVP}%
    Find $(u,X)\in V\times \op{int} \omega$ minimizing the energy $\cJ + \cV_{\soft}$  
    subject to curve constraints \eqref{eq:INTBC}, \eqref{eq:BNDDATA}  on
    $\Gamma_i(X)=X_i+ \Gamma_i^0, \quad i=1,\dots,N$.
\end{problem}

We describe a gradient flow approach to the iterative solution of Problem~\ref{p:VLOCBVP}.
For $t\geq 0$ let $X(t)$ be a  trajectory of locations of particles $B_i(X(t))$ with
corresponding boundaries $\Gamma_i(X(t))$, $i=1,\dots,N$. 
Assuming $\overline{B}_i(X(t)) \cap \overline{B}_j(X(t))=\emptyset$ for $i\neq j$ and 
$\overline{B}_i(X(t))\cap \partial \Omega=\emptyset$,
let $u(X)\in V^{h,s}$  be the   solution of Problem~\ref{p:CC} with  fixed $X=X(t)$.  
Denoting  $\nabla_X=(\nabla_{X_1},\dots, \nabla_{X_N})$ and $\nabla_{X_i}=\left(\frac{\partial}{\partial X_{i,1}},\frac{\partial}{\partial X_{i,2}}\right)$, 
we consider the gradient flow
\begin{equation}\label{eq:GFCC}
X' = -\nabla_X\cJ(u(X)) -  \nabla_X \cV_{\soft}(X),\quad t>0,\qquad X(0)=X_0,
\end{equation}
with suitable initial locations $X_0\in \omega$. Then, by construction, we have monotonically decreasing
energy $\cJ(u(X(t))) + \cV_{\soft}(X(t))$  along a solution $X(t)$ of \eqref{eq:GFCC} and we may expect 
that $(u(X(t)),X(t))$ tends to a solution  of Problem~\ref{p:VLOCBVP} for $t\to \infty$.

While the gradient $ \nabla_X\cV_{\soft}(X)$ can be obtained by straightforward calculus, the derivation
of  $\nabla_X\cJ(u(X))$ requires some care. 
In particular, in light of a generic lack of smoothness of the minimizers 
$u(X(t))\in V^{h,s}$ across $\Gamma_i$,
it is convenient to consider their restriction  to $\Omega_B$.   
Utilizing basic techniques  from shape calculus,
we then obtain the following representation of $\nabla_{X_i}\cJ(u(X))$
(see \cite{C04-wolfDiss15} for details).

\begin{proposition}
Assume that
$u: \Omega_B\times [0,\infty)\to \R$ is sufficiently smooth. Then
\begin{equation*}
    \nabla_{X_i}\cJ(u(X)) = \int_{\Gamma_j} \left( \kappa {\textstyle \frac{\partial}{\partial n} }\Delta u - \sigma s_j \right)\nabla u 
    - \kappa \Delta u  \nabla (  {\textstyle\frac{\partial}{\partial n} }u )
    + \left( {\textstyle \frac{1}{2}} \kappa(\Delta u)^2  + {\textstyle \frac{1}{2}} \sigma|\nabla u|^{2} \right)  n\; \ds
\end{equation*}
holds for $i=1,\dots, N$ and the abbreviation $u=u(X)$.
\end{proposition}

\subsection{Soft curve constraints} \label{sec:SOFT}
\subsubsection{Penalization of curve constraints}
The numerical approximation of the curve constrained  Problem~\ref{p:CC}
leads to saddle point problems involving suitable approximating spaces
for  primal and dual variables.   
As an alternative, we present an adaptation of the boundary penalty approach, \cite{Bab73, BarEll86} 
to essential boundary conditions, see also Nitsches approach~\cite{Nit71}. More precisely,
we  penalize  deviation from curve constraints \eqref{eq:INTBC} by 
the additional energy term
\begin{align} \label{eq:SCPAR}
    \frac{1}{2\varepsilon}\sum_{i=1}^N \|T_iu - (h_i,s_i)\|_{\Gamma_i}^2
\end{align}
with given  penalty parameter $\varepsilon>0$,  
the trace operators $T_i$ defined in \eqref{def:TRACEOP}, 
and given data  $h_i \in H^{\frac{3}{2}}(\Gamma_i)$, $s_i \in H^{\frac{1}{2}}(\Gamma_i)$, $i=1,\dots,N$.
Here, $\|\wc\|_{\Gamma_i} = (\wc,\wc)_{\Gamma_i}^{1/2}$ is the $L^2$-norm
on the trace space $\cX_i$. Note that the resulting mathematical problem may be considered as a model in its own right by 
interpreting the resulting boundary conditions as physical contact conditions and $\epsilon$ as a modelling parameter.
The counterpart of Problem~\ref{p:CC} with {\em soft constraints}
then reads as follows.

\begin{problem}[Soft curve constraints] \ \\ \label{p:SCC}%
    Find $u_\varepsilon \in V$ minimizing the energy
    $$\cJ(u_\varepsilon) + \frac{1}{2\varepsilon}\sum_{i=1}^N \|T_i u_\varepsilon - (h_i,s_i)\|_{\Gamma_i}^2.$$
\end{problem}

An equivalent variational formulation of Problem~\ref{p:SCC} is to find $u_\varepsilon \in V$ such that
\begin{equation}\label{eq:SCCVAR}
    a(u_\varepsilon,v) +  \frac{1}{\varepsilon}\sum_{i=1}^N (T_i u_\varepsilon, T_i v)_{\Gamma_i}
    =\frac{1}{\varepsilon}\sum_{i=1}^N ((h_i,s_i), T_i v)_{\Gamma_i} \qquad \forall v\in V.
\end{equation}

\begin{proposition}\label{Pro:EXCCH}
    Problem~\ref{p:SCC} admits a unique solution $u_\varepsilon \in V$.
\end{proposition}
\begin{proof}
   The bilinear form  in \eqref{eq:SCCVAR}
    is $V$-elliptic, because $a(\wc,\wc)$ is $V$-elliptic and,
    by Lemma~\ref{lem:SINGLE_TRACE}, the trace operators $T_i:V\to L^2(\Gamma_i)^2$
    are continuous. The right hand side in \eqref{eq:SCCVAR} is bounded by continuity of $T_i$
    and the assumptions on the data.    
    Thus the Lax-Milgram lemma provides existence and uniqueness.
\end{proof}

Problem~\ref{p:SCC} with soft constraints can be regarded as an approximation
of  Problem~\ref{p:CC} with curve constraints.

\begin{proposition}\label{prop:SCC}
    Assume that $\overline{B}_i \cap \overline{B}_j = \emptyset$ for $i\neq j$.
    Let $u$ denote the solution of Problem~\ref{p:CC} and  $u_\varepsilon$  
    denote the solution of Problem~\ref{p:SCC} for  fixed $\varepsilon >0 $.
    Then we have
    \[
     u_\varepsilon \to u \quad \text{in }  V \qquad \text{as }  \quad \varepsilon \to 0.
    \]
\end{proposition}
\begin{proof}
    By Theorem~\ref{prop:EXCC} Problem~\ref{p:CC} admits a unique solution $u\in V^{h,s}=
    \{v\in V\st Tv=(h,s)\}$, 
    and the trace operator $T$ introduced in Lemma~\ref{lem:PRODUCT_TRACE}
    is a bounded, linear map into  $\cX$.
    Hence, the assertion follows from Proposition~\ref{prop:A_PENALTY_CONV} in the appendix.
\end{proof}

\subsubsection{Parametrized soft curve constraints}
We consider a soft version of  the  para\-me\-trized curve constraints \eqref{eq:PACU}
involving  variable heights $\gamma_i\in \R$ of particles $B_i$.
The corresponding additional energy term reads
\begin{align} \label{eq:Pen1}
    \frac{1}{2\varepsilon}\sum_{i=1}^N \|T_i u - (h_i+\gamma_{i},s_i)\|_{\Gamma_i}^2
\end{align}
with given penalty parameter $\varepsilon >0$,  given functions $h_i,s_i$ as in \eqref{eq:PACU}, 
and unknown heights  $\gamma_{i}\in \R$, $i=1,\dots,N$.
The corresponding  minimization problem reads as follows.

\begin{problem}[Soft parametrized curve constraints]\ \\ \label{p:SCCH}%
Find $(u_\varepsilon, \gamma_\varepsilon)\in V \times \R^N$  minimizing the energy
\begin{align*}
   \cJ(u_\varepsilon) + \frac{1}{2\varepsilon}\sum_{i=1}^N \|T_i u_\varepsilon - (h_i+\gamma_{\varepsilon,i},s_i)\|_{\Gamma_i}^2.
\end{align*} 
\end{problem}

Again, we  use the projections $P_{\Gamma_i}$  defined in \eqref{eq:OPROP}
to decouple the computation of $u_{\varepsilon}$ and $\gamma_{\varepsilon}$.
In particular,  orthogonality of $P_{\Gamma_i}$ and $(\gamma_{\varepsilon,i},0)\in \op{ker}P_{\Gamma_i}$ 
provides the  reformulation
\[
    \| T_iu_{\varepsilon} -  (h_i+ \gamma_{\varepsilon,i},s_i)\|^2_{\Gamma_i}= 
    \| P_{\Gamma_i}(T_i u_{\varepsilon} -  (h_i,s_i))\|^2_{\Gamma_i} +
    \| (I-P_{\Gamma_i})(T_i u_{\varepsilon} - (h_i + \gamma_{\varepsilon,i}, s_i)) \|^2_{\Gamma_i}
\]
of the penalty term \eqref{eq:Pen1}.
Here, the second summand on the right hand side is vanishing for 
\[
    \gamma_{\varepsilon, i} =  ((I-P_{\Gamma_i})(T_i u_{\varepsilon} - (h_i,s_i)))_1
    =\frac{1}{|\Gamma_i|}\int_{\Gamma_i} (u_\varepsilon-h_i)\, ds.
\]

\begin{proposition} \label{prop:SCCH_EQUIVALENCE}
    The pair $(u_\varepsilon, \gamma_\varepsilon)\in V\times\R^N$ solves Problem~\ref{p:SCCH},  if and only if
    $u_\varepsilon\in V$ is minimizing the energy
    \begin{align} \label{eq:SEPEN}
        \cJ(u_\varepsilon) + \frac{1}{2\varepsilon}\sum_{i=1}^N \|P_{\Gamma_i}(T_i u_\varepsilon - (h_i,s_i))\|_{\Gamma_i}^2
    \end{align}
    and $\gamma_{\varepsilon,i} = \frac{1}{|\Gamma_i|}\int_{\Gamma_i} (u_\varepsilon-h_i)\, ds$, $i=1,\dots,N$.
\end{proposition}
\begin{proof}
     The assertion follows from a  general result stated in  Proposition~\ref{lem:A_PENALIZED_MIN_EQIV}.
\end{proof}

Proposition~\ref{prop:SCCH_EQUIVALENCE} provides an explicit representation
of the heights $\gamma_{\varepsilon}=(\gamma_{\varepsilon,i})$ 
in terms of the deformation $u_\varepsilon$ 
that in turn  can be computed from a fully decoupled minimization problem.
The  variational form of  the minimization problem  is  to find $u_\varepsilon \in V$ such that
\begin{align*}
    a(u_\varepsilon,v)
    + \frac{1}{\varepsilon} \sum_{i=1}^N(P_{\Gamma_i}(T_i u_\varepsilon), P_{\Gamma_i}(T_i v))_{\Gamma_i}
    =   \frac{1}{\varepsilon} \sum_{i=1}^N (P_{\Gamma_i}(h_i,s_i), P_{\Gamma_i}(T_i v))_{\Gamma_i} \quad \forall v \in V.
\end{align*}

\begin{proposition}
    Problem~\ref{p:SCCH} admits a unique solution
    $(u_\varepsilon,\gamma_\varepsilon)\in V \times \R^N$.
\end{proposition}
\begin{proof}
    Since $a(\wc,\wc)$ is $V$-elliptic and  $T_i:V\to L^2(\Gamma_i)^2$  
    and $P_{\Gamma_i}: L^2(\Gamma_i)^2\to L^2(\Gamma_i)^2$  are continuous,
    the Lax-Milgram Lemma provides the assertion.
\end{proof}

\begin{proposition}\label{prop:SCCH_CONV}
    Assume that $\overline{B}_i \cap \overline{B}_j = \emptyset$ for $i\neq j$.
    Let $(u,\gamma)$ denote the solution of Problem~\ref{p:CCH} and
    $(u_\varepsilon, \gamma_\varepsilon)$ denote the solution of
    Problem~\ref{p:SCCH} for fixed $\varepsilon >0 $.
    Then we have
    \begin{align*}
        (u_\varepsilon,\gamma_\varepsilon) \to (u,\gamma) \quad \text{in }  V\times \R^N
        \qquad \text{as }  \quad \varepsilon \to 0.
    \end{align*} 
\end{proposition}
\begin{proof}
    The same arguments as in the proof of Proposition~\ref{prop:SCCH_EQUIVALENCE}
    show that the assumptions of Proposition~\ref{prop:A_PENALTY_CONV} in the appendix
    are satisfied and the latter provides the assertion.
\end{proof}

\begin{remark} \label{rem:SOFTTILT}
    The same arguments can be applied to particles with variable height  and
    tilt, i.e., to parametrized data of the form~\eqref{eq:HEIGHTTILT1} with
    unknown height $\gamma_i$ and tilt $\alpha_i$, $i=1,\dots, N$,
    as introduced in Remark~\ref{rem:TILTCURVE}. In this case the
    projections $P_{\Gamma_i}$ have to be replaced by
    $L^2(\Gamma_i)^2$-orthogonal projections $\tilde{P}_{\Gamma_i}$
    with the kernel
    \begin{align*}
        \textstyle \op{ker} \tilde{P}_{\Gamma_i} = \op{span}\{(0,1), ((\wc -X_i)_1, \frac{\partial}{\partial n}(\wc -X_i)_1), ((\wc -X_i)_2, \frac{\partial}{\partial n}(\wc -X_i)_2)\}
    \end{align*}
    where the additional two functions describe tilt  as  in Remark~\ref{rem:TILTCURVE}.
\end{remark}

\subsubsection{Varying the location of  particles} \label{subsubsec:SCCL}
We consider varying locations $X = (X_i)$ of particles $B_i = B_i(X) = X_i + B_i^0$ 
with  reference particles $B_i^0$ and data $(h_i,s_i)$ prescribed on $\Gamma_i(X)=X_i+\Gamma_i^0$
according to \eqref{eq:BNDDATA}.
In light of Proposition~\ref{prop:SCCH_EQUIVALENCE} and Remark~\ref{rem:SOFTTILT},  
we assume that height and tilt of particles are fixed, for simplicity.

We  first consider  hard-wall constraints  $\cV_{\hard}$ as defined in \eqref{eq:HWC}. 
This means that we constrain  the locations $X=(X_i)$ to the domain 
$\omega$ of $\cV_{\hard}$    as given in  \eqref{eq:FFSP}.

\begin{problem}[Soft curve constraints with varying locations and hard-wall constraints] \label{p:SCCLH}
    Find $(u_{\varepsilon}, X_{\varepsilon})\in V \times \omega$
    minimizing the energy
    \[
        \cJ(u_\varepsilon)+ \frac{1}{2\varepsilon}\sum_{i=1}^N \|T (X_\varepsilon)_i u_\varepsilon - (h_i,s_i)\|_{\Gamma_i(X_\varepsilon)}^2.
    \]
\end{problem}

Here, the trace operators $T(X)_i $ are defined  according to~\eqref{def:TRACEOP} 
but with $\Gamma_i$ replaced by $\Gamma_i(X)$.  We set $T(X)=(T(X)_i)$ in accordance with 
Lemma~\ref{lem:PRODUCT_TRACE}.

Now we augment the energy by the additional term $\cV_{\soft}(X)$ associated with soft wall constraints
that penalize overlapping and escaping particles (cf.\ Subsection~\ref{sec:HARDSOFT}).

\begin{problem}[Soft  curve constraints with varying locations and soft wall constraints] \label{p:SCCLS} 
    Find $(u_{\varepsilon}, X_{\varepsilon})\in V \times \op{int}\omega$ minimizing the energy
    \[
    \cE(u_{\varepsilon},X)=\cJ(u_{\varepsilon})+ 
    \frac{1}{2\varepsilon}\sum_{i=1}^N \|T(X_\varepsilon)_i u_\varepsilon - (h_i,s_i)\|_{\Gamma_i(X)}^2 + \cV_{\soft}(X).
    \]
\end{problem}
We briefly describe a straightforward 
gradient flow approach to the iterative solution of Problem~\ref{p:SCCLS}.
As in Subsection~\ref{subsec:GRADFLOW}, we consider a  trajectory  
$X(t)\in \op{int}\omega$ of locations of particles with
corresponding boundaries $\Gamma_i(X(t))$, $i=1,\dots,N$.
Let $u(X)\in V$ denote the unique minimizer of $\cE(\cdot,X)$ on $V$
for given $X=X(t)$ (cf.\ Proposition~\ref{Pro:EXCCH}).
Again denoting 
$\nabla_X=(\nabla_{X_1},\dots, \nabla_{X_N})$ and $\nabla_{X_i}=\left(\frac{\partial}{\partial X_{i,1}},\frac{\partial}{\partial X_{i,2}}\right)$, 
we might consider the gradient flow 
\begin{equation}
    X' = -\nabla_X \cE(u(X),X), \quad t>0 ,
\end{equation}
with given initial iterate  $X(0)=X_0\in \op{int}\omega$.
For $i=1,\dots,N$ a formal calculation yields the representation
 \[
        \displaystyle 
        \nabla_{X_i}\cE(u(X)) = 
        \displaystyle 
        \frac{1}{\varepsilon}\int_{\Gamma_i(X)}  (u-h_i)\nabla u  
        + \left( {\textstyle  \frac{\partial}{\partial n}}u-s_i \right)
        \nabla ( {\textstyle \frac{\partial}{\partial n}} u) \; ds  +
        \nabla_{X_i}\cV(X)
    \]
    with $u=u(X)$ which, however, might not be applicable for lack of smoothness. 
    Existence and approximation of global minimizers for varying locations will be considered in a separate paper.

\subsection{Discussion}
We have analysed various descriptions of proteins in lipid membranes
in terms of rigid particles of finite size, 
interacting with a linearised  Canham-Helfrich membrane model
by suitable conditions on the membrane displacement and its normal derivative  (angle condition) at the particle boundaries.
The history of such kind of hybrid models
can be traced back to the early nineties, see, e.g., \cite{GouBruPin93,HelfrichJakobsson90,Huang86,KimNeuOst98,ParLub96,WeiKozHel98}.
For example, Problems~\ref{p:BVP} and~\ref{p:CCH}  
set out in Section~\ref{sec:BVP} and \ref{sec:PCC} are equivalent 
to what is referred to as the `strong coupling regime' in \cite{GouBruPin93} 
or the `microscopic model' in \cite{ParLub96}.
See, e.g., \cite{NieGouAnd98} for a 
careful discussion of the boundary conditions~\eqref{eq:INTBC}.
Investigations particularly focus on the 
effect of the mechanical properties of the membrane
on the interaction between different particles in interdependency with  corresponding membrane deformations.
Such kind of membrane-mediated interaction was studied for particles with parametrized 
height and tilt angle, circular or non-circular cross-section, respecting or breaking 
reflection symmetry~\cite{GouBruPin93,KimNeuOst00,ParLub96}.
Varying particle positions as treated in Section~\ref{subsec:VCC}
have also been considered in this context, see, e.g., \cite{KimNeuOst98}.
The notion of 'soft inclusions' as introduced from a physical point of view, 
also referred to as `perturbative regime'  \cite{GouBruPin93} 
or `phenomenological model' \cite{ParLub96},
does not completely agree with the notion of `soft constraints' as considered
in Section~\ref{sec:SOFT}.
Soft inclusions were originally introduced to study the influence of regions with excess concentration of lipids on fluctuating membranes,
where it is appropriate to assume the bending rigidity of these regions to be close to the value of the surrounding membrane.
Note, that regions of excess lipid concentration usually 
respect reflection symmetry (across the membrane) and thus, 
do not impose any constraints on the curvature. However, applying this approach to model `soft', i.e., non-rigid proteins 
that break reflection symmetry, due to spontaneous curvature terms $c_i$,
leads to an energy $\cJ_{\Omega_B}(u) + \sum_{i=1}^{N}\int_{B_i}\frac{1}{2}\kappa_i(\Delta u - c_i)^2\,dx$, 
which takes the form of soft point curvature constraints as considered
in Section~\ref{sec:pcc}, when formally  passing to the limit of point-like particles.

\section{Averaged curve constraints} \label{sec:AVCU}
\subsection{Fixed heights and locations}
In order to derive   approximations of boundary conditions
or equivalent curve constraints  \eqref{eq:INTBC}  on $\Gamma_i$, 
we introduce the averaging functionals $f_i$, $g_i:V \rightarrow \R$, 
\begin{equation} \label{eq:DISCO}
    f_i(v)= \fint_{\Gamma_i} v\;\ds,   \qquad 
    g_i(v)= \fint_{\Gamma_i} \frac{\partial}{\partial n}v\;\ds, \qquad i=1,\dots, N,
\end{equation}
where $\fint_D:=\frac{1}{|D|}\int_D.$
Possible variations of traces along $\Gamma_i$ are ignored in this way.
Recall from \eqref{def:TRACEOP} that $(T_i v)_1 = v|_{\Gamma_i}$ and 
$(T_i v)_2 = \frac{\partial}{\partial n}v|_{\Gamma_i}$.
The functionals $f_i$, $g_i$, are  linear and bounded on $V$,
because each $T_i$ is a linear and bounded map into
$\cX_i \subset L^1(\Gamma_i)\times L^1(\Gamma_i)$ (cf.\ Lemma~\ref{lem:SINGLE_TRACE}).

\begin{remark} \label{rem:CURVATURE}
    Utilizing  Green's formula,  the functionals $g_i$, $i=1,\dots,N$, can be expressed as 
    averaged mean curvature according to
    \begin{equation} \label{eq:CURVATCO}
        g_i(v)=-\frac{|B_i|}{|\Gamma_i|}\fint_{B_i}\Delta v\; dx, \qquad v\in V,
    \end{equation}
    where the sign results from the fact that $n$ is an inward normal to $B_i$.
    Conditions on $g_i$ are therefore called {\em mean curvature constraints} in the sequel.
\end{remark}
We introduce the averaged data 
\begin{equation} \label{eq:AVDAT}
    \bar{h}_i= \fint_{\Gamma_i}h_i\;\ds, \quad \bar{s_i}=\fint_{\Gamma_i}s_i\;\ds, \quad i=1,\dots, N,
\end{equation}
with  $h_i$, $s_i$ given according to \eqref{eq:BNDDATA} and the notation
\begin{equation} \label{eq:VECFG}
    f_X=(f_i)\in (V')^N, \;g_X=(g_i) \in (V')^N, \qquad \bar{h}=(\bar{h}_i) \in \R^N,\; \bar{s}=(\bar{h}_i)  \in \R^N.
\end{equation}
As an approximation of the boundary value Problem~\ref{p:BVP} or
its  equivalent fixed-domain formulation Problem~\ref{p:CC}, we now consider the 
following minimization problem with {\em average constraints}.

\begin{problem}[Averaged curve constraints]\ \\ \label{p:AC} %
    Find   $u\in V$ minimizing the energy $\cJ$ subject to the  average constraints 
    \begin{equation} \label{eq:AC}
        f_X(u) =\bar{h}, \qquad g_X(u)=\bar{s}.
    \end{equation}
\end{problem}
Note that, using the componentwise integral operator $\int:\cX \to \R^{ N \times 2}$ defined by
\[
\textstyle \fint v  
= \left(\big(\fint_{\Gamma_i}v_{i,1}\;\ds,\fint_{\Gamma_i}v_{i,2}\;\ds\big)\right), 
\qquad v=\big( (v_{i,1}, v_{i,2}) \big)\in \cX,
\]
the constraints \eqref{eq:AC} can  be rewritten as
\begin{align} \label{eq:INTEGRATED_CONSTRAINTS}
   \textstyle \fint T u
       = \fint (h,s)
       = (\bar{h},\bar{s})
,\qquad T=(T_i),
\end{align}
i.e., as an average of the constraints in Problem~\ref{p:CC}
(cf. Theorem~\ref{prop:EXCC}).
Obviously, averaged  constraints \eqref{eq:AC} or \eqref{eq:INTEGRATED_CONSTRAINTS} 
are weaker than  curve constraints  \eqref{eq:INTBC}.

\begin{proposition}
    Assume that $\overline{B}_i \cap \overline{B}_j = \emptyset$ for $i\neq j$.
    Then there is a unique solution $u\in V$ of Problem~\ref{p:AC}.
\end{proposition}
\begin{proof}
    The trace operator $T=(T_i)$ is linear, continuous,
    and surjective by Lemma~\ref{lem:PRODUCT_TRACE}.
    Continuity of averaging $\int$ implies continuity of  $\int T: V \to \R^{N \times 2}$.
    Therefore, the Lax-Milgram lemma provides  existence and uniqueness in the
    non-empty, affine, closed, subspace
    \[
        \{v \in V \st \textstyle\fint T v = \fint (h,s) \}\subset V.
    \]
    See also Proposition~\ref{prop:A_CONSTR_MIN} in the appendix.
\end{proof}

\subsection{Averaged  mean curvature constraints}
We now investigate averages of parametrized curve constraints \eqref{eq:PACU} 
of the form 
\begin{equation}\label{eq:AVPAR}
    f_i(u)=\fint_{\Gamma_i} h_i  \ds+ \gamma_i\;, \qquad  g_i(u)= \int_{\Gamma_i} s_i \;\ds, \qquad i=1,\dots,N,
\end{equation}
with  $h_i$, $s_i$ given according to \eqref{eq:BNDDATA} and
freely varying heights $\gamma=(\gamma_i)\in \R^N$.
In analogy to \eqref{eq:INTEGRATED_CONSTRAINTS} the 
parametrized constraints \eqref{eq:AVPAR} can be rewritten as
\begin{equation} \label{eq:APACO}
    \textstyle \fint Tu =  (\gamma,0) + \fint(h,s)
\end{equation}
i.e.  as averages of parametrized curve constraints \eqref{eq:PACU} occurring in  Problem~\ref{p:CCH}.
Utilizing the orthogonal projection $P=(P_{\Gamma_i}): \cX \to \cX$ with $P_{\Gamma_i}$ 
introduced in \eqref{eq:OPROP}, the constraints \eqref{eq:APACO} can be decoupled according to
\begin{equation} \label{eq:APACOP}
    \textstyle \fint PTu = \fint P (h,s),\qquad  \gamma=\fint (Tu)_1 -\fint h.
\end{equation}
Now observe that by definition of $P=(P_{\Gamma_i})$ the constraints
\[
    \textstyle (\fint PTv)_{i,1}= \fint_{\Gamma_i} (P_{\Gamma_i}T_iv)_1=0= \fint_{\Gamma_i} (P_{\Gamma_i} (h_i,s_i))_1,
    \qquad i=1,\dots,N,
\]
hold for all $v\in V$.  The remaining constraints in \eqref{eq:APACOP} 
and thus the parametrized constraints \eqref{eq:AVPAR} take the form
\[
    \textstyle (\fint PTu)_{i,2}=(\fint Tu)_{i,2}= \int s.
\]
In light of $g_X(u)= (\fint Tu)_{i,2}$, Remark~\ref{rem:CURVATURE}, 
and $\bar{s}=\fint s$,  minimization of $\cJ$ subject to averaged
parametrized constraints \eqref{eq:AVPAR} finally amounts to  the following 
averaged version  of Problem~\ref{p:PCCH}.

\begin{problem}[Averaged mean curvature constraints]\ \\  \label{p:AMC}%
    Find $u \in V$ minimizing the energy $\cJ$ subject to the constraints
    \begin{equation} \label{eq:AVCUCO}
        g_X(u) = \bar{s}.
    \end{equation}
\end{problem}
According to \eqref{eq:APACOP}, the optimal heights $\gamma=(\gamma_i)\in \R^N$ are obtained from
 \begin{equation} \label{eq:HEICON}
    \gamma_i= f_i(u)-\fint_{\Gamma_i}h_i\; \ds, \qquad i=1,\dots,N.
\end{equation}
once the solution $u\in V$ of Problem~\ref{p:AMC} is available.
Note that, by the representation \eqref{eq:CURVATCO}, one can
construct a function $w \in V$ with $g_X(w)=\bar{s}$, even
without the assumption $\overline{B}_i \cap \overline{B}_j \neq \emptyset$.
Hence, the following existence result  is a consequence of  the Lax-Milgram lemma.

\begin{proposition} \label{prop:AMC}
    There exists a unique solution $u\in V$ to Problem~\ref{p:AMC}.
\end{proposition}

\begin{remark}
    As a consequence of 
    \[
        \int_{\Gamma_i} \alpha_i^\top {\textstyle  \frac{\partial}{\partial n}} (x-X_i)\; \ds =0 \qquad \forall \alpha_i, X_i \in \R^2
    \]
    the tilt of particles is no longer represented by parametrized constraints of the form \eqref{eq:HEIGHTTILT1}
    after averaging.
\end{remark}

\subsection{Varying the location of  particles}
We now consider particles with varying location and hard- and soft-wall constraints
as defined in Section~\ref{sec:HARDSOFT}.
In case of hard-wall constraints, we allow  the locations $X=(X_i)$  
to vary only in the domain  $\omega$ of $\cV_{\hard}$  as given in \eqref{eq:FFSP}.

\begin{problem}
    [Averaged mean curvature constraints with varying locations and hard-wall constraints] %
    \label{p:AMCLH} %
    Find $(u,X)\in V\times \omega$ minimizing the energy $\cJ(u)$
    subject to the constraint
    \begin{equation}\label{eq:AVH}
        g_X(u)=\bar{s}.
    \end{equation}
\end{problem}

Notice that Problem~\ref{p:AMCLH} is equivalent to minimizing
$\cJ(u) + \cV_{\hard}(X)$ over $V \times \omega$ under the
constraint~\eqref{eq:AVH}.
For soft-wall constraints, we obtain the following related problem.

\begin{problem}
    [Averaged mean curvature constraints with varying locations and soft-wall constraints] %
    \label{p:AMCLS} %
    Find $(u,X)\in V\times \omega$ minimizing the energy $\cJ(u) + \cV_{\soft}(X)$
    subject to the constraint
    \begin{equation*}
        g_X(u)=\bar{s}.
    \end{equation*}
\end{problem}

The key ingredient to show existence of solutions for Problems~\ref{p:AMCLH} and~\ref{p:AMCLS}
is the following lemma.

\begin{lemma}\label{lem:CONTINUITY_OF_AVERAGING}
    Let $M \subset \R^n$ be bounded and measurable and $p>1$.
    Then the operator
    \begin{align*}
        A: \Omega \to (L^p(\Omega)'), \qquad
        A(x)v = \int_{(M+x)\cap \Omega} v(\xi)\; d\xi
    \end{align*}
    is continuous.
    Especially, $A(\wc)v:\Omega \to \R$ is continuous for all $v \in L^p(\Omega)$.
\end{lemma}
\begin{proof}
    Assume that $v \in L^p(\Omega)$ is extended by zero to $\R^n$ and
    $q = \tfrac{p}{p-1}$ is the dual exponent.
    Now let $x,y \in \Omega$.
    Then H\"older's inequality gives
    \begin{align*}
        \begin{split}
            |A(x)v - A(y)v|
            &= \Bigl| \int_{\R^n} (\indicator_{M+x} - \indicator_{M+y})(\xi) v(\xi) d \xi \Bigr| \\
            &\leq \|\indicator_{M+x} - \indicator_{M+y}\|_{L^q(\R^n)} \|v\|_{L^p(\Omega)}\\
            &= \|\indicator_{M+x-y} - \indicator_{M}\|_{L^q(\R^n)} \|v\|_{L^p(\Omega)}
        \end{split}
    \end{align*}
    where $\indicator_{M+z} \in L^q(\R^n)$ is the
    indicator function of $M+z$ for $z \in \R^n$.
    Hence, by \cite[Lemma~4.3]{Brezis2011FunctionalAnalysis}, we have
        $\|A(x)-A(y)\|_{L^p(\Omega)'} 
        \leq \|\indicator_{M+x-y} - \indicator_{M}\|_{L^q(\R^n)}
        \xrightarrow[y\to x]{} 0$.
\end{proof}

\begin{proposition} \label{prop:EXISTENCE_AMCLH}
    Assume that $\omega \neq \emptyset$.
    Then there exists a solution $(u,X)\in V\times \omega$ to Problem~\ref{p:AMCLH}.
\end{proposition}
\begin{proof}
    We want to apply the general result of Proposition~\ref{QPP_Global_exist}
    in the appendix.
    To this end we first select an $X_0 \in \omega \neq \emptyset$.
    Note that we have $\cV_{\hard}(X_0)=0<\infty$.
    Furthermore Proposition~\ref{prop:AMC} implies that there is
    $u_0 \in V$ satisfying $g_{X_0}(u_0)=\bar{s}$.

    By Lemma~\ref{lem:COMPO}, the set $\omega$ is compact.
    Furthermore it is easily checked that the functional
    $\cV_{\hard}: \R^{N \times 2} \to \Rinfty$ is
    lower semi-continuous.
    In order to apply Proposition~\ref{QPP_Global_exist}, it remains to show that
    $X \mapsto g_X(\wc) \in V'$ is continuous.
    This follows from the representation of $g_X$ given in Remark~\ref{rem:CURVATURE},
    Lemma~\ref{lem:CONTINUITY_OF_AVERAGING} with $p=2$, and
    the continuity of $\Delta: V \to L^2(\Omega)$.
\end{proof}

Under a slightly stronger assumption we also get existence for soft-wall constraints.
\begin{proposition} \label{prop:EXISTENCE_AMCLS}
    Assume that $\op{int} \omega \neq \emptyset$.
    Then there exists a solution $(u,X)\in V\times \op{int}\omega$ to Problem~\ref{p:AMCLS}.
\end{proposition}
\begin{proof}
    We want to apply Proposition~\ref{QPP_Global_exist} in the appendix.
    To this end, we now fix $X_0 \in \op{int} \omega \neq \emptyset$.
    Then, equation~\eqref{eq:INTERIOR_OF_OMEGA} implies that we have
    $\op{dist}(B_i(X_0),B_j(X_0)) >0$ and $\op{dist}(B_i(X_0),\partial \Omega)>0$
    for $i \neq j$ and thus $\cV_{\soft}(X_0)<\infty$.
    Now we can literally proceed as in the proof of Proposition~\ref{prop:EXISTENCE_AMCLH},
    but with $\cV_{\hard}$ replaced by $\cV_{\soft}$.
    This provides a minimizer $(u,X) \in V \times \omega$.
    Since $\cV_{\soft}$ is infinite on $\omega \setminus \op{int} \omega$,
    we even have $X \in \op{int} \omega$.
\end{proof}

Note that we cannot expect uniqueness, e.g., for reasons of symmetry.

\subsection{Soft averaged mean curvature  constraints}
\subsubsection{Fixed locations of particles}
We skip penalty approximations of  the fixed-height  Problem~\ref{p:AC}
and directly concentrate on averaged mean curvature constraints \eqref{eq:AVCUCO} and varying heights
at fixed location of particles.
We introduce the penalty term
\begin{equation}\label{eq:SAMC}
    \frac{1}{2 \varepsilon}  \|g_X(u)-\bar{s}\|_{\R^N}^2
\end{equation}
with $g_X=(g_i)$ defined in \eqref{eq:CURVATCO} and  penalty parameter $\varepsilon >0$.
The corresponding minimization problem reads:

\begin{problem}[Soft averaged mean curvature constraints] \ \\  \label{p:SAMC}%
    Find $u_{\varepsilon}\in V$ minimizing 
    \[
        \cJ(u) +  \frac{1}{2 \varepsilon}   \|g_X(u)-\bar{s}\|_{\R^N}^2.
    \]
\end{problem}

An equivalent variational formulation amounts to find $u_{\varepsilon}\in V$ such that
\begin{equation} \label{eq:VARSAC}
    a(u_{\varepsilon},v)+ \frac{1}{\varepsilon}  \sum_{i=1}^N  g_i(u_{\varepsilon}) g_i(v) =
    \frac{1}{\varepsilon} \sum_{i=1}^N  \bar{s}_i g_i(v)\qquad \forall v\in V.
\end{equation}
The Lax-Milgram lemma yields existence and uniqueness.
\begin{proposition}
    There exists a unique solution $u_{\varepsilon}\in V$ to Problem~\ref{p:SAMC}.
\end{proposition}

Problem~\ref{p:SAMC} can be regarded as an approximation of Problem~\ref{p:AMC}.

\begin{proposition} \label{prop:SAMC}
    Let $u$ denote the solution of Problem~\ref{p:AMC} and  $u_\varepsilon$  
    denote the solution of Problem~\ref{p:SAMC}  for fixed $\varepsilon >0$.
    Then we have
    \[
        u_\varepsilon \to u \quad \text{in }  V \qquad \text{as }  \quad \varepsilon  \to 0.
    \]
\end{proposition}
\begin{proof}
The assertion is a direct consequence of Proposition~\ref{prop:A_PENALTY_CONV} in the appendix.
\end{proof}

\subsubsection{Varying the locations  of particles} \label{subsubsec:SAMCL}
We  first consider  hard wall constraints  $\cV_{\hard}$ as defined in \eqref{eq:HWC}. 
This means that we allow  the locations $X=(X_i)$  to vary only in the domain 
$\omega$ of $\cV_{\hard}$  as defined in \eqref{eq:FFSP}.
Recall that $\omega$ is compact in $\R^{N \times 2}$.

\begin{problem}[Soft averaged mean curvature constraints with varying locations and hard wall constraints] \label{p:SAMCLH}
    Find $(u_\varepsilon,  X_\varepsilon)\in V \times \omega$ minimizing the energy
    \[
      \cJ(u)+  \frac{1}{2 \varepsilon}  \|g_X(u)-\bar{s}\|_{\R^N}^2.
    \]
\end{problem}

\begin{proposition} \label{prop:SAMCLHEX}
Assume that $\omega \neq \emptyset$.
Then there exists a solution $(u_\varepsilon, X_\varepsilon)\in V \times \omega$  to Problem~\ref{p:SAMCLH}.
\end{proposition}
\begin{proof}
    In the proof of Proposition~\ref{prop:EXISTENCE_AMCLH}
    we have shown that Problem~\ref{p:AMCLH}
    satisfies the assumptions of Proposition~\ref{QPP_Global_exist} in the appendix.
    Under the same assumptions Proposition~\ref{QPP_Global_penalized} in the appendix
    provides existence  of a solution of its penalized analogue, i.e., Problem~\ref{p:SAMCLH}.
\end{proof}

Now we consider soft wall constraints associated with  the additional energy term $\cV_{\soft}(X)$
that penalize overlapping and escaping particles (cf.\ Subsection~\ref{sec:HARDSOFT}).

\begin{problem}[Soft averaged mean curvature constraints with varying locations and soft wall constraints] \label{p:SAMCLS}
    Find $(u_\varepsilon, X_\varepsilon)\in V \times  \op{int}\omega$ minimizing the energy
    \[
      \cJ(u)+  \frac{1}{2 \varepsilon}  \|g_X(u)-\bar{s}\|_{\R^N}^2 + \cV_{\soft}(X).
    \]
\end{problem}

\begin{proposition} \label{prop:SAMCLSEX}
    Assume that $\op{int} \omega \neq \emptyset$.
    Then there exists a solution $(u_\varepsilon,  X_\varepsilon)\in V \times \op{int}\omega$
    to Problem~\ref{p:SAMCLS}.
\end{proposition}
\begin{proof}
    Again we only need to note that the existence result
    of Proposition~\ref{QPP_Global_penalized} for penalized problems
    is valid under the same assumptions as the one for the non-penalized analogue
    in Proposition~\ref{QPP_Global_exist}, and that we already verified these
    assumptions in the proof of Proposition~\ref{prop:EXISTENCE_AMCLS}.
    Hence we have a solution $(u_\varepsilon,X_\varepsilon) \in V \times \omega$
    that must also satisfy $X_\varepsilon \in \op{int} \omega$ because we would
    otherwise have $\cV_{\soft}(X_\varepsilon)=\infty$.
\end{proof}

\subsection{Discussion}
To our knowledge,  averaged hybrid models  were not yet considered  in the existing literature.
They could be regarded as an intermediate approximation step
from  finite-size hybrid models, performing the coupling
by  boundary conditions at the particle boundaries  (cf.\ Section~\ref{sec:rpfs}),
to  well-established point-particle models to be considered below  
(cf.\ Section~\ref{sec:pcc}).
In particular,  
the transition from particle boundary conditions  on the membrane displacement and its normal derivative  (angle condition) 
to curvature constraints turns out to be equivalent to simple averaging of boundary conditions
(see, e.g.,  \cite{DomFou02} for another  motivation in terms of finite differences).
Further  transition to point-particle models  amounts to approximations of mean values by point values.
Averaged curve constraint models no longer represent tilt (cf.\ Remark~\ref{rem:TILTCURVE}),
but  still preserve some information on the shape of particles which is no longer present  in point-particle models.

Averaged hybrid models could as well be regarded as  regularizations of, 
in the first instance ill-posed, point-particle models.
The size of particles then acts as a regularization parameter
which,
in contrast to previous approaches (cf., e.g., \cite{NajAtzBro09,NajBro07}),
 is physically meaningful  and independent of discretization.

\section{Point  curvature constraints}\label{sec:pcc}
\subsection{Point approximation of mean values} When the particles have a small diameter with respect to the diameter of 
the domain then it is of interest to consider modelling the particles as points. 
One approach to obtaining such models is to replace integrals by point evaluations.
That is, the mean value
$\fint_{B_i}\Delta u\; dx$ is naturally approximated by $\Delta u(X_i)$ 
by sending the diameter of the 
particle $B_i$ to zero.  This may be understood in a different way as  
approximating the integrals in the averages \eqref{eq:DISCO}
by a first-order Gauss formula. 

The  constraints \eqref{eq:AVCUCO} 
in Problem~\ref{p:AMC} then take the form
\begin{equation}\label{eq:PC}
    \textstyle 
    G u= \left( \frac{1}{|B_i|}\bar{s}_i \right)
\end{equation} 
with $G=(G_i)$,  and functionals $G_i$ defined by
\begin{equation}\label{eq:GDELTA}
    G_i u= \delta_{X_{i}}(\Delta u), \qquad i=1,\dots,N,
\end{equation}
and given  $\bar{s}_i\in \R$ according to \eqref{eq:AVDAT}.
Here, 
\[
    \delta_{x} v = v(x),\qquad x\in \overline{\Omega},
\]
denotes the Dirac functional.


\subsection{Well posedness}
Due to the continuous embedding $H^2(\Omega)\subset C(\overline{\Omega})$
(see, e.g., \cite[Theorem~4.12]{AdamsFournier2003SobolevSpaces}), 
the Dirac functional  $\delta_x$ is a bounded linear functional on $H^2(\Omega)$.
However, the functionals $G_i$ are not well-defined on $v\in H^2(\Omega)$,
because  the linearised mean curvature $\Delta v \in L^2(\Omega)$ 
in general does not allow for point values.
In order to state a well-posed minimization problem
on a smaller solution space of  sufficiently regular functions,
we augment the Canham-Helfrich energy $\cJ$ defined in \eqref{eq:monge_energy}
by  additional higher order terms to obtain
\begin{equation} \label{eq:RHT}
 \cJJ(u)=\cJ(u) + \int_\Omega \textstyle \frac{\kappa_{8}}{2} |\Delta^2 u|^2+\frac{\kappa_{6}}{2} |\nabla \Delta u|^2\;dx,
 \qquad u\in H^4(\Omega),
\end{equation}
with some given regularization parameters  $\kappa_{8}$, $\kappa_{6}>0$.
This artificial  extension could be replaced by 
a more realistic fourth-order expansion
of the bending energy with respect to principal 
curvatures~\cite{C04-JudDiss1998,C04-Mitov1978}.

The strict positivity of $\kappa_{8}$ 
guarantees that functions which have bounded energy  lie in $H^4(\Omega)$.
In turn, the  continuous embedding $H^4(\Omega)\subset C^2(\overline{\Omega})$ 
implies that  the $G_i=\delta_{X_{i}}(\Delta\wc)$
are bounded linear functionals on $H^4(\Omega)$.
Note that the functionals  $G_i$ are linearly 
independent for distinct locations $X_i$, $i=1,\dots, N$.
This also holds for point values of  second-order derivatives
$\delta_{X_{i}}(\partial_{xx}\wc)$, $\delta_{X_{i}}(\partial_{xy}\wc)$,
and $\delta_{X_{i}}(\partial_{yy}\wc)$.

Differentiation  of $\cJJ$ yields the associated bilinear form
\begin{equation} \label{eq:HBI}
\aJ(u,v)= \int_\Omega \kappa_8\Delta^2 u \Delta^2 v + \kappa_6 \nabla\Delta u \cdot \nabla\Delta v + \kappa\Delta u \Delta v + \sigma \nabla u \cdot \nabla v.
\end{equation} 
The higher order terms in $\cJJ$ give rise to additional boundary conditions
defining a suitable closed solution space $\VJ\subset H^4(\Omega)$.
For example, we might choose
\[
\VJ = \begin{cases}
	H^4(\Omega) \cap H^3_0(\Omega)=\{v \in H^4(\Omega)\st v  = \textstyle \frac{\partial}{\partial n} v  = \frac{\partial^2}{\partial n^2}v =0\text{ on }\partial \Omega \}, \\
	\{v \in H^4(\Omega)\st v=0,\; \Delta v =0 \text{ on }\partial \Omega\}, \\
    H^4_{p,0}(\Omega) =\overline{\{v|_\Omega \st v \in C^\infty(\mathbb{R}^2) \text{ is } \Omega \text{-periodic and } \textstyle \int_{ \partial\Omega}v \; ds=0\}}  .
	\end{cases}
\]
For the final two cases we consider only rectangular domains $\Omega$.
Each choice for $\VJ$ also provides complementary natural boundary conditions for 
solutions to variational problems posed in that space. 
Observe that $\aJ(\wc,\wc)$ is bounded and symmetric on $\VJ$. It is also coercive for any 
$\kappa_8>0$ and $\kappa_6$, $\kappa$, $\sigma \geq 0$, we refer to \cite{Graeser2015, hobbsDiss16} for a proof. For a more detailed discussion of boundary conditions see \cite{hobbsDiss16}.

\subsection{Fixed locations of particles} 
We consider the following version of Problem~\ref{p:AMC} with hard point constraints.

\begin{problem}[Point mean curvature constraints]\ \\ \label{p:PC}%
    Find $u\in \VJ$ minimizing the energy $\cJJ(u)$ on $\VJ$ subject to the constraints \eqref{eq:PC}.
\end{problem}

\begin{proposition}
    There exists a unique solution $u\in \VJ$ to Problem~\ref{p:PC}.
\end{proposition}
\begin{proof}
    The particles $B_i$ are disjoint, so that we have  $X_i\neq X_j$, for $i\neq j=1,\dots, N$.
    Hence, the functionals $G_i$ are linearly independent.
    Now the assertion follows from
    Proposition~\ref{quad_exist_NEU} in the appendix.
\end{proof}

Possible anisotropies can be represented by the geometry of particles $B_i$
and boundary conditions in Problem~\ref{p:BVP}.
These are lost completely in the approximation by point mean curvature constraints. 
Accounting for anisotropies we now prescribe different curvatures 
\begin{equation} \label{eq:CF}
    G_{i,1} u=\delta_{X_i}(\partial_{xx} u),\quad G_{i,2} u=\delta_{X_i}(\partial_{xy} u)
        \quad G_{i,3} u=\delta_{X_i}(\partial_{yy} u)
\end{equation}
at one point $X_i$ for $i=1,\dots,N$. We set  $G=(G_{i,j})\in (\VJ')^{N \times 3}$.

\begin{problem}[Point curvature constraints]\ \\ \label{p:PDC}%
    Find $u\in \VJ$ minimizing the energy $\cJJ(u)$ on $\VJ$ subject to the constraints
    \begin{equation} \label{eq:GDEF}
        G(u)=r
    \end{equation}
    with given $r=(r_{i,j}) \in \R^{N \times 3}$.
\end{problem}

The functionals $G_{i,j}$ defined in \eqref{eq:CF} are linearly independent for 
distinct locations $X_i$.
Hence, existence and uniqueness again follows from Proposition~\ref{quad_exist_NEU}.

\begin{proposition}\label{prop:PDC_EXIST}
    There exists a unique solution $u\in \VJ$ to Problem~\ref{p:PDC}.
\end{proposition}

We now derive an explicit representation of $u$ in terms of Green's functions
$\phi_{i,j}\in \VJ$, which are defined as 
the unique  solutions of the variational problems
\begin{equation}\label{eq:GREENS_FUNC_DEF}
    \aJ(\phi_{i,j},v)=G_{i,j} v \qquad \forall v\in \VJ,\quad i=1,\dots,N, \quad j=1,2,3.
\end{equation} 
Note that existence and uniqueness of solutions $\phi_{i,j}$ 
of \eqref{eq:GREENS_FUNC_DEF}
follows from the Lax-Milgram lemma, because $ \aJ(\cdot,\cdot)$ is bounded and coercive 
and  the linear functionals $G_{i,j}$ are bounded on $\VJ\subset H^4(\Omega)$.

\begin{proposition} \label{prop:FLREP}
    Let  $A=(\aJ(\phi_{i,j},\phi_{k,l}))\in \R^{(N\times 3) \times (N \times 3)}$.
    Then 
    \begin{equation}
        u= \sum_{i=1}^{N} \sum_{j=1}^3 u_{i,j}\phi_{i,j}
    \end{equation}
    holds with $(u_{i,j})=A^{-1} r \in \R^{N \times 3}$.
\end{proposition}
\begin{proof}
    The assertion follows from the abstract Proposition~\ref{quad_exist_NEU}
    as applied to the special case $\ell=0$ and thus $\varphi=0$. 
\end{proof}

\begin{remark}
    Note that Proposition~\ref{quad_exist_NEU} also provides a 
    corresponding representation of the solution of Problem~\ref{p:AMC}. 
\end{remark}

We now consider {\em soft curvature constraints}. To this end we augment the energy $\cJJ$
by the penalty term
\begin{equation}\label{eq:PDCPENALTY}
    \frac{1}{2\varepsilon} \|G u- r\|_{\R^{N \times 3}}^2
\end{equation}
with some small penalty parameter $\varepsilon >0$
and the Frobenius norm $\| \cdot \|_{\R^{N \times 3}}$.

\begin{problem}[Soft point curvature constraints]\ \\ \label{p:SPDC}%
    Find $u_{\varepsilon}\in \VJ$ minimizing the energy 
    \[
        \cJJ(u_\varepsilon)+\frac{1}{2\varepsilon}  \|G u_\varepsilon - r\|_{\R^{N \times 3}}^2
    \]
    on $\VJ$.
\end{problem}

While existence and uniqueness follows from the Lax-Milgram lemma,
Proposition~\ref{prop:A_PENALTY_CONV} implies convergence to the hard-constrained solution.

\begin{proposition}\label{prop:SPDC}
    Let $u$ denote the solution of Problem~\ref{p:PDC} and  $u_\varepsilon$  
    denote the solution of Problem~\ref{p:SPDC}  for fixed $\varepsilon >0$.
    Then we have
    \[
        u_\varepsilon \to u \quad \text{in }  \VJ \qquad \text{as }  \quad \varepsilon  \to 0.
    \]
\end{proposition}

\begin{remark}\label{rem:MCPOINT}
The results stated in Proposition~\ref{prop:FLREP} and~\ref{prop:SPDC}) 
also hold literally for  $G=(G_i)\in (\VJ')^N$ with functionals $G_i$ defined in~\eqref{eq:GDELTA}.
\end{remark}

\subsection{Varying the locations of particles} 
We now seek  a global minimizer over prescribed curvatures 
in the sense that we allow for varying locations $X=(X_i)$ of particles.
To emphasize that the functionals
$G=(G_{i,j})$ defined in  \eqref{eq:CF}
depend on the locations $X_i$, we introduce the notation 
\begin{equation}\label{eq:GLOC}
    G_X= (G_{X,i,j}) \in (\VJ')^{N \times 3}
    \qquad G_{X,i} \in (\VJ')^3.
\end{equation}
First we consider hard-wall constraints, i.e., 
we restrict the locations $X$ to $\omega$ defined in \eqref{eq:FFSP}
so that the particles would not overlap.

\begin{problem}[Point curvature constraints with varying locations]\ \\ \label{min_curv}%
    Find $(u,X) \in \VJ\times \omega$ 
    minimizing the energy $\cJJ(u)$ on $\VJ$ subject to the
    constraint that there is a $X=(X_i)\in \omega$  
    such that 
    \[
        G_{X} u=r
    \] 
    holds with given $r\in \R^{N \times 3}$.
\end{problem}   

\begin{lemma}\label{lem:GX_SURJECTIVE}
    For any $X \in \omega$ the family $(G_{X,i,j}) \in (\VJ')^{N \times 3}$ of functionals $G_{X,i,j} \in \VJ'$
    is linearly independent and $G_X:\VJ \to \R^{N \times 3}$ is surjective.
\end{lemma} 
\begin{proof}
    Let $X \in \omega$.
    First we note that surjectivity of $G_X:\VJ \to \R^{N \times 3}$
    is equivalent to linear independence of the family $(G_{X,i,j})$.
    By the assumption $0 \in B_i^0$ we have $X_i \in B_i(X)$ and hence,
    by~\eqref{eq:FFSP}, $X_i \neq X_j$ for $i\neq j$.
    Hence we can construct smooth functions $v$ with $G_X v=r$
    for any $r \in \R^{N \times 3}$.
\end{proof}

\begin{proposition} \label{prop:min_curv_existence}
    Assume that $\omega \neq \emptyset$.  Then
    there exists a solution $(u,X) \in \VJ\times \omega$ to Problem~\ref{min_curv}.
\end{proposition}
\begin{proof}
    In order to apply Proposition~\ref{QPP_Global_exist}, we first note that,
    by Lemma~\ref{lem:GX_SURJECTIVE}, for any $Y \in \omega$ there is a $v\in \VJ$ with $G_{Y} v=r$,
    i.e., the feasible set  is non-empty.
    
    It remains to show that the mapping $\omega \ni X \to G_{X}  \in (\VJ')^{N \times 3}$ is continuous
    on the compact set $\omega$ (cf. Lemma~\ref{lem:COMPO}).
    To this end let $X,Y \in \omega$, $v \in \VJ$, $i \in \{1,\dots,N\}$, and, without loss of generality, $j=1$.
    Since the Sobolev embedding theorem provides
    $\VJ \subset H^4(\Omega) \to C^{2,\lambda}(\overline{\Omega})$
    for any H\"older-exponent $0<\lambda<1$
    (see, e.g., \cite[Theorem~4.12]{AdamsFournier2003SobolevSpaces}),
    we have
    \begin{align*}
        \begin{split}
            |(G_X - G_Y)_{i,j} v|
                = |\partial_{xx}v(X_i) - \partial_{xx}v(Y_i)|
                \leq  \|v\|_{C^{2,\lambda}} |X_i - Y_i|^\lambda
                \leq  C \|v\| |X_i - Y_i|^\lambda
        \end{split}
    \end{align*}
    and thus $\|G_X - G_Y\|_{(\VJ')^{N \times 3}} \to 0$ as $X \to Y$.
    This concludes the proof.
\end{proof}

We now provide a reformulation of Problem~\ref{min_curv} in terms of suitable Green's functions.

\begin{proposition}\label{lem:EQCOND}
    Assume that $\omega \neq \emptyset$.
    For given $Y \in \omega$, let the Green's functions $\phi_{Y,i,j}\in \VJ$
    and the matrix $A_Y=\aJ(\phi_{Y,i,j},\phi_{Y,k,l})\in \R^{(N\times 3) \times (N\times 3)}$ be defined 
    as in~\eqref{eq:GREENS_FUNC_DEF} and Proposition~\ref{prop:FLREP}, respectively.
    Then each solution $(u,X)$  of Problem~\ref{min_curv} has the representation
    \begin{equation} \label{eq:REPVL}
            u= \sum_{i=1}^N\sum_{j=1}^{3} u_{i,j} \phi_{X,i,j}, \quad (u_{i,j})=A_{X}^{-1}r
    \end{equation}
    where $X$ is  a minimizer of the mapping 
    \[
        \omega \ni Y \mapsto r^\top A_Y^{-1}r \in \R.
    \]
\end{proposition} 
\begin{proof}
    Note that Lemma~\ref{lem:GX_SURJECTIVE} implies that the family $(G_{Y,i,j})$
    is linearly independent for all $Y \in \omega$.
    Hence Proposition~\ref{min_con} can be applied for the special case $\ell=0$ (and thus $\phi_0=0$).
\end{proof}

Now we consider soft-wall constraints by augmenting the energy functional
with the term $\cV_{\soft}(X)$.

\begin{problem}[Point curvature constraints with varying locations and soft-wall constraints]\ \\ \label{min_curv_softwall}%
    Find $(u,X) \in \VJ\times \op{int} \omega$ 
    minimizing the energy $\cJJ(u) + \cV_{\soft}(X)$ subject to the
    constraint
    \[
        G_{X} u=r
    \] 
    for given $r\in \R^{N \times 3}$.
\end{problem}   

\begin{proposition} \label{prop:min_curv_softwall_existence}
    Assume that $ \op{int} \omega \neq \emptyset$. Then
    there exists a solution $(u,X) \in \VJ\times \op{int} \omega$ to Problem~\ref{min_curv_softwall}.
\end{proposition}
\begin{proof}
    Noting that $\cV_{\soft}(Y)<\infty$ for all $Y \in \op{int} \omega$,
    we can prove existence of solutions $(u,X) \in \VJ \times \omega$
    as in the proof of Proposition~\ref{prop:min_curv_existence}.
    Then $X \in \op{int} \omega$ follows from the definition of $\cV_{\soft}$.
\end{proof} 

\begin{proposition}\label{prop:min_curv_softwall_representation}
    Assume that $\omega \neq \emptyset$.
    Then for each solution $(u,X)$ of Problem~\ref{min_curv_softwall} we can represent $u$ as in
    Proposition~\ref{lem:EQCOND} where
    where $X$ is now a minimizer of the mapping 
    \[
        \omega \ni Y \mapsto r^\top A_Y^{-1}r + \cV_{\soft}(Y) \in \R.
    \]
\end{proposition} 
\begin{proof}
    The proof can be carried out literally as for Proposition~\ref{lem:EQCOND}.
\end{proof} 

\begin{remark}\label{rem:MCPOINTVAR}
    The results stated in the Propositions~\ref{prop:min_curv_existence},
    \ref{lem:EQCOND}, \ref{prop:min_curv_softwall_existence},
    and~\ref{prop:min_curv_softwall_representation} 
    also hold  for  $G_X=(G_{X,i})\in (\VJ')^N$ 
    with functionals $G_{X,i}=\delta_{X_i}(\Delta \wc)$ defined according to~\eqref{eq:GDELTA}.
\end{remark}

\begin{remark}
    A similar representation to \eqref{eq:REPVL} can be derived for the solutions
    of the average constrained
    Problems~\ref{p:AMCLH} and~\ref{p:AMCLS} with hard- and soft-wall constraints.
\end{remark}

\subsection{Unbounded domains} \label{subsec:UNBD}
In the preceding sections, we have  focussed on problems with particles on a membrane that is
parametrized over a bounded domain $\Omega\subset \R^2$.
However, our variational approach is not limited to this case. 
As an example, we will now consider a physical model
as suggested by Bartolo and Fournier~\cite{BarFou03} 
with an unbounded membrane parametrized over $\R^2$.
We will  formulate  this model
in terms of our variational framework 
and then use our general theory to recover some, but not all results
that were obtained  for bounded domains. 
In order to facilitate the transfer 
between~\cite{BarFou03} and our presentation, we will mostly adopt new notation to~\cite{BarFou03}.

\subsubsection{Hard and soft point constraints in $\R^2$}

Following Bartolo and Fournier~\cite{BarFou03}, we consider an extension
\begin{equation}
    F_m(u)=\cJJ(u) + \bfgamma \int_{\R^2} u^2 \;dx,\qquad u\in H^4(\R^2),
\end{equation}
of the energy functional  $\cJJ$ defined in \eqref{eq:RHT}
by an additional  confining potential $\bfgamma u^2$
with a given constant $\bfgamma>0$ (there is no danger of confusion with 
the height $\bfgamma$ used elsewhere in this paper). The coupling of the membrane with
$N$  pointwise isotropic particles  at pairwise distinct locations $X_i\in \R^2$
is represented by the interaction term
\begin{equation} \label{eq:INTAC}
    \frac{\bfGamma}{2}  \sum_{i=1}^N (G_{X,i} u-C_{i})^\top \bfN (G_{X,i} u -C_{i}) 
\end{equation}
with the $3\times 3$ matrix
\begin{equation}
    \bfN=
    \left(
        \begin{matrix}
            1+ \epsilon & 0         & \epsilon\\
            0           & 2 & 0\\
            \epsilon    & 0         & 1 + \epsilon
        \end{matrix}
    \right),
\end{equation}
the Dirac functionals $G_{X,i,j}$ defined in \eqref{eq:GLOC}, $i=1,\dots,N$,
and given data $C=(C_{i,j})\in \R^{N \times 3}$, $\bfGamma \geq 0$, and $\epsilon > -1/2$,
such that $\bfN$ is symmetric positive definite.
Now the model considered by Bartolo and Fournier~\cite{BarFou03} reads as follows.

\begin{problem}[Soft point curvature constraints in  $\mathbb{R}^2$] \ \\ \label{p:BFSOFT}%
    Find $u\in H^4(\R^2)$ minimizing the energy 
    \begin{equation} \label{eq:BFEN}
        F_m(u)+ \frac{\bfGamma}{2}\sum_{i=1}^N (G_{X_i}u-C_i)^\top \bfN (G_{X_i}u-C_i).
    \end{equation}
\end{problem}

\begin{proposition}
    There exists a unique solution to Problem~\ref{p:BFSOFT}.
\end{proposition}
\begin{proof}
    First we note that the bilinear form
    \begin{equation}
        a_m(u,v) = \aJ(u,v) + \bfgamma \int_{\R^2} u v\;dx
    \end{equation} 
    associated with the energy functional $F_m$
    is bounded on $H^4(\R^2)$.
    By partial integration and the fact that the $C^\infty$-functions with compact support
    are dense $a_m(\wc,\wc)$ is also coercive on this space.
    Furthermore the Sobolev embedding $H^4(\R^2) \to C^{2}(\R^2)$
    (see, e.g., \cite[Theorem~4.12]{AdamsFournier2003SobolevSpaces}),
    guarantees continuity of each $G_{X,i,j}$.
    Since $\bfN$ is symmetric and positive definite for $\epsilon>-1/2$
    this implies that the  bilinear form 
    \begin{equation*}
        a_m(u,v) + \frac{\bfGamma}{2}\sum_{i=1}^N (G_{X,i} u)^\top \bfN (G_{X,i} v)
    \end{equation*} 
    associated with the energy functional in \eqref{eq:BFEN}  is 
    symmetric and $H^4(\R^2)$-elliptic.
    Hence, the assertion follows from the Lax-Milgram lemma.
\end{proof}

In order to identify the hard constrained version of Problem~\ref{p:BFSOFT},
we reformulate the interaction term defined in \eqref{eq:INTAC} according to
\begin{equation}
     \frac{\bfGamma}{2}  \sum_{i=1}^N (G_{X,i} u-C_{i})^\top \bfN (G_{X,i} u -C_{i}) = 
     \frac{\bfGamma}{2}  \sum_{i=1}^N   |Q (G_{X,i} u- C_{i})|^2,
\end{equation}
where $Q = \bfN^{\frac{1}{2}} \in  \R^{3\times 3}$.
For increasing penalty parameter
$\bfGamma\to \infty$, we therefore obtain the following hard-constrained problem.

\begin{problem}[Point curvature constraints in  $\mathbb{R}^2$]\ \\
    \label{p:BFHARD} %
    Find $u\in H^4(\R^2)$ minimizing the energy $F_m(u)$ subject to the constraints
    \begin{equation} \label{eq:BFCON}
        G_{X} u =C.
    \end{equation}
\end{problem} 

\begin{proposition} \label{prop:UBDP}
    There exists a unique solution $u\in H^4(\R^2)$ to Problem~\ref{p:BFHARD}.
    Moreover, denoting the solution of Problem~\ref{p:BFSOFT} for fixed $\bfGamma\geq 0$
    by $u_\bfGamma$, we have 
    \begin{equation} \label{eq:PENLIM}
        u_\bfGamma \to u\quad \text{in} \quad H^4(\R^2)\quad \text{for}\quad \bfGamma \to \infty.
    \end{equation}
\end{proposition}
\begin{proof}
    First recall that the bilinear form $\aJ(u,v) + \bfgamma \int_{\R^2}u v\;dx$ is bounded, symmetric, and coercive
    on $H^4(\R^2)$ and that the operator $G_X$ is bounded.
    Furthermore by Lemma~\ref{lem:GX_SURJECTIVE} the functionals $G_{X,i,j}$ are
    linearly independent for pairwise distinct locations $X_i$.
    Hence existence and uniqueness follows from Proposition~\ref{quad_exist_NEU}.

    Finally, since $Q$ is regular, the constraint~\eqref{eq:BFCON} is equivalent to 
    \begin{equation*}
        Q(G_{X,i} u -C_{i})=0,
        \quad i=1,\dots,N
    \end{equation*}
    such that the convergence~\eqref{eq:PENLIM} is a consequence of Proposition~\ref{prop:A_PENALTY_CONV}.
    \end{proof}

We now derive a representation of the stationary energy $F_m(u)$
of Problem~\ref{p:BFHARD} which will  finally turn out to be 
a hard constrained limit version of the stationary energy obtained by Bartolo and Fournier~\cite{BarFou03}.
Since we would like to emphasize the dependence of this energy
on the locations $X=(X_i)$, $X_i\in \R^2$, we will denote the solution of Problem~\ref{p:BFHARD}
for fixed $X$ by $u_X$ from now on.
Then Proposition~\ref{quad_exist_NEU}
provides the representation
\begin{align*}
    u_X = \sum_{i=1}^N\sum_{j=1}^{3} u_{X,i,j} \phi_{X,i,j}, \qquad (u_{X,i,j})=A_X^{-1}C \in \R^{N \times 3}
\end{align*} 
with Green's functions $\phi_{X,i,j}$ and the
matrix $A_X \in \R^{(N\times 3) \times (N \times 3)}$
given by
\begin{align}\label{eq:UNBOUNDED_GREENS_FUNCTIONS}
    a_m(\phi_{X,i,j},v )=G_{X,i,j}(v) \quad \forall v \in H^4(\R^2)
\end{align} 
and $A_X=(a_m(\phi_{X,i,j},\phi_{X,k,l}))$, respectively.
Lemma~\ref{lem:min_con} implies that the energy at the minimizer
is given by
\begin{align} \label{eq:BFENERGY}
    F_m(u) = \tfrac{1}{2} C^\top A_X^{-1}C.
\end{align}
We conclude this section with  an explicit representation of the entries of $A_X$ 
as appearing in~\cite{BarFou03}.
To this end let $\cG \in H^4(\R^2)$ denote the Green's function
given by
\begin{align}\label{eq:eigth_order_green_def}
    a_m(\cG,v) = v(0) \qquad \forall v \in H^4(\R^2)
\end{align}
and define the differential operators
$\partial^{(1)} = \partial_{xx}$, $\partial^{(2)} = \partial_{xy}$,
and $\partial^{(3)} = \partial_{yy}$.
Then $\cG$ can be related to the Green's functions $\phi_{X,i,j}$ in the following way.
\begin{proposition} \label{prop:GREENSREP}
    For $j=1,2,3$ we have $\partial^{(j)} \cG(\wc - X_i) = \phi_{X,i,j}$.
\end{proposition} 
\begin{proof}
    Let $w \in C^\infty(\R^2)$ and set $v = \partial^{(j)}w(\wc + X_i) \in C^\infty(\R^2)$.
    Inserting this in \eqref{eq:eigth_order_green_def} gives
    \begin{align*}
        a_m(\cG,v) = v(0) = \partial^{(j)}w(X_i) = G_{X,i,j}(w).
    \end{align*}
    By the regularity result given in Lemma~\ref{BF_4o_reg}
    in the appendix we have $\cG \in H^6(\R^2)$.
    Hence we have $\partial^{(j)} \cG \in H^2(\R^2)$
    and partial integration and translation invariance of integrals over $\R^2$ yields
    \begin{align*}
        a_m(\cG,v)
        = a_m(\cG,\partial^{(j)}w(\wc + X_i))
        = a_m(\partial^{(j)}\cG,w(\wc + X_i))
        = a_m(\partial^{(j)}\cG(\wc -X_i), w).
    \end{align*}
    Since $C^\infty(\R^2)$ is dense in $H^4(\R^2)$
    we have shown that $\partial^{(j)}\cG(\wc -X_i)$ coincides
    with the unique solution $\phi_{X,i,j}$ of~\eqref{eq:UNBOUNDED_GREENS_FUNCTIONS}.
\end{proof} 
Now we can insert $\phi_{X,i,j} = \partial^{(j)}\cG(\wc-X_i)$
into~\eqref{eq:UNBOUNDED_GREENS_FUNCTIONS} to obtain
\begin{align*}
    (A_X)_{(i,j),(k,l)} = a_m(\phi_{X,i,j},\phi_{X,k,l}) = G_{X,i,j}\phi_{X,k,l} = \partial^{(j)} \partial^{(l)} \cG(X_k-X_i).
\end{align*} 
Therefore, in the case of $N=2$ particles, identify $\kappa A_X \in \R^{(2\times 3) \times (2\times 3)}$
with the matrix $\mathbf{M} \in \R^{6 \times 6}$ given in equation (8) of~\cite{BarFou03}.
Hence, \eqref{eq:BFENERGY} leads to
\[
    F_m(u_{X})=  {\textstyle \frac{1}{2}} \kappa C^\top( \mathbf{M} )^{-1}C
\]
which precisely agrees with equation (14) in~\cite{BarFou03}, 
i.e.,
 \[
 F_{\text{tot,min}} = \frac{1}{2} \kappa C^\top \left( \mathbf{M} + 
 \frac{\kappa}{\bfGamma} \left( \bfN \otimes 
 \left(   \begin{array}{cc}
 1 & 0 \\
 0 & 1\\ \end{array} \right)
 \right)^{-1}  \right)^{-1} C,
 \]
after formally taking the limit $\bfGamma\to \infty$. 
As a consequence, using the  approximation for the matrix 
$\mathbf{M}$ from~\cite{BarFou03}, 
our results are reproducing the interaction potential for hard 
constraints denoted by $F_{\text{int,hard}}(r) = F_{\text{int,hard}}(|X_2-X_1|)$
in equation (20) of~\cite{BarFou03}.    

\subsubsection{Varying the location of particles}
In analogy to Problem~\ref{min_curv} with locations of particles 
varying in the compact set $\overline{\Omega}$,
we now consider  varying locations in $\R^2$. We first fix some notation for the $\mathbb{R}^2$ analogue of the set $\omega \subset \overline{\Omega}^N$.  
\[
  \tilde{\omega} = \{X \in \mathbb{R}^{N \times 2} \st B_i(X) \cap B_j(X) =\emptyset  \;\forall i \neq j \}. 
\]
Recall $G_X=(G_{X,i,j})$ with $G_{X,i,j}$ defined in \eqref{eq:GLOC}
and let $C=(C_{i,j}) \in \R^{N\times 3}$ be given.

\begin{problem}[Point curvature constraints with varying locations in $\R^2$]\ \\ \label{p:BFVAR}%
    Find $(u,X)\in H^4(\R^2)\times \tilde{\omega}$ such that $u$ minimizes the energy $F_m(u)$ 
    subject to the constraint
    \begin{equation}
        G_X=C.
    \end{equation}
\end{problem}

Note that Proposition~\ref{QPP_Global_exist} can be no longer applied to prove existence,
because $\tilde{\omega}$ is not compact.
While we could still show that $\tilde{\omega}$ is closed, it fails to be bounded here.
Furthermore, we should not expect a solution to Problem~\ref{p:BFVAR} to exist. 
To see this, consider the case of $N=2$  particles at varying locations 
$X=(X_1,X_2)\in \tilde{\omega}$ and choose $C=(C_{1}, C_{2})$
in such a way that their interaction  energy decreases with separation. 
As an example, one might chose  identical isotropic particles  
as considered, e.g., by Bartolo and Fournier~\cite{BarFou03}. 
Let $u_X$ be the unique solution of 
the corresponding Problem~\ref{p:BFHARD} with fixed locations $X$. 
Then $F_m(u_X)$ is precisely the interaction energy 
and is well-known to strictly decrease as the points $X_1$ and $X_2$ are moved apart. 
As the distance between $X_1$ and $X_2$ can become arbitrary large,
there can be no minimal set of locations and thus  no solution of the corresponding Problem~\ref{p:BFVAR}.

\subsection{Discussion}
The constrained minimization
Problems~\ref{p:PC} and~\ref{p:PDC} stem from the models discussed, e.g., in
\cite{BarFou03,DomFou99,DomFou02,KimNeuOst98,MarMis02,NajAtzBro09,Net97, WeiDes13}.
See also \cite{GouBruPin93,ParLub96}. The addition of the higher order 
terms to the energy functional (see \eqref{eq:RHT}) to ensure well 
posedness is done in \cite{BarFou03}. Prior to this the ill posedness 
is dealt with by only studying large separation distances between 
particles \cite{MarMis02} or by truncating the Fourier expansion 
of the solution, termed the high wave-vector cutoff in \cite{DomFou02}.    
When put in our framework, these papers study an energy functional of the form given in  equation \eqref{eq:BFEN}. 

The hard inclusions limit as described in \cite{BarFou03} corresponds to the limit $\bfGamma\rightarrow \infty$, here we are able to rigorously understand this limit. 
By what we have established in Proposition~\ref{prop:UBDP}, 
the hard inclusions limit is equivalent to the quadratic minimization Problem~\ref{p:BFHARD}.

This limit problem with anisotropic particles prescribing curvatures as in \eqref{eq:CF} is studied in \cite{DomFou02}. The elastic interaction energy is calculated and a thermal equilibrium is approximated using a Monte-Carlo algorithm. In the equilibrium configuration proteins aggregate into one region and form an egg carton type structure with the anisotropic particles located at the saddle points of this structure. This equilibrium is analogous to the global minimizer in Problem \ref{p:BFVAR}.

\section{Point value constraints and  point forces }\label{sec:pvc}

\subsection{Particles interacting with the cytoskeleton}
In the preceding sections we have studied a variety of models  describing the 
interaction of lipid membranes with embedded particles, scaffolds or wrapped particles
(cf. Figure~\ref{C04-fig:INTERACTION1}).
We now  consider interactions of  the membrane with thin actin filaments
that are anchored to the cytoskeleton 
and may prescribe displacements of the membrane
or with particles  that apply forces
that may be due to either entropic or direct chemical
interactions (see, e.g.,~\cite{EvaTurSen03, GovGop06}).
In the mathematical models to be considered in this section,
these effects are represented by  point value constraints or  point forces.

\subsection{Point value constraints at fixed locations} \label{sec:PVF}
Prescribed point values at given locations $X=(X_i)\in \overline{\Omega}^N$ are represented by the constraints
\begin{equation} \label{dispcon}
    F_X u = \alpha
\end{equation}
with given $\alpha \in \mathbb{R}^N$ and   $F_X= (F_{X,i}) : V \to \R^N$  defined by
\begin{equation} \label{Fdelta}
    F_{X,i} v = \delta_{X_i} v\in \R,
\end{equation}
as illustrated in Figure~\ref{fig:INTERACTION2}.
\begin{figure}[ht]
    \includegraphics[width=0.25\textwidth]{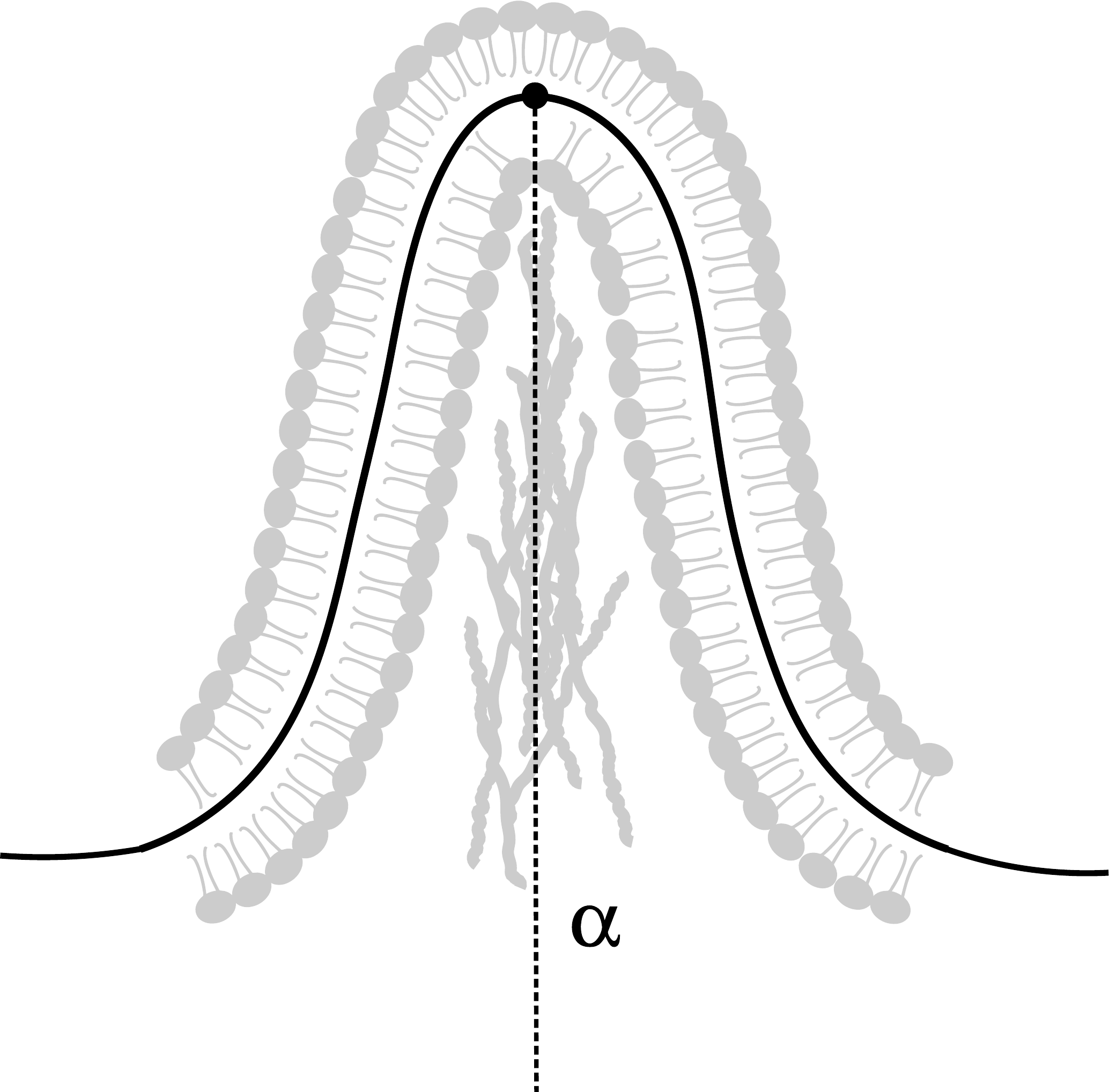}
   \caption{Displacements 
    caused by filaments anchored in the cytoskeleton.}
    \label{fig:INTERACTION2}
\end{figure}
Note that $\delta_{X_i}\in V'$ and thus $F_X\in (V')^N$
due to the continuous embedding 
$V\subset H^2(\Omega) \subset C(\overline{\Omega})$.
Hence, from a mathematical point of view, 
minimization problems with point value constraints \eqref{dispcon}
share  their basic properties 
with minimization problems with point curvature constraints \eqref{eq:PC}
as considered above.
We first consider prescribed  point values at fixed locations.

\begin{problem}[Point value constraints]\ \\ \label{p:PHC}%
Find $u \in V$ minimising the energy $\cJ$ on $V$ subject to the constraints \eqref{dispcon}.
\end{problem}

In order to avoid  possible conflicts of point constraints \eqref{dispcon}
with boundary conditions \eqref{eq:VDEF}, we  exclude $X_i \in \partial \Omega$.

\begin{proposition}
    For distinct locations $X_1, \dots, X_N \subset \Omega$ there exists a unique solution $u\in V$ to Problem~\ref{p:PHC}.
\end{proposition}
\begin{proof}
    As the locations  $X_1, \dots, X_N$ are distinct and
    contained in $\Omega$
    the functionals $\delta_{X_i}$ are linearly independent
    for all three  choices \eqref{eq:VDEF} of $V$.
    Hence, the assertion follows from Proposition~\ref{quad_exist_NEU}. 
\end{proof}

\begin{remark}
In applying Proposition~\ref{quad_exist_NEU}, we can also derive a representation of the solution
in terms of Green's functions $\phi_i\in V$, defined by
\[
a(\phi_i,v)=\delta_{X_i}(v)\quad \forall v\in V, \qquad i=1,\dots, N.
\]   
\end{remark}
\begin{remark}
Utilizing  general results stated in Propositions~\ref{lem:A_PENALIZED_MIN_EQIV} 
and~\ref{prop:A_PENALTY_CONV},
soft point value constraints at fixed locations can be treated in complete analogy to
soft point curvature constraints \eqref{eq:PC} as considered in  Problem~\ref{p:SPDC}
and Proposition~\ref{prop:SPDC}.
\end{remark}

\subsection{Point value constraints with varying locations} \label{sec:PVG}
\subsubsection{Existence of global minimizers}
We now seek a global minimizer over prescribed point values 
in the sense that we allow for varying locations $X=(X_i)\in \overline{\Omega}^N$.
This can be viewed as finding the optimal locations for filaments which prescribe a particular set of displacements.

\begin{problem}[Point value constraints with varying locations]\ \\ \label{min_heights}%
Find $(u,X) \in V \times \overline{\Omega}^N$ such that $u$ is minimising the energy 
$\cJ$ on $V$ subject to the constraint
\[
    F_X u = \alpha
\]
with given $\alpha \in \mathbb{R}^N$. 
\end{problem}

\begin{lemma}\label{lem:FCONT}
    The mapping 
    $\overline{\Omega} \ni X_i \to \delta_{X_i}\in (V')$
    and thus the mapping 
    $\overline{\Omega}^N \ni X \to F_X\in (V')^N$ is continuous.
\end{lemma}
\begin{proof}
    Since the Sobolev embedding theorem provides 
    continuity of the injection
    $V \subset H^2(\Omega) \to C^{0,\lambda}(\overline{\Omega})$
    for any H\"older-exponent $0<\lambda<1$
    (see, e.g., \cite[Theorem~4.12]{AdamsFournier2003SobolevSpaces}),
    the estimate
    \begin{align*}
        \begin{split}
            |(\delta_{X_i} - \delta_{Y_i}) v|
                = |v(X_i) - v(Y_i)|
                \leq  \|v\|_{C^{0,\lambda}} |X_i - Y_i|^\lambda
                \leq  C \|v\| |X_i - Y_i|^\lambda
        \end{split}
    \end{align*}
    holds for  each $i=1,\dots,N$.
    Thus $\|\delta_{X_i} - \delta_{Y_i}\|_{(V')^N} \to 0$ as $X \to Y$.
\end{proof}

\begin{proposition} \label{prop:FGHWC}
    There exists a solution $(u,X) \in V \times \overline{\Omega}^N$ to Problem \ref{min_heights}.
\end{proposition}

\begin{proof}
    It is well-known  that $\overline{\Omega}^N \in \R^N$ is compact.
    Furthermore, the mapping $\overline{\Omega}^N \ni Y \to F_Y\in (V')^N$
    is continuous by Lemma~\ref{lem:FCONT} and 
    $V_{\alpha,Y}=\{v\in V\st F_{Y} v=\alpha\}$ is non-empty for some
    $Y =(Y_i) \in  \overline{\Omega}^N$, e.g. for pairwise distinct locations $Y_i\in \Omega$.
    Now the assertion follows from Proposition~\ref{QPP_Global_exist}.
\end{proof}

\begin{remark}
Existence of a solution of  a penalized version of Problem~\ref{min_heights}
follows from Proposition~\ref{QPP_Global_penalized_exist} in complete analogy 
to penalized curvature constraints as considered in Problems~\ref{p:SAMCLH} and \ref{p:SAMCLS}.
\end{remark}

\subsubsection{Characterization of global minimizers}
Having shown the existence of  global minimizers we will now  produce equivalent characterizations
of solutions. First, we note considerable simplifications of Problem~\ref{min_heights} 
depending upon the signs of the prescribed point values.
%
%
%

\begin{proposition} \label{prop:SAMESIGN}
    Assume that the prescribed point values have the same sign
    and let $0\leq |\alpha_1| \leq \cdots \leq |\alpha_N|$.
    Then $(u,X)\in V\times \overline{\Omega}^N$ is a solution of Problem~\ref{min_heights}, if and only if
    $(u,X_N)\in V\times \overline{\Omega}$ solves Problem~\ref{min_heights} with $N=1$ and $\alpha=\alpha_N$.
    %
\end{proposition} 
\begin{proof}
    The solution of Problem~\ref{min_heights} is equivalent to solve
    \[
        u= \argmin{v \in V_{\alpha,N}}  \cJ(v), \qquad 
    V_{\beta,k}=\{ v \in V\st \exists Y \in \overline{\Omega}^k: \;\; \delta_{Y_i} v=\beta_i \textnormal{ for } i=1,\dots,k \},
    \]
    and to take $X \in \overline{\Omega}^N$ such that $F_X u= \alpha$.
    Hence, it is sufficient to show that $V_{\alpha,N}=V_{\alpha_1,1}$.
    The inclusion $V_{\alpha,N} \subset V_{\alpha_1,1}$ is obvious by definition.
    It remains to show $V_{\alpha_1,1} \subset V_{\alpha,N}$.

    To this end let $v \in V_{\alpha_1,1}$ and $X_N \in \overline{\Omega}$ such that $v(X_N)= \alpha_N$.
    Then, by continuity of $v$ on $\overline{\Omega}$,
    for all three  choices \eqref{eq:VDEF} of $V$
    there is $X_0\in \overline{\Omega}$ with $v(X_0)=0$.
    Now, by continuity of $v$ and convexity of $\Omega$,
    the intermediate value theorem implies that $v$ attains each value
    $\alpha_i \in \op{co}\{0, \alpha_N\}$
    at some point $X_i \in \overline{\Omega}$ and hence $v \in V_{\alpha,N}$.
\end{proof}

We now move on to the case where the prescribed point values $\alpha_i$  do not have the same sign.  Similarly to the previous case, the behaviour is governed by the extreme values of $\alpha$, in this case the greatest and least.

\begin{proposition}  \label{prop:DIFFSIGN}
    Let $\alpha_1 \leq ... \leq \alpha_N$.
    Then $(u,X)\in V\times \overline{\Omega}^N$ is a solution of Problem~\ref{min_heights}, if and only if
    $(u,(X_1,X_N))\in V\times \overline{\Omega}^2$ 
    solves Problem~\ref{min_heights} with $N=2$ and $\alpha=(\alpha_1,\alpha_N)$.
    %
\end{proposition}
\begin{proof}
    Utilizing the notation as introduced in  the proof of Proposition~\ref{prop:SAMESIGN},
    it is sufficient to show $V_{\alpha,N}=V_{(\alpha_1, \alpha_N),2}$.
    While $V_{\alpha,N} \subset V_{(\alpha_1, \alpha_N),2}$
    is obvious by definition, the converse inclusion again follows by convexity of $\Omega$,  continuity of $v\in V$,
    and the intermediate value theorem.
\end{proof}
%
%

We have thus reduced  $N$  to one or two constraints, based upon the signs of the prescribed point values.
We first concentrate on the case of point values with the same sign 
and reformulate Problem~\ref{min_heights} 
in terms of the Green's function $\phi_x\in V$, defined by
\begin{equation}\label{eq:GREEN_DEF}
    a(\phi_x,v)=\delta_{x} v \quad \forall v \in V,  \qquad x \in \overline{\Omega}.
\end{equation} 
By definition, $a(\phi_x,\phi_x)=\phi_x(x)$ holds for all $x \in \overline{\Omega}$.
In order to exclude the degenerate case $\phi_x=0$ we will constrain $x$ to the set
\begin{align*}
    \Omega_V = \{x \in \overline{\Omega} \st \delta_x \neq 0 \in V'\}.
\end{align*}
Notice that $\Omega_V$ depends on the choices \eqref{eq:VDEF} of the boundary
conditions incorporated in $V$:
For $V=H_0^2(\Omega)$ and $V=H^2(\Omega) \cap H_0^1(\Omega)$ we have $\Omega_V=\Omega$
whereas $V=H^2_{p,0}(\Omega)$ allows for $\Omega_V=\overline{\Omega}$.

\begin{proposition}  \label{one_height}
    Assume that the prescribed point values have the same sign and 
    let  $0\leq |\alpha_1| \leq \cdots \leq |\alpha_N|$.
    Then the solution of Problem~\ref{min_heights} is given by
    \begin{equation} \label{eq:1REP}
        X_N= \argmin{x\in \Omega_V} \frac{\alpha_N^2}{\phi_x(x)} , \qquad u= \frac{\alpha_N}{\phi_{X_N}(X_N)}\phi_{X_N}
    \end{equation}
    and $X_1,\dots, X_{N-1}$ such that $u(X_i) = \alpha_i$, $i=1,\dots, N-1$.
\end{proposition}
\begin{proof}
    First we note that~\eqref{eq:1REP} is well-defined because we
    have $\phi_x(x)\neq 0$ for $x \in \Omega_V$.
    Proposition~\ref{prop:SAMESIGN} implies that the solution $u$ 
    of Problem~\ref{min_heights}  is the minimizer of $\cJ$ subject to the constraint that
    $\delta_{X_N} u=\alpha_N$ holds with some  $X_N\in \overline{\Omega}$.
    For $\alpha_N=0$ we only have the trivial minimizer $u=0$ which
    is in accordance with~\eqref{eq:1REP}.

    Now let $\alpha_N\neq 0$.
    For $x \in \overline{\Omega} \setminus \Omega_V$ we have $\delta_x v=0 \neq \alpha_N$
    for all $v\in V$.
    Hence we must have $X_N \in \Omega_V$ for all solutions $(u,X_N)$ and
    the representation \eqref{eq:1REP} follows from Proposition~\ref{min_con} for $N=1$,
    $\cM = \overline{\Omega}$, {$\cM'=\Omega_V$}, and $T_1(y)=\delta_y$ for $y \in \cM$.
\end{proof}

Note that in the generic case $\alpha_N\neq 0$, the optimal location $X_N$  is independent of $\alpha_N$ and $u$ depends linearly on $\alpha_N$.

If the prescribed point values do not have the same sign,
then Problem~\ref{min_heights} can be reformulated in terms of 
two Green's functions $\phi_{Y_1}$, $\phi_{Y_2}\in V$ defined by 
\[
    a(\phi_{Y_1},v)=\delta_{Y_1}(v),\quad a(\phi_{Y_2},v)=\delta_{Y_2}(v)  \qquad \forall v\in V,
    \quad\qquad (Y_1,Y_2) \in \overline{\Omega}^2,
\]
and the associated Gramian matrix $A_Y=(a(\phi_{Y_j}, \phi_{Y_i}))=(\phi_{Y_i}(Y_j))$.
\begin{proposition} \label{two_heights}
    Let $\alpha_1 \leq \dots \leq \alpha_N$ and assume that $\alpha_1 < 0 < \alpha_N$.
    Then the solution of Problem~\ref{min_heights} is given by
    \begin{align}
        \label{eq:2REP_X}
        (X_1,X_N) &= \argmin{Y\in \Omega_V^2, Y_1 \neq Y_2} (\alpha_1,\alpha_N) A^{-1}_Y (\alpha_1,\alpha_N)^\top, \\
        \label{eq:2REP_u}
        u &= U_1 \phi_{X_1} + U_2 \phi_{X_N}  \qquad U = A^{-1}_{(X_1,X_N)}(\alpha_1,\alpha_N)^\top \in \R^2, 
    \end{align}
    and $X_2,\dots, X_{N-1}$ such that $u(X_i) = \alpha_i$, $i=2,\dots, N-1$.

\end{proposition}
\begin{proof}
    First we note that $A_Y$ is regular and~\eqref{eq:2REP_X}, \eqref{eq:2REP_u} are well-defined because we
    have $\phi_{Y_1} \neq 0 \neq \phi_{Y_2}$
    and $\phi_{Y_1} \neq \phi_{Y_2}$ for $Y \in \cM' = \{Y \in \Omega_V^2 \st Y_1 \neq Y_2\}$.
    
    Proposition~\ref{prop:DIFFSIGN} implies that  the solution $u\in V$
    of Problem~\ref{min_heights}  is the minimizer of $\cJ$ subject to the constraints that
    $\delta_{X_1} u=\alpha_1$, $\delta_{X_N} u=\alpha_N$  hold with some  $(X_1,X_N)\in \overline{\Omega}^2$.
    For $Y \in \overline{\Omega}^2 \setminus \cM'$ we either have $Y_i \in \overline{\Omega} \setminus \Omega_V$
    for some $i$ and hence $\alpha_1 < \delta_{Y_i} v = 0 < \alpha_N$ for all $v \in V$
    or we have $Y_1=Y_2$ and thus $\delta_{Y_1} = \delta_{Y_2}$ such that
    there is  again no $v \in V$  that
    satisfies the constraints $\delta_{Y_1}v =\alpha_1 < \alpha_N =\delta_{Y_2}v $.
    Hence,   $(X_1,X_N) \in \cM'$  must hold for all solutions $(u,(X_1,X_N))$ of Problem~\ref{min_heights}.
    Now the representation \eqref{eq:2REP_X}, \eqref{eq:2REP_u} follows from Proposition~\ref{min_con} for $N=2$,
    $\cM=\overline{\Omega}^2$, $\cM'$ as given above,
    and $(T_1(y_1), T_2(y_2))=(\delta_{y_1}, \delta_{y_2})$ for  $(y_1,y_2) \in \cM$.
\end{proof}

If the assumption $\alpha_1 < 0 < \alpha_N$ is not fulfilled, then
all $\alpha_i$ have the same sign.
In this case we can use the representation given by Proposition~\ref{one_height}.

\subsection{Point forces at fixed locations}
\label{sec:point_forces_fixed}
We now consider  forces exerted to the membrane
that are localized to certain points $X_i\in \overline{\Omega}$, $i=1,\dots,N$.
These forces are perpendicular to $\Omega$
with  positive or negative direction and give rise to the 
additional term 
\begin{equation} \label{eq:PFW_NEU}
    \ell_X(u) = \sum_{i=1}^N \beta_i \delta_{X_i}u
\end{equation} 
in the energy functional to be minimized.
Here, $\beta_i\in \R\setminus \{0\}$ are  given constants representing the magnitude of  point forces
at the locations $X_i$.
We set 
\[
    \cE(u,X)=\cJ(u)-\ell_{X}(u), \qquad u\in V, \quad 
     X \in \overline{\Omega}^{N}
\]
with the closed subspace $V \subset H^2(\Omega)$  defined in \eqref{eq:VDEF}
and consider the following minimization problem.

\begin{problem}[Point forces at fixed locations]\ \\ \label{pres_point}%
    For given $X = (X_i) \in \overline{\Omega}^N$ 
    find $u \in V$ minimising the energy $\cE(u,X)$ on $V$.
\end{problem}   

Existence and uniqueness of a solution $u \in V$
of Problem~\ref{pres_point}  follows from the Lax-Milgram lemma.
It is characterized by the variational equality
\begin{equation} \label{eq:POFOEQ}
    a(u,v)=\ell_X(v) \qquad \forall v \in V.
\end{equation}
The solution can be represented by Green's functions $\phi_x$ as defined
in~\eqref{eq:GREEN_DEF}.
\begin{lemma}\label{lem:POINT_FORCES_REPRESENTATION}
    For given $X \in \overline{\Omega}^N$ the solution $u\in V$ of Problem~\ref{pres_point}
    is given by
    \begin{align*}
        u = \sum_{i=1}^N \beta_i \phi_{X_i}.
    \end{align*}
\end{lemma} 
\begin{proof}
    The assertion follows directly from the linear representation \eqref{eq:PFW_NEU} of
    $\ell_X$ by the  functionals $\delta_{X_i}$.
\end{proof}

\subsection{Point forces at varying locations}\label{sec:point_forces_vary}
\subsubsection{Existence of global minimizers}
We now seek a global minimizer over prescribed point forces, in the sense that we allow the
point forces to be applied at  varying locations $X=(X_i) \in \overline{\Omega}^N$.

\begin{problem}[Point forces at varying locations]\ \\\label{point_force_glob_prob}%
    Find $(u,X) \in V \times \overline{\Omega}^N$
    minimising the energy 
    $\cE$  on  $V\times \overline{\Omega}^N$.
\end{problem}

\begin{proposition} \label{global_min}
    There exists a solution $(u,X) \in V \times \overline{\Omega}^N$
    to Problem \ref{point_force_glob_prob}. 
\end{proposition}
\begin{proof}
 In light of the continuity of $\overline{\Omega}\ni X_i\to \delta_{X_i}\in V'$ as stated in Lemma~\ref{lem:FCONT}, the assertion follows from Proposition~\ref{pro:PARSOU}. 
\end{proof}

In general, there is no  uniqueness of solutions of Problem~\ref{point_force_glob_prob}. 
For example, let $N=2$, $\beta_2 = -\beta_1$, and assume that
$(u, (X_1,X_2))$  is a solution of Problem~\ref{point_force_glob_prob}.
Then $(-u, (X_2,X_1))$  is another solution.

Using the representation for fixed $X$ given in
Lemma~\ref{lem:POINT_FORCES_REPRESENTATION},
we will also construct a representation of solutions to
Problem~\ref{point_force_glob_prob}.
To this end, we from now on denote by $u_Y \in V$ the unique
minimizer of $\cE(\wc,Y)$ for given $Y \in \overline{\Omega}^N$.
As a first step, we compute the energy of such minimizers.
\begin{lemma}\label{lem:POINT_FORCES_ENERGY}
    Let $Y \in \overline{\Omega}^N$ be given and $A_Y = (a(\phi_{Y_i}, \phi_{Y_j})) \in \R^{N \times N}$.
    Then
    \begin{align}
        \min_{v \in V} \cE(v,Y)
        = \cE(u_Y,Y)
         = - \tfrac{1}{2} a(u_Y,u_Y)
        = -\tfrac{1}{2} \ell_Y(u_Y)
        = - \tfrac{1}{2} \beta^\top A_Y \beta.
    \end{align} 
\end{lemma} 
\begin{proof}
    After inserting $v=u_Y$ into the variational equality \eqref{eq:POFOEQ} for $u_Y$, 
    we  use the definition of $\ell_Y$,
    and the representation of $u_Y$ as given in Lemma~\ref{lem:POINT_FORCES_REPRESENTATION}
    to obtain
    \begin{align*}
        \cE(u_Y,Y)
         = - \tfrac{1}{2} a(u_Y,u_Y)
        = -\tfrac{1}{2} \ell_Y(u_Y)
        = -\tfrac{1}{2} \sum_{i=1}^N \beta_i u_Y(Y_i)
        = -\tfrac{1}{2}  \sum_{i,j=1}^N  \beta_i \beta_j \phi_{Y_j}(Y_i).
    \end{align*} 
    Now definition~\eqref{eq:GREEN_DEF} of the Green's functions $\phi_{Y_i}$ yields
    $(A_Y)_{i,j} = a(\phi_{Y_i},\phi_{Y_j}) = \phi_{Y_j}(Y_i)$.
    This completes the proof.
\end{proof} 
As a direct consequence we get
\begin{proposition}\label{prop:POINT_FORCES_GLOBAL_MINIMITER}
    Let $A_Y \in \R^{N \times N}$ as in Lemma~\ref{lem:POINT_FORCES_ENERGY}.
    Then $(u,X) \in V \times \overline{\Omega}^N$ minimizes $\cE$, if and only if
    $u=u_X$ with $X$ minimizing the function
    \begin{align*}
        \overline{\Omega}^N \ni Y \mapsto -\tfrac{1}{2} \beta^\top A_Y \beta \in \R.
    \end{align*} 
\end{proposition}

\subsubsection{Clustering} \label{subsec:CLUSTF}
Having established the existence of global minimizers we will now explore
the properties of minimizers for particular combinations of the parameters
$\beta_i$, $i=1,\dots,N$.
Of particular interest will be exhibiting cases, 
where optimal locations of point forces lie inside $\Omega$. 
Of course, this is of no interest under periodic boundary conditions. 
As such, we will assume $V=H^2_0(\Omega)$ or $V=H^2(\Omega) \cap H^1_0(\Omega)$ for the rest of this section, 
but will remark where this assumption plays a role. 
We will also show clustering behaviour for larger numbers of point forces
 and that  opposite point forces do not annihilate each other. 
 
We first show that point forces 
do not cluster on the boundary $\partial \Omega$ of $\Omega$.
\begin{lemma} \label{lem:BND}
    Assume that $(u,X) \in V \times \overline{\Omega}^N$
    is a solution of Problem~\ref{point_force_glob_prob}.
    Then $\cE(u,X)<0$ and $X \notin (\partial \Omega)^N$.
\end{lemma}
\begin{proof}
    Assume that $(u,X)\in
    V\times (\partial \Omega)^{N}$
    solves Problem~\ref{point_force_glob_prob}. Then $\ell_X(v)=0$  holds for all $v\in V$ and therefore
    $u=u_X=0$. Hence, we have
    \[
        \textstyle \cE(u,X ) 
        = - \frac{1}{2}a(u,u)=0 > -\frac{1}{2}a(u_Y,u_Y)
        = \cE(u_Y,Y)
    \]
    for $Y=(Y_i)$ with $Y_1 \in \Omega$ and $Y_i=X_i$, $i=2,\dots,N$.
    This contradicts optimality of $(u,X)$.
\end{proof}

The following lemma quantifies the change of energy that is caused by moving
a single point force. This is the key ingredient to prove clustering of  point forces later on.
\begin{lemma}\label{lem:energy_single_particle_move}
    Let $X,Y \in \overline{\Omega}^N$ and assume
    $Y_i=X_i$ for $i\neq k$ with some fixed $k$.
    Then
    \begin{align*}
        \cE(u_Y,Y)
        = \cE(u_X,X) - \beta_k(\delta_{Y_k}-\delta_{X_k})(u_X)
        - \tfrac{1}{2} \beta_k^2 a(\phi_{Y_k} - \phi_{X_k},\phi_{Y_k} - \phi_{X_k}).
    \end{align*}
\end{lemma} 
\begin{proof}
    The representation of energy in Lemma~\ref{lem:POINT_FORCES_ENERGY}
    and the binomial formula provide the estimate
    \begin{align}\label{eq:energy_estimate_perturbed_forces}
        \cE(u_Y,Y)
            = \cE(u_X,X) - (\ell_Y-\ell_X)(u_X)
            - \tfrac{1}{2} a(u_Y-u_X,u_Y-u_X)
    \end{align}
    for any $X,Y \in \overline{\Omega}^N$.
    Now let $X_i=Y_i$ for $i\neq k$. Then
    we have $\ell_Y = \ell_X + \beta_k(\delta_{Y_k}-\delta_{X_k})$
    and $u_Y = u_X + \beta_k(\phi_{Y_k} - \phi_{X_k})$.
    Inserting these identities into~\eqref{eq:energy_estimate_perturbed_forces}
    we obtain the assertion.
\end{proof} 

In the forthcoming clustering analysis we will
make use of the equivalence relation
\begin{align} \label{eq:EQRE}
    x \REL y \quad \Leftrightarrow \quad  \delta_x v= \delta_y v \;\; \forall v\in V.\
\end{align} 
Recall that  we have $\delta_x =0$ on $V$ for all $x \in \partial \Omega$.
Hence, $x \REL y $ holds, if and only if $x=y$ or  $x,y \in \partial \Omega$.
By definition, the locations of a solution
can be replaced by equivalent locations,
i.e., if $(u,X)$ solves Problem~\ref{point_force_glob_prob} and
$Y_i\REL X_i$ holds for all $i=1,\dots,N$, then
$(u,Y)$ does also solve Problem~\ref{point_force_glob_prob}.

Now we are ready to prove clustering of point forces.

\begin{proposition}\label{prop:clustering_three_cases}
    Assume that $(u,X) \in V \times \overline{\Omega}^N$
    is a solution to Problem~\ref{point_force_glob_prob}.
    Then there exist $(X^+, X^-) \in \overline{\Omega}^2$
    such that $(X^+,X^-) \notin (\partial \Omega)^2$ and
    \begin{align} \label{eq:clustering_solution}
        \beta_i>0 \;\; \Rightarrow \;\;X_i \REL X^+, \qquad \quad
        \beta_i<0\;\; \Rightarrow \;\; X_i \REL X^-
    \end{align}
    holds for all $i=1,\dots,N$.
\end{proposition} 
\begin{proof}
    Let $(u_X,X) \in V \times \overline{\Omega}$ be a solution
    of Problem~\ref{point_force_glob_prob}.
    Then we have 
    \begin{equation} \label{eq:GEQH}
        \beta_i \delta_{X_i}(u_X) \geq 0 \qquad \forall =1,\dots, N.
    \end{equation}
    Indeed,  if there is a $k$ such that $\beta_k\delta_{X_k}(u_X) < 0$ then we can chose
    $Y_i=X_i$, $i\neq k$ and $Y_k\in \partial \Omega$ to obtain the contradiction 
    $\cE(u_Y,Y) < \cE(u_X,X)$ from Lemma~\ref{lem:energy_single_particle_move}.

    Recall that Lemma~\ref{lem:BND} implies  $X \notin (\partial \Omega)^N$. 
    Hence, at least one point force must be located in $\Omega$.
    Let $X_j \in \Omega$ be arbitrarily chosen.
    Then it is sufficient to show that  $X_i = X_j$  must hold for all
    $i \in \cI_j = \{l=1,\dots, N \st \op{sgn} (\beta_l)=\op{sgn}(\beta_j) \}$.

    Without loss of generality assume that $\beta_j>0$ and that
    $j$ is selected such that
    $\delta_{X_j}(u_X) \geq \delta_{X_i}(u_X)$ holds for all other $X_i \in \Omega$ with $i \in \cI_j$.
    Then the same estimate  is valid for  all $i \in \cI_j$,
    because \eqref{eq:GEQH} yields 
    $\delta_{X_j}(u_X) \geq 0= \delta_{X_i}(u_X)$ for all $X_i \in \partial \Omega$.

    In contradiction to the assertion, we now  assume that
     $X_k \neq X_j$ holds for some $k \in \cI_j$.
    Application of Lemma~\ref{lem:energy_single_particle_move}
    with $Y_i=X_i$, $i\neq k$ and $Y_k=X_j$,
    together with $\delta_{X_j}(u_X) \geq \delta_{X_k}(u_X)$ provides
    \begin{align*}
        \cE(u_Y,Y)
        &= \cE(u_X,X) - \beta_k(\delta_{X_j}-\delta_{X_k})(u_X)
        - \tfrac{1}{2} \beta_k^2 a(\phi_{X_j} - \phi_{X_k},\phi_{X_j} - \phi_{X_k}) \\
        &\leq \cE(u_X,X)
        - \tfrac{1}{2} \beta_k^2 a(\phi_{X_j} - \phi_{X_k},\phi_{X_j} - \phi_{X_k}).
    \end{align*}
    Now we have either  $X_k \in \partial \Omega$, and therefore  $ \phi_{X_k}=0$, or  $X_k \in  \Omega$,
    and therefore that $\phi_{X_k}$ and  $\phi_{X_j}$ are linearly independent. In both cases
    $a(\phi_{X_j} - \phi_{X_k},\phi_{X_j} - \phi_{X_k})>0$ providing $\cE(u_Y,Y) < \cE(u_X,X)$.
    This contradicts optimality of $(u_X,X)$.
\end{proof}

\begin{remark} \label{rem:CLUST}
As a consequence of Proposition~\ref{prop:clustering_three_cases}, 
the point forces with positive (negative) sign  either
cluster in one point $X^+\in \Omega$ ($X^-\in \Omega$) or 
are all located on the boundary $\partial \Omega$.
By Lemma~\ref{lem:BND}, not all $N$ point forces can be located at the boundary.
Hence,  point forces of a solution $(u,X)$ of Problem~\ref{point_force_glob_prob}
are clustering in exactly one of the following three ways.
\begin{enumerate}[label=(\roman{*})]
    \item \label{GlobMin_InDomain}
        $X_i = X^+ \in \Omega$ for all $i$ with $\beta_i>0$
        and $X_i = X^- \in \Omega$ for all $i$ with $\beta_i<0$,
    \item 
        $X_i = X^+ \in \Omega$ for all $i$ with $\beta_i>0$
        and $X_i \in \partial \Omega$ for all $i$ with $\beta_i<0$,
    \item \label{GlobMin_OnBoundary}
        $X_i = X^- \in \Omega$ for all $i$ with $\beta_i<0$
        and $X_i \in \partial \Omega$ for all $i$ with $\beta_i>0$.
\end{enumerate}
We may regard the occurrence of one of these three cases as a property of
$a(\wc,\wc)$, the parameters $\beta_i$, and $\Omega$.
\end{remark}

As another consequence of Proposition~\ref{prop:clustering_three_cases}
we can characterize the solutions to Problem~\ref{point_force_glob_prob}
with $N$ forces in terms of an equivalent problem with at most two forces.

\begin{corollary}\label{cor:clustering_two_particle_problem}
    Let 
    \begin{equation} \label{eq:CVAL}
    \beta^+=\sum_{\beta_i >0} \beta_i \geq 0, \qquad\beta^- = \sum_{\beta_i<0}\beta_i\leq 0.
    \end{equation}
    Then $(u,X) \in V \times \overline{\Omega}^N$ is a solution of Problem~\ref{point_force_glob_prob}
    with $(X^+,X^-) \in \overline{\Omega}^2$ satisfying   \eqref{eq:clustering_solution},
    if and only if  $(u,(X^+,X^-)) \in V \times \overline{\Omega}^2$
    is a solution of Problem~\ref{point_force_glob_prob} with point forces
    \begin{align*}
        \ell_{\XX} = \beta^+ \delta_{X^+} + \beta^-\delta_{X^-},\qquad \XX=(X^+,X^-).
    \end{align*}
\end{corollary} 
\begin{proof}
    Proposition~\ref{prop:clustering_three_cases} implies that
    the locations $X$ of all solutions to Problem~\ref{point_force_glob_prob} are contained in the subset
    \begin{align*}
        M = \bigl\{X \in \overline{\Omega}^N \st \exists (X^+,X^-) \in \overline{\Omega}^2 \text{ with \eqref{eq:clustering_solution} } \bigr\}  \subset  \overline{\Omega}^N.
    \end{align*}
    Hence minimizing $\cE(u,X)$ over $V\times \overline{\Omega}^N$ is equivalent to minimization 
    over $V\times M$.

    By definition of $M$, we can identify $M$ with $\overline{\Omega}^2$ 
    by the condition \eqref{eq:clustering_solution} up to componentwise equivalence 
    in the sense of \eqref{eq:EQRE}. 
    Now, let $X\in M$ be identified with $\XX=(X^+,X^-)\in \overline{\Omega}^2$  in this way.
    As a consequence of  \eqref{eq:clustering_solution},  we then have $\ell_X= \ell_{\XX}$
    and thus $u_X=u_{\XX}$. In light of Lemma~\ref{lem:POINT_FORCES_ENERGY}, this leads to
    \begin{equation*}
        \begin{split}
            \cE(u_X,X) &= - \tfrac{1}{2} a(u_X,u_X)= - \tfrac{1}{2}a(u_{\XX},u_{\XX})\\
                       &= \cJ(u_{\XX})- \ell_{\XX}(u_{\XX})=: \cE_0(u_{\XX},\XX)
        \end{split}
    \end{equation*}
    Therefore, minimization of $\cE(u,X)$ over $V\times M$ is equivalent to
    minimization of  the energy $\cE_0(u_{\XX},\XX)$ over $V\times \overline{\Omega}^2$.
    This concludes the proof.
\end{proof} 

By Proposition~\ref{prop:clustering_three_cases} 
at least one of the clustering  points $X^+$, $X^-\in \overline{\Omega}$ must be contained in $\Omega$.
Utilizing  the values of $\beta^+$ and $\beta^-$ defined in \eqref{eq:CVAL},
we can often exclude one of the three cases in Remark~\ref{rem:CLUST}.

\begin{proposition}
    Let $(u,X)$ be a solution of Problem~\ref{point_force_glob_prob}.
    If $|\beta^+| > |\beta^-|$, then  $X^+ \in \Omega$ and if $|\beta^+| < |\beta^-|$,
    then $X^- \in \Omega$.
\end{proposition} 
\begin{proof}
    Let $(u_X,X)$ be a solution of Problem~\ref{point_force_glob_prob}
    and $|\beta^+| > |\beta^-|$.
    In contradiction to the assertion, we assume that  $X^+ \in \partial \Omega$.
    Then,  Lemma~\ref{lem:BND} yields $X^- \in \Omega$ and thus $\delta_{X^-}\neq 0$.
    In addition, Corollary~\ref{cor:clustering_two_particle_problem} implies that
    $( u_X,(X^+,X^-))$ is a minimizer of  the energy $\cE_0 =\cJ - \ell_{\XX}$ on $V \times \overline{\Omega}^2$.
    From  $X^+ \in \partial \Omega$, we get $u_X = \beta^- \phi_{X_-}$.  This leads to 
    \begin{align*}
        \begin{split}
            \cE_0(u_X,(X^+,X^-))
                &= - \tfrac{1}{2} |\beta^-|^2 a(\phi_{X^-}, \phi_{X^-})\\
                &> - \tfrac{1}{2} |\beta^+|^2 a(\phi_{X^-}, \phi_{X^-})
                    = \cE_0(u_{(X^-,X^+)},(X^-,X^+))
        \end{split} 
    \end{align*} 
    in contradiction to the optimality of $(u_X,(X^+,X^-))$.
    In the remaining case $|\beta^+| < |\beta^-|$ the assertion follows by symmetry.
\end{proof} 

We now assume that all forces point in the same direction. In this case, 
the solutions of Problem~\ref{point_force_glob_prob} can be obtained 
by solving Problem~\ref{point_force_glob_prob} with a single force.

\begin{corollary}
    Assume that all of the coefficients $\beta_i$ have the same sign.
    Then $(u,X) \in V \times \overline{\Omega}^N$
    is a solution to Problem~\ref{point_force_glob_prob},
    if and only if
    $X_1= \cdots = X_N \in \Omega$
    and $(u,X_1)$ is a solution of Problem~\ref{point_force_glob_prob}
    with one point force
    \begin{align*}
        \ell_{X_1} = \Bigl(\sum_{i=1}^N \beta_i\Bigr) \delta_{X_1}.
    \end{align*}
\end{corollary} 
\begin{proof}
    By Lemma~\ref{lem:BND} there must be at least one $X_k \in \Omega$.
    Then, Proposition~\ref{prop:clustering_three_cases} provides  $X_1=\dots= X_k= \dots=X_N$
    and the assertion follows from Corollary~\ref{cor:clustering_two_particle_problem}.
\end{proof} 

We have thus characterised the behaviour of systems with forces  that are all  pointing in one direction:
The global minimizer is simply when all of the particles lie 
at the same point and that point is a global minimizer for only one point force. 
There is still no uniqueness however, as global minimizers for the one  point force
problem are not unique in general. The uniqueness for the one  point force
problem may be regarded as a property of the domain $\Omega$.

\begin{remark}
    All results given above can be extended to the case $V=H^2_{p,0}(\Omega)$ by replacing
    all occurrences of $\Omega$ by $\overline{\Omega}$ and dropping all cases where
    $\partial \Omega$ shows up.
\end{remark}

\subsection{Discussion}
The models formulated in Problem \ref{p:PHC} and Problem \ref{min_heights} describe the optimal shape of a membrane under point constraints and the optimal location of such constraints. This approach could be used to describe the action of actin filaments bound to the membrane. 
Such kind of problems also occur in the study of thin plates. 
For example, Problem \ref{p:PHC} is the central object in the study of thin plate splines
and  Problem \ref{min_heights} is analysed in \cite{ButNaz12} 
which studies support points of a plate, producing this problem with 
homogeneous data $\alpha=0$.

The model set out in  Problem~\ref{pres_point} and further extended  and analysed
in Section~\ref{sec:point_forces_vary}
is motivated by the general approach in \cite{EvaTurSen03} where protein membrane
interaction is described by an additional term
in the membrane energy functional
representing the work done by the pressure exerted by  proteins.
In  \cite{EvaTurSen03} the particles are assumed to have a positive diameter
and are bound to membrane. We have adapted this model to particles anchored
to the cytoskeleton applying point forces. Note that the results on clustering
of point forces, derived above, do not agree with the interaction of finite sized
particles, as investigated in \cite{EvaTurSen03}.
The key difference
between the two models is that the point forces in Problem~\ref{point_force_glob_prob} do apply a net force to the
membrane which is not the case for the interactions studied in \cite{EvaTurSen03}.
The action of protrusive forces on a membrane is discussed in \cite{GovGop06,VekGov07}.
The variational framework we have introduced may be also applied in this case and
used to analyse the  membrane mediated interactions between particles.

\section{Numerical experiments} \label{sec:NUMEX}

In this section, we present various numerical computations 
with hybrid models as introduced above.
The variational setting of these models naturally suggests discrete finite element counterparts
which will be specified and analysed in more detail in future publications.  
We first concentrate on a comparison of finite-size and point-like particle descriptions
(cf. Section~\ref{sec:rpfs} and \ref{sec:pcc}).
In particular, we investigate
the effect of the arrangement and the separation distance of particles 
on the membrane-mediated interaction energy.
Later, we complement our theoretical investigations of clustering of
point forces (cf.\ Section~\ref{sec:pvc}) 
by numerical computations.
Our numerical results are discussed in comparison with qualitative and quantitative 
results in the existing literature (see, e.g.,
\cite{BarFou03,EvaTurSen03,GouBruPin93,KimNeuOst98,WeiKozHel98}). 

\subsection{Finite size particles  and point curvature constraints} \label{subsec:FBCC}

We consider a membrane decorated with $N=2$ particles with circular or 
elliptical cross-sections and equal or opposite orientations.
Here, we focus on 
two different types of models, treating  particles as rigid objects 
with finite size (cf.\ Section~\ref{sec:rpfs}) 
or as described by point curvature constraints (cf.\ Section~\ref{sec:pcc}), respectively.

\subsubsection{Finite size particles}\label{subsubsec:CPBVP}

We consider a parametrized curve constrained problem on $V=H_0^2(\Omega)$
(cf.\ Section~\ref{sec:BC})
that involves both  varying height and tilt angle of the particle
as described in  Remark \ref{rem:TILTCURVE}.
However, the coupling to the membrane is now 
performed by boundary conditions on $\partial \Omega_B\setminus \partial \Omega$
instead of curve constraints on $\Gamma_i\subset \Omega$ (cf.\ Section~\ref{sec:rpfs}).
More precisely, we aim to approximate
$(u,\gamma,\alpha)\in V_{\Omega_B}\times\R^2\times\R^{2\cdot 2}$, 
minimizing the energy $\cJ(u)$ subject to the boundary conditions \eqref{eq:HEIGHTTILT1}
with unknown height  $\gamma=(\gamma_1, \gamma_2) \in \R^2$
and tilt $\alpha=(\alpha_1, \alpha_2) \in \R^{2\cdot 2}$ of particles.
In the elliptical case, the particle rotation
is an additional parameter that is fixed in each computation.

The discretization is performed by (non-conforming) Morley finite elements 
(see \cite{Ciarlet74},\cite{LasLes75}). 
Straightforward (elementwise) application of the
bilinear form \eqref{eq:bilinearform} to the
Morley finite element space yields
an indefinite stiffness matrix.
We therefore utilize the following equivalent reformulation of \eqref{eq:bilinearform}, depending on a parameter $c\in (0,\kappa]$
\begin{align*}
    \begin{split}
        a(v,w) =
        & \int_{\Omega_B} \left( (\kappa-c) \Delta v \Delta w + c \left(\sum_{i,j = 1}^{2} \frac{\partial^{2}v}{\partial x_i\partial x_j} \frac{\partial^{2}w}{\partial x_i\partial x_j}\right) + \sigma \nabla v \cdot \nabla w\right)\,dx \\
        & + c \Biggl(
            \left\langle \frac{\partial^{2}v}{\partial\tau^{2}},\frac{\partial w}{\partial n}\right\rangle_{\partial \Omega_B}
            - \left\langle\frac{\partial^{2}v}{\partial\tau\partial n},\frac{\partial w}{\partial\tau}\right\rangle_{\partial \Omega_B}
            \Biggr),
    \end{split}
\end{align*}
where $\langle \wc, \wc \rangle_{\partial \Omega_B}$ stands for the dual pairing of 
$H^{-\frac{1}{2}}(\partial \Omega_B)$ and $H^{\frac{1}{2}}(\partial \Omega_B)$. 
For sufficiently smooth functions $\langle \wc, \wc \rangle_{\partial \Omega_B}$ is just the $L^2$-inner product on $\partial \Omega_B$.
For details on the numerical treatment of this boundary term,
we refer to~\cite{C04-wolfDiss15}.
We use the parameter value  $c=\kappa$ in all subsequent computations.

We compute approximate displacements $u$ of a membrane 
over the domain $\Omega=(-3,3)^2$, generated by two particles 
that are either circular with radius $r=0.1$ 
or elliptical with major and minor half-axis $a=0.14$ and $b=0.06$. 
Note that the unit of length is $10r$, i.e.,
length is measured in terms of the particle size. 
We choose height functions $h_1 = h_2 = 0$, and
normal derivatives with equal  $s_1 = s_2 = 1$ (left) or
opposite orientation $s_1 = -s_2 = 1$ (right), respectively.  
We select the material parameter $\kappa=1$ and also fix the unit of energy in this way.
Figures \ref{fig:CIRC} and~\ref{fig:CROSS} show the clipping of the results to $(-1,1)^2$ for $\sigma=0$. 
The elliptical particles are arranged collinearly (\;\HEllipse\;\HEllipse\;)
for equal orientation (cf.\ Figure~\ref{ellipse_sameOrientation})
or parallel (\;\VEllipse\;\VEllipse\;)
for opposite orientation (cf.\ Figure \ref{ellipse_differentOrientation}).
As illustrated in Figure~\ref{fig:ARRA}, 
these configurations seem to be energetically optimal, 
in the sense of having lower energies than other configurations.
In these figures, the approximate interaction potential,
i.e., the approximate minimal energy $\cJ(u)$, 
of collinear and parallel arrangements is plotted 
over the separation distance $R$ between the two particles, here defined as the shortest distance between the particle's boundaries.
While collinear particles are energetically preferred for equal orientation
(cf.\ Figure~\ref{ellipse_sameOrientation_horizontalVsVertical}),
parallel arrangement is preferred for oppositely oriented particles
(cf.\ Figure~\ref{ellipse_diffOrientation_horizontalVsVertical}).
For other arrangements we found
energy curves lying in between these two extremal configurations.
We depict the interaction potential for an orthogonal arrangement
(\;\VEllipse\;\HEllipse\;) as an example.
These numerical results confirm related theoretical considerations in \cite{KimNeuOst00}.

\begin{figure}
\begin{subfigure}[t]{0.45\linewidth}
\centering
\includegraphics[width=1\textwidth]{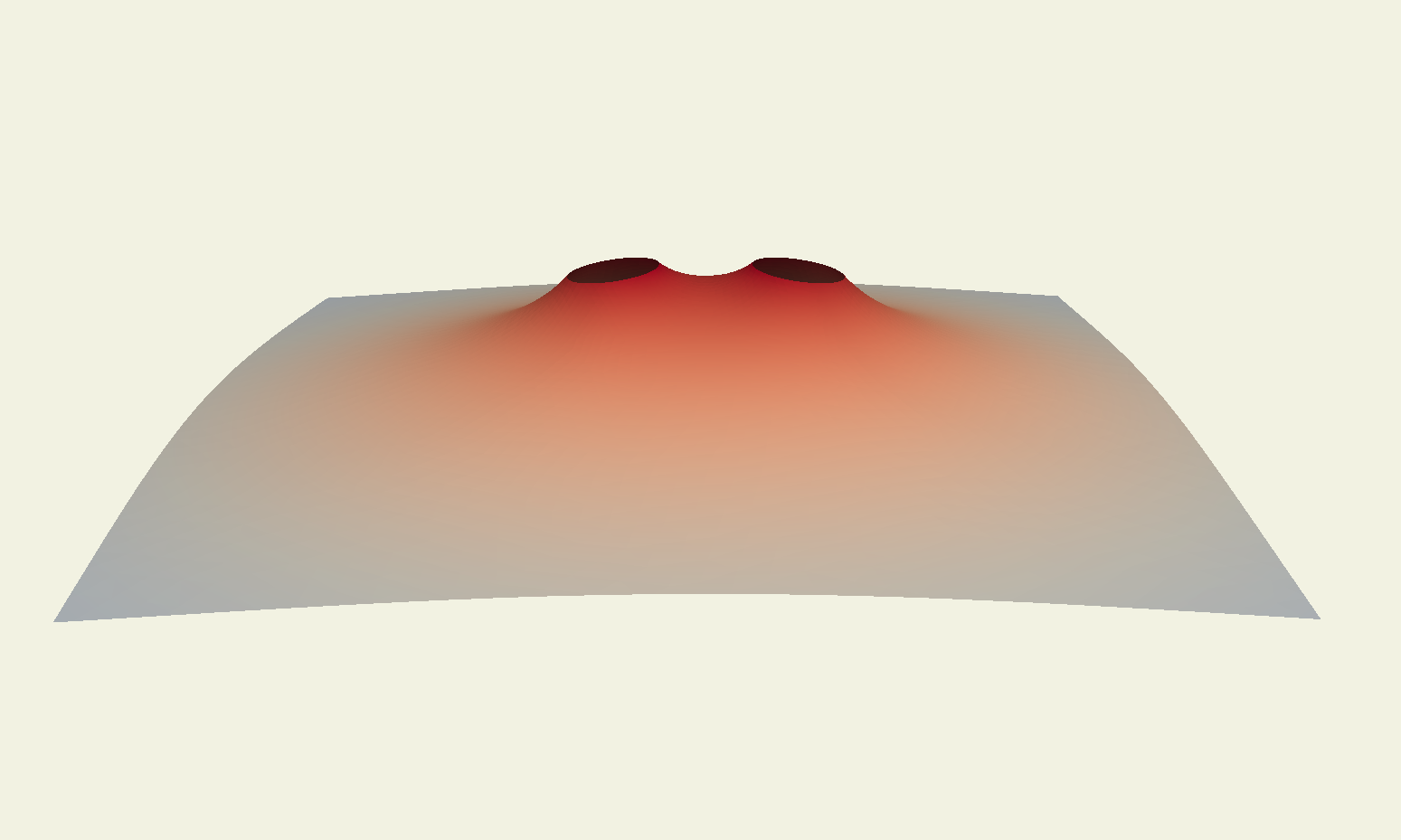}
\caption{equally oriented particles 
}
\label{circle_sameOrientation}
\end{subfigure}
\begin{subfigure}[t]{0.45\linewidth}
\centering
\includegraphics[width=1\textwidth]{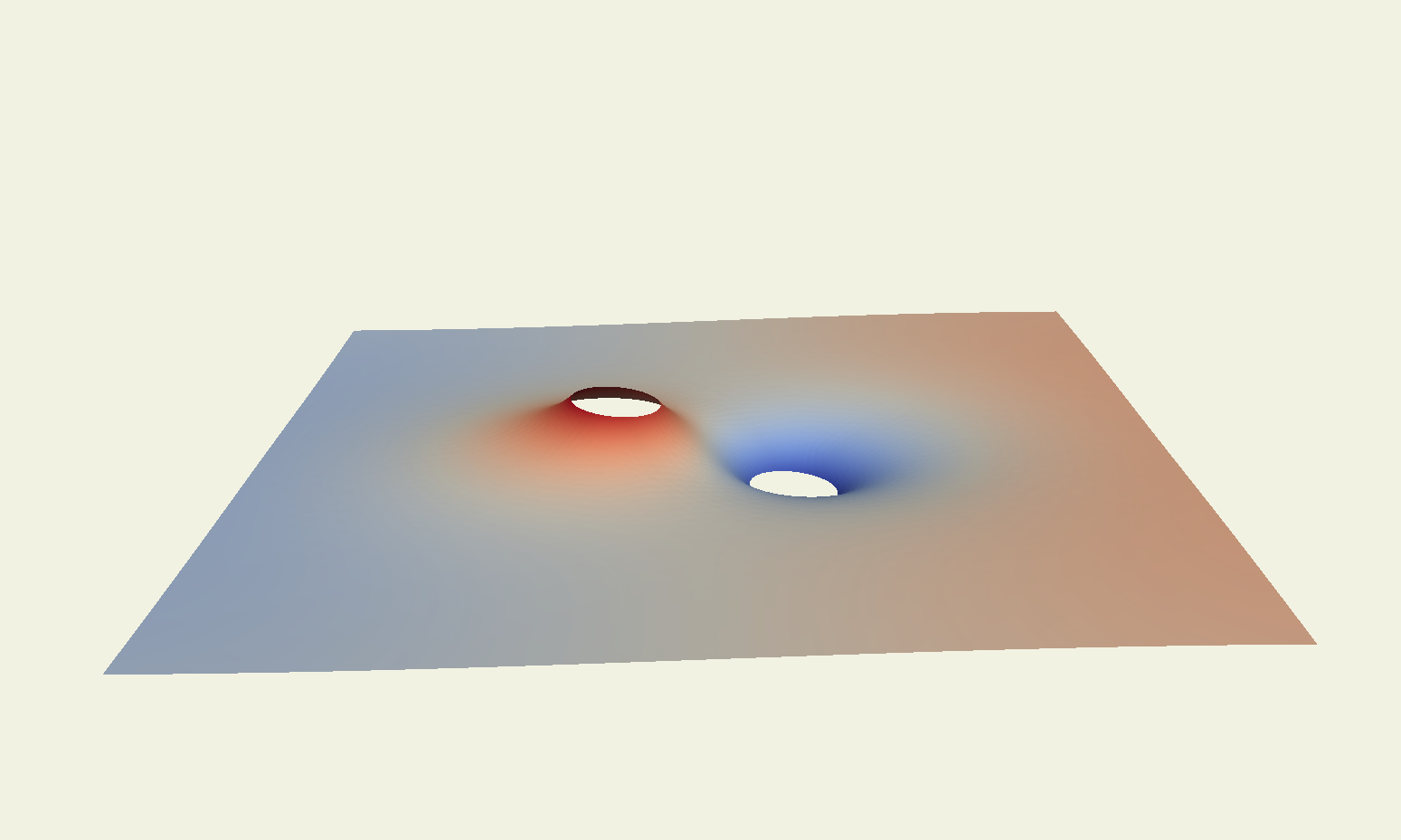}
\caption{oppositely oriented particles 
}
\label{circle_differentOrientation}
\end{subfigure}
\caption{Approximate membrane displacement for particles with circular cross-section.} \label{fig:CIRC}
\end{figure}
\begin{figure}
\begin{subfigure}[t]{0.45\linewidth}
\centering
\includegraphics[width=1\textwidth]{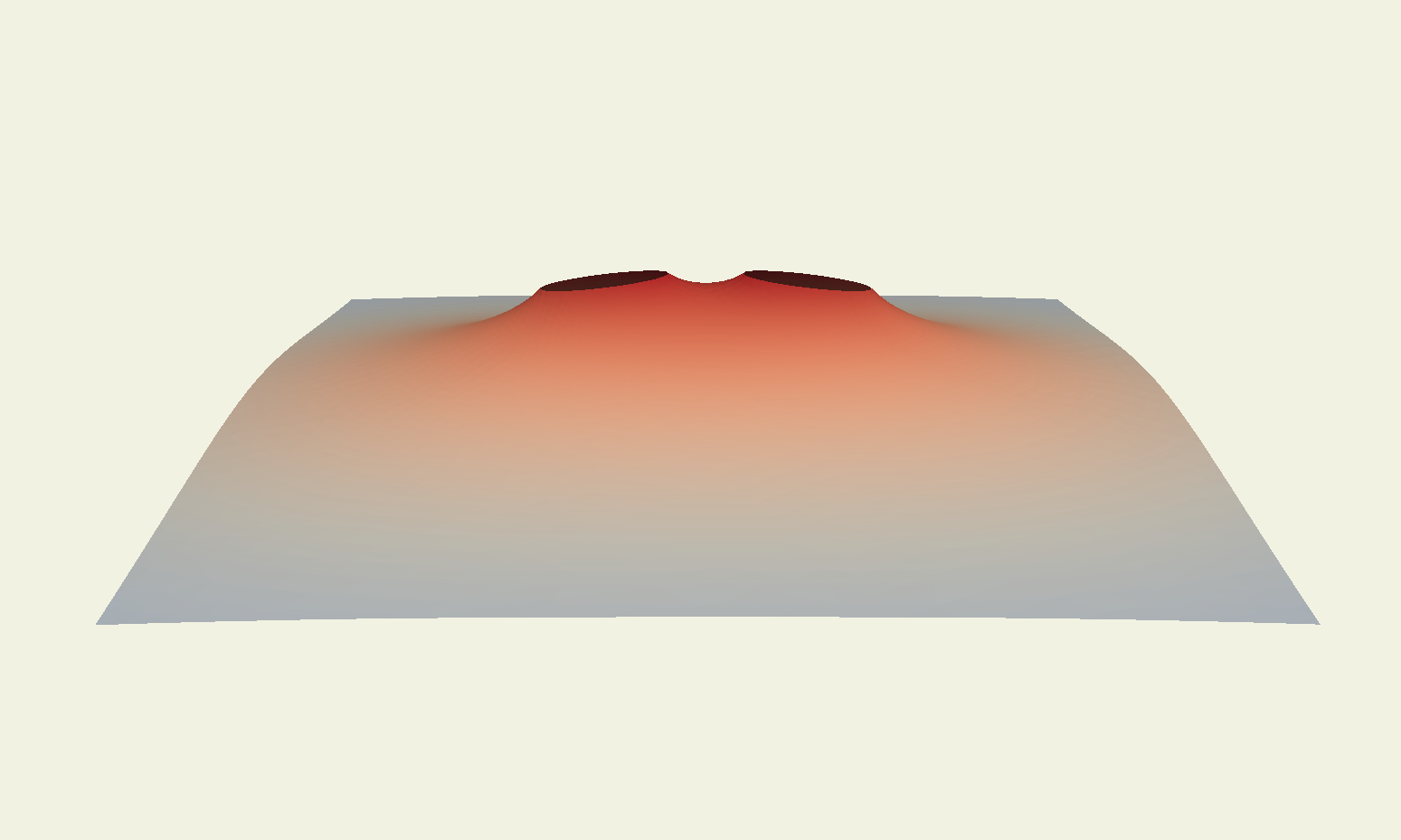}
\caption{equally oriented particles 
}
\label{ellipse_sameOrientation}
\end{subfigure}
\begin{subfigure}[t]{0.45\linewidth}
\centering
\includegraphics[width=1\textwidth]{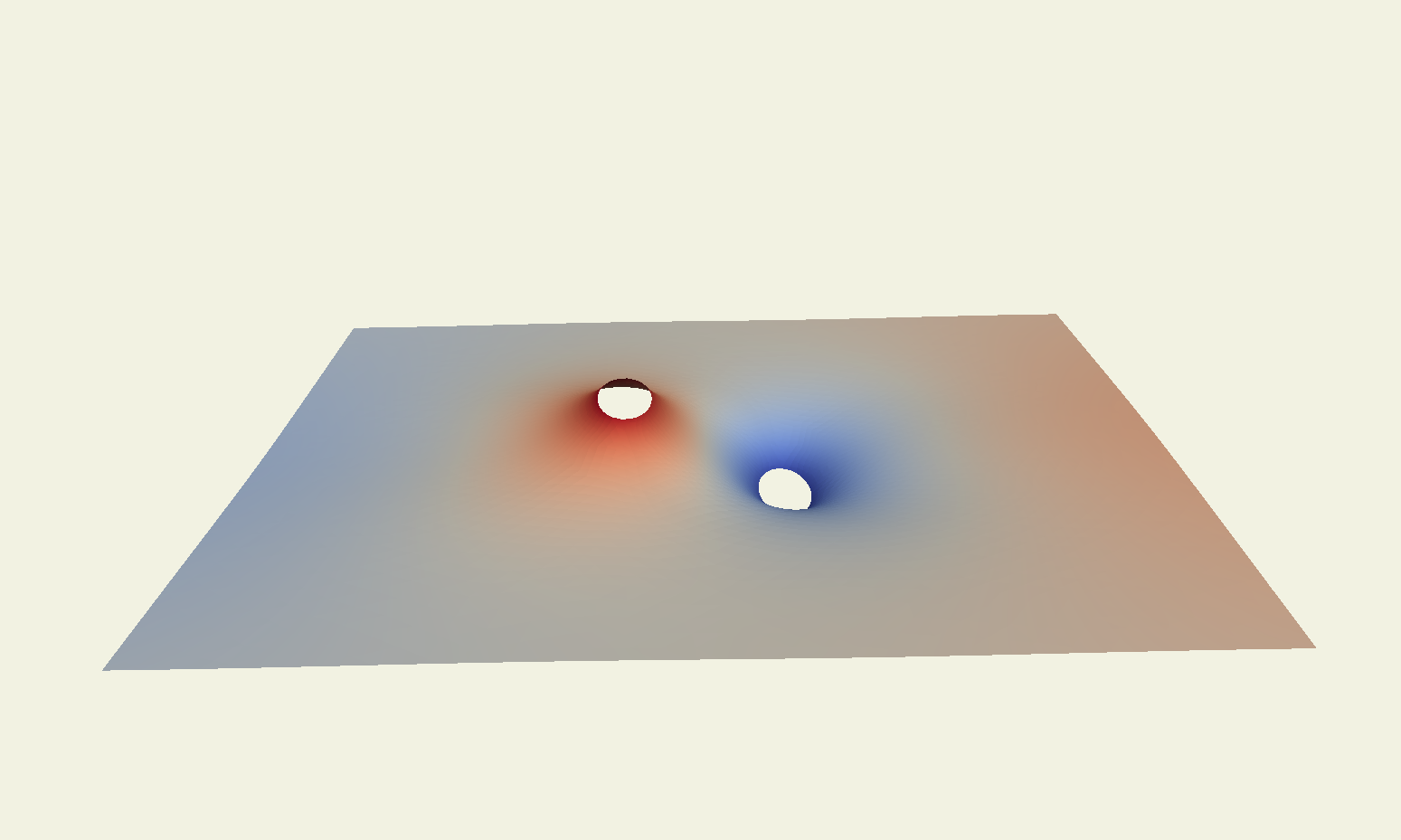}
\caption{oppositely oriented particles 
}
\label{ellipse_differentOrientation}
\end{subfigure}
\caption{Approximate membrane displacement for particles with 
elliptical cross-section.}\label{fig:CROSS}
\end{figure}

\begin{figure} 
\begin{subfigure}[b]{0.45\linewidth}
\centering
\includegraphics{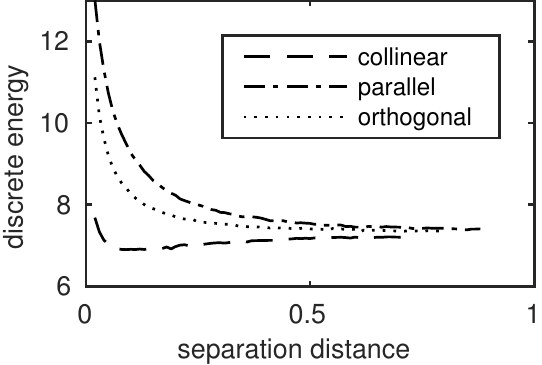}
\caption{equally oriented particles 
}
\label{ellipse_sameOrientation_horizontalVsVertical}
\end{subfigure}
\begin{subfigure}[b]{0.45\linewidth}
\centering
\includegraphics{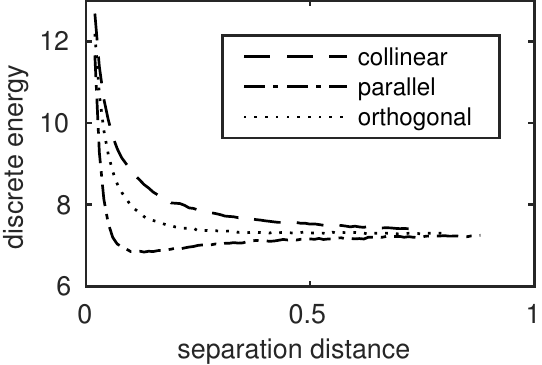}
\caption{oppositely oriented particles 
}
\label{ellipse_diffOrientation_horizontalVsVertical}
\end{subfigure}
\caption{Interaction potential over separation distance for particles 
with elliptical cross-section with collinear, parallel, and orthogonal major axes.}\label{fig:ARRA}
\end{figure}

We now study the membrane-mediated interaction potential
as a function of the separation distance $R$ between two particles which are arranged as in Figure~\ref{fig:CIRC} and \ref{fig:CROSS}.
For ease of comparison with 
point-like particle descriptions (cf.\ Subsection~\ref{subsec:PCCN}), 
the separation distance $R$ from now on stands for the distance 
between the midpoints of  particles. In order to reduce the possible influence of
the (artificial) boundary $\partial \Omega$,
we select separation distances $R\leq 1$.

Figure~\ref{fig:ENERGYCIRCLE} shows the interaction potential
over the separation distance $R$  for two  particles with circular cross-section
and various values of  $\sigma$.
We choose $\sigma=0$ and $\sigma = 1, 4, 9, 16, 25$
which, for a particle size of $3 \hspace*{1pt} nm$, corresponds to 
usual interaction lengths $\sqrt{\kappa/\sigma}$ 
ranging from $6 \hspace*{1pt} nm$ to $30 \hspace*{1pt} nm$~\cite{EvaTurSen03}.
We say that the interaction is repulsive or attractive if the derivative of the energy 
with respect to the separation distance is negative or positive, respectively.
In the case of equally oriented particles, i.e., for contact angles $s_1 = s_2=1$, 
we observe repulsion at all separation distances $R$ and for all values of $\sigma$.
 (cf.\ Figure~\ref{energy_circle_sameOrientation}).
In contrast, for oppositely oriented particles, 
i.e., for  contact angels $s_1 = -s_2=1$,
the interaction is repulsive for small and attractive for larger distances $R$ 
provided that $\sigma>0$ (cf.\ Figure~\ref{circle_differentOrientation}).
The preferred separation distance $R^\ast >0$, i.e., the distance resulting in the lowest energy,
decreases with increasing  membrane tension.
In both cases, the interaction potential becomes larger for growing $\sigma$. 
For  vanishing membrane tension, the interaction is always repulsive.

In the  case of particles with elliptical cross-section, 
we obtain a fundamentally different picture:
Figure~\ref{fig:ENERGYELLIPSE} shows the interaction potential 
over the separation distance for various elliptical shapes  in the case $\sigma=0$.
Both for equal (left) and opposite orientation (right),
the  interaction is repulsive for small and attractive for larger separation distances.
The resulting preferred separation distance $R^\ast>0$
decreases with increasing ratio $a/b$, i.e. for  slimming  particles.
Increasing membrane tension $\sigma$ also yields 
a decrease in preferred distance $R^\ast$ and an
increase in interaction potential, similar to the circular case.

\begin{figure}
\begin{subfigure}[b]{0.45\linewidth}
\centering
\includegraphics{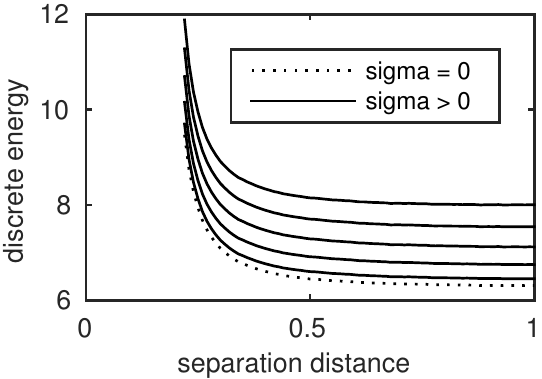}
\caption{equally oriented particles}
\label{energy_circle_sameOrientation}
\end{subfigure}
\begin{subfigure}[b]{0.45\linewidth}
\centering
\includegraphics{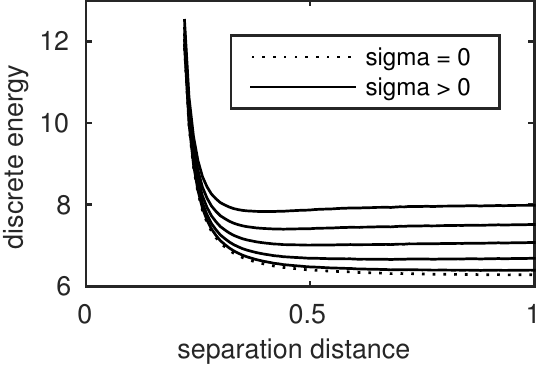}
\caption{oppositely oriented particles}
\label{energy_circle_diffOrientation}
\end{subfigure}
\caption{Interaction potential over separation distance for 
particles with circular cross-section with radius 
$r=0.1$ for membrane tension $\sigma = 0,1,4,9,16,25$ (bottom up).}\label{fig:ENERGYCIRCLE}
\end{figure}

\begin{figure}
\begin{subfigure}[b]{0.45\linewidth}
\centering
\includegraphics{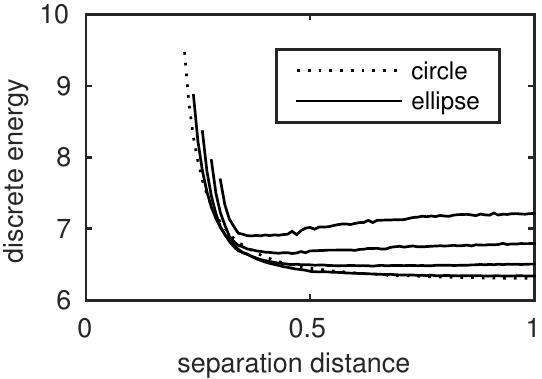}
\caption{equally oriented particles}
\label{energy_ellipse_sameOrientation}
\end{subfigure}
\begin{subfigure}[b]{0.45\linewidth}
\centering
\includegraphics{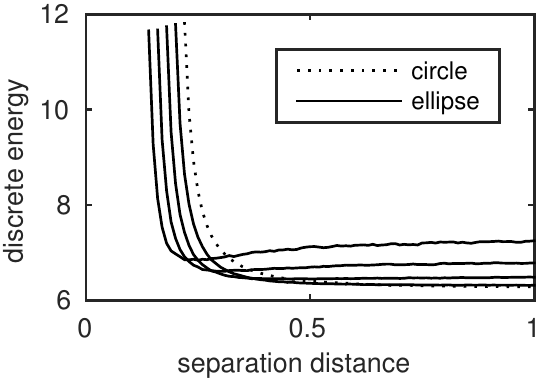}
\caption{oppositely oriented particles}
\label{energy_ellipse_diffOrientation}
\end{subfigure}
\caption{Interaction potential over separation distance for particles with elliptical cross-section 
with half-axes $(a,b)=(0.1,0.1),\ldots ,(0.14,0.06)$ (bottom up) for  membrane tension $\sigma = 0$.}\label{fig:ENERGYELLIPSE}
\end{figure}

\subsubsection{Point curvature constraints} \label{subsec:PCCN}
We now consider the interaction of particles and the membrane 
described by point mean curvature constraints as stated 
in Problem~\ref{p:PC}. 
We select the solution space 
$\VJ=\{v \in H^4(\Omega)\st v=0,\; \Delta v = 0 \text{ on }\partial \Omega\}$.
Our numerical approximation is based on the 
penalized formulation of Problem~\ref{p:SPDC}
with $G=(G_i)\in (\VJ')^N$ and $G_i=\delta_{X_i}(\Delta \wc)$
defined in \eqref{eq:GDELTA}. Recall that  the solution to  Problem~\ref{p:SPDC} 
converges to the solution of Problem~\ref{p:PC} 
as the penalty parameter $\varepsilon$ tends to zero (cf. Remark~\ref{rem:MCPOINT}).

The numerical approximation utilizes a splitting of the eighth order 
Problem~\ref{p:SPDC} into an equivalent system of three  second order equations
for the unknown functions  $u$, $w=\Delta u$ and $z=\Delta w$.
For this purpose, we formally rewrite the energy  $\cJJ(u)$ defined in \eqref{eq:RHT} in terms of $w$: 
\begin{equation} \label{eq:ENGW}
{\textstyle \frac{1}{2}} \int_\Omega \kappa_8(\Delta w)^2 + \kappa_6 |\nabla w|^2 + 
\kappa w^2 + \sigma |\nabla \Delta^{-1}w|^2 \; dx + \frac{1}{2\varepsilon} \sum_{k=1}^N (\delta_{X_k}w - r_k)^2.
\end{equation}
Note that the  corresponding Euler-Lagrange equation is fourth order in $w$.
 We impose essential boundary conditions $u=0$, $w=\Delta u=0$, and $z=\Delta^2 u=0$
 on $\partial \Omega$.
For  fixed $p \in (2,\infty)$ and $q \in (1,2)$ such that $1/p+1/q=1$,
we consider the variational problem to find
 $(u,w,z) \in H^1_0(\Omega) \times W^{1,p}_0(\Omega) \times W^{1,q}_0(\Omega)$ 
 such that the three equations
\begin{equation} \label{eq:SYSTEM}
\begin{array}{rl}
\displaystyle \int_\Omega -\kappa_8 \nabla z \cdot \nabla v + \kappa_6 \nabla w \cdot \nabla v + \kappa wv - \sigma uv \; dx \qquad &  \\
\displaystyle + \frac{1}{\varepsilon} \sum_{k=1}^N w(X_k)v(X_k) 
\hspace*{-2mm}&=  \displaystyle  \frac{1}{\varepsilon} \sum_{k=1}^N r_k v(X_k), \\[4mm]
\displaystyle \int_\Omega \nabla u \cdot \nabla v + wv \; dx  \displaystyle  \hspace*{-2mm} &=0, \\[4mm]
\displaystyle \int_\Omega \nabla w \cdot \nabla v + zv \; dx \displaystyle  \hspace*{-2mm} &=0
\end{array}
\end{equation} 
hold for all $v \in W^{1,p}_0(\Omega)$.

Regularity of solutions to Problem \ref{p:SPDC} implies that 
$u\in\VJ$ solves Problem~\ref{p:SPDC}, if and only if 
$(u,\Delta u, \Delta^2 u)$  solves the system \eqref{eq:SYSTEM}.
For details we refer to \cite{hobbsDiss16}.
The system \eqref{eq:SYSTEM} is finally discretized by $P^1$ finite elements. 
The penalty parameter is taken as 
$\varepsilon= 1 \times 10^{-8}$ in our computations.
For the material parameters we use $\kappa_8=1.23\times 10^{-6}$ 
and  $\kappa_6=1.11\times 10^{-3}$ or, equivalently,
fix the ratios $(\kappa_6/\kappa)^{1/2}=(\kappa_8/\kappa)^{1/4}=1/30$.
Selecting $\kappa=1$ and the same physical length scale 
as in the previous finite size particle case,
this means that both ratios correspond to $1 \hspace*{1pt}nm$ which is of order of 
the thickness of the membrane as suggested in  \cite[section 2]{BarFou03}. 
We also choose the computational
domain to be $\Omega=(-3,3)^2$. The values $r_k = \pm 20$ for the point constraints 
associated with circular particles of radius 0.1 
are obtained from the approximation \eqref{eq:PC}.

Figures \ref{point_curv_same} and \ref{point_curv_opp} plot the clipping to $(-1,1)^2$ of the approximate membrane displacement
obtained for equal and opposite curvature constraints, respectively, which correspond to equal and opposite 
orientations of particles. Investigating their interaction potential for
various values of membrane tension $\sigma$ 
in analogy to the previous section,
Figure~\ref{energy_pointcurv_same} shows that equally oriented particles repel each other for separation distances $R>0.2$. 
The strength depends upon membrane tension $\sigma$. 
For distances $R<0.2$ we observe attraction. For oppositely oriented particles the interaction is repulsive for small and attractive for 
larger separations as depicted in Figure~\ref{energy_pointcurv_opp}.
The strength of the attraction increases with $\sigma$. 
\begin{figure}
\begin{subfigure}[t]{0.45\linewidth}
\centering
\includegraphics[width=1\textwidth]{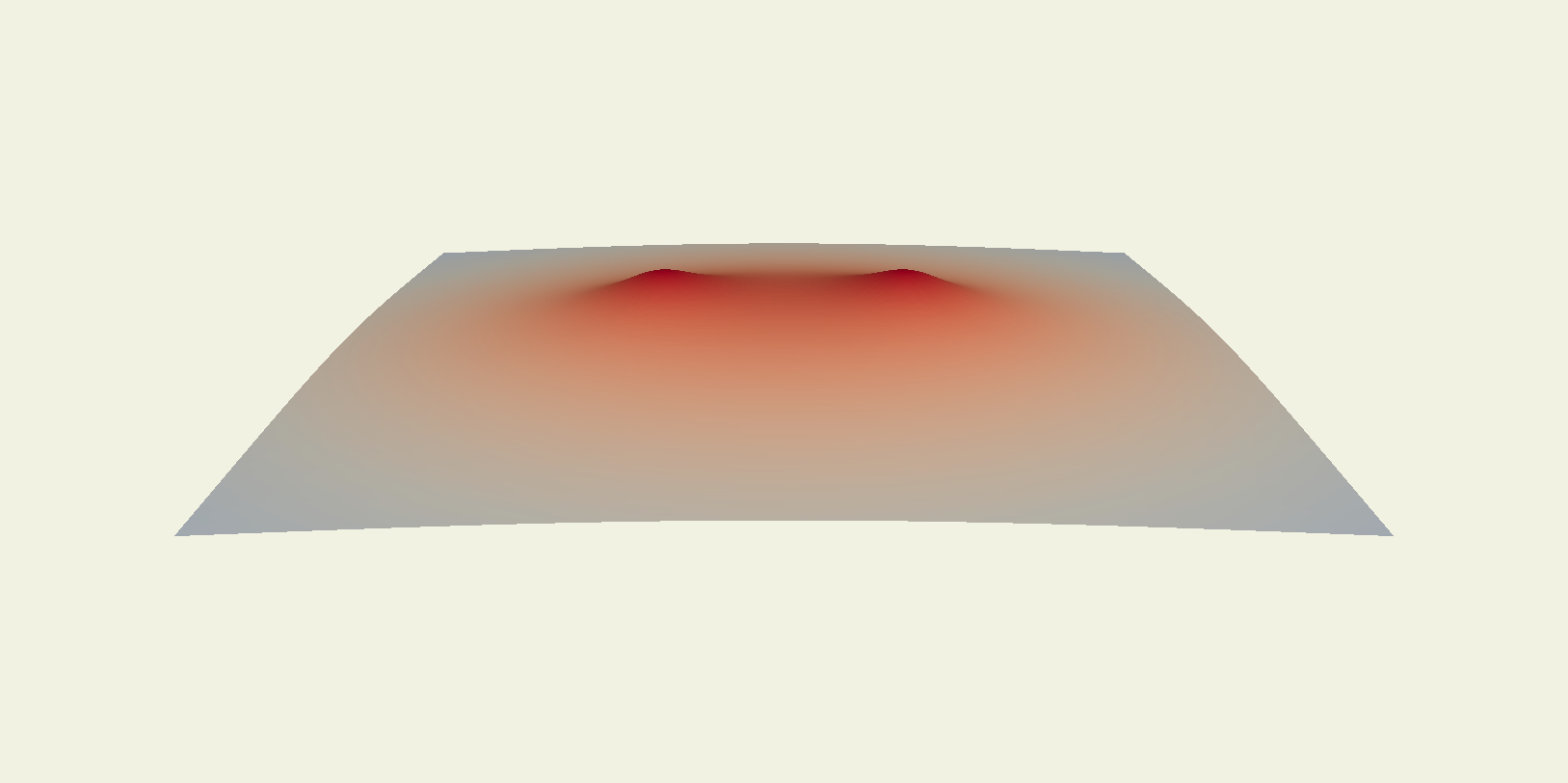}
\caption{equally oriented particles 
}
\label{point_curv_same}
\end{subfigure}
\begin{subfigure}[t]{0.45\linewidth}
\centering
\includegraphics[width=1\textwidth]{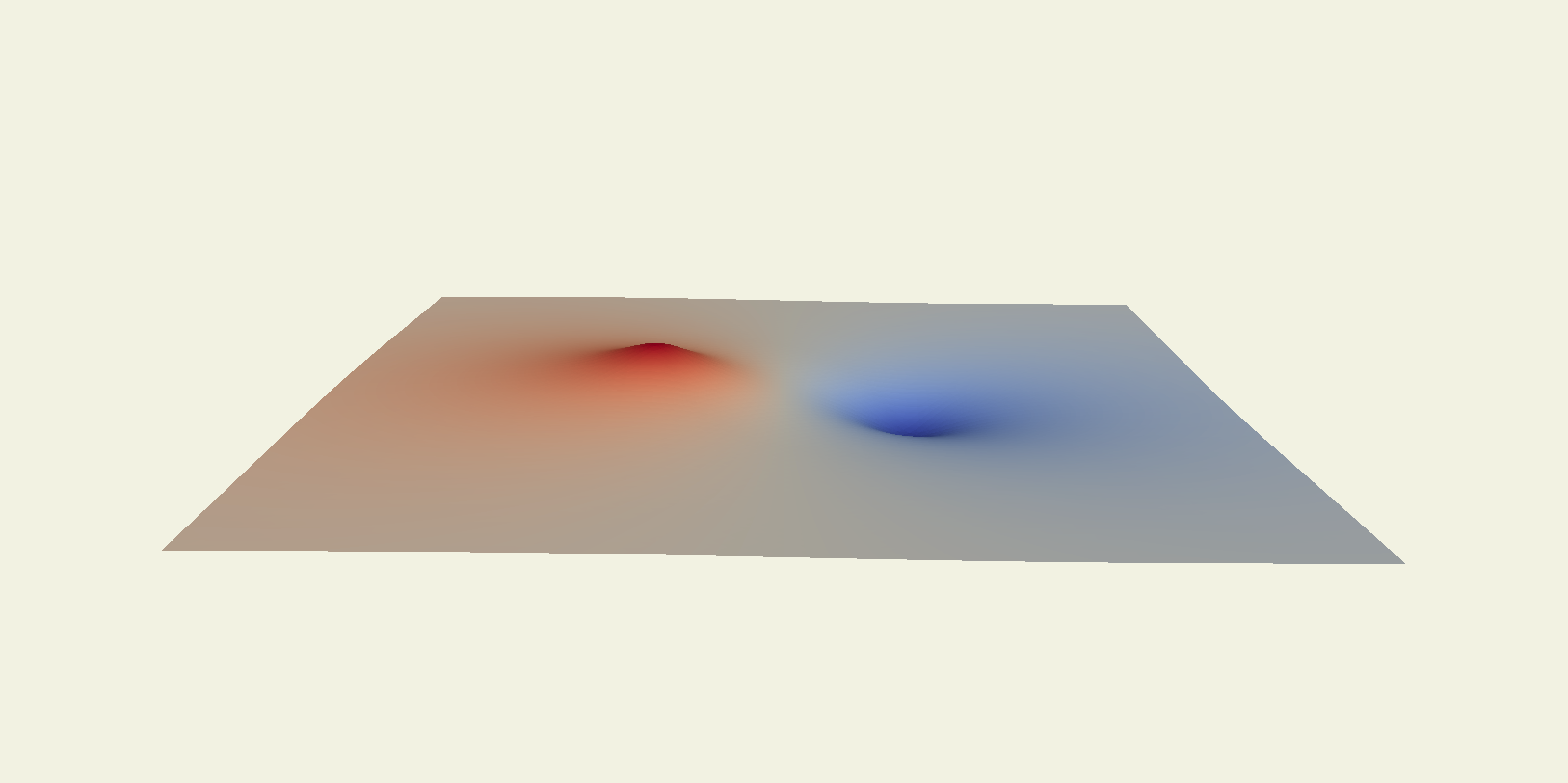}
\caption{oppositely oriented particles 
}
\label{point_curv_opp}
\end{subfigure}
\caption{Approximate membrane displacement for point mean curvature constraints.}\label{fig:POINT}
\end{figure}

\begin{figure}
\begin{subfigure}[b]{0.45\linewidth}
\centering
\includegraphics[angle=270]{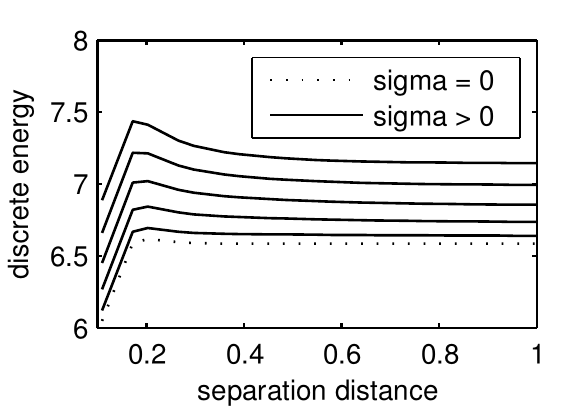}
\caption{equally oriented particles}
\label{energy_pointcurv_same}
\end{subfigure}
\begin{subfigure}[b]{0.45\linewidth}
\centering
\includegraphics[angle=270]{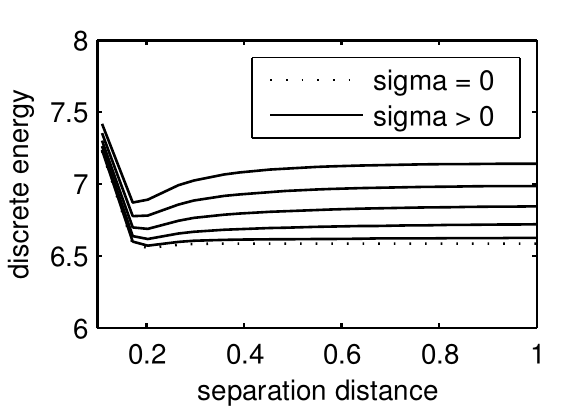}
\caption{oppositely oriented particles}
\label{energy_pointcurv_opp}
\end{subfigure}
\caption{Interaction potential for point mean curvature constraints over
 separation distance for $\sigma=0,1,4,9,16,25$ (bottom up).} \label{fig:INTERACP}
\end{figure}

\subsubsection{Discussion} 
For the point inclusion model, separation distances 
shorter than $R=0.2$ are physically impossible,
as they would represent situations where the particles with radius $r=0.1$ would overlap. 
This explains and contextualises the observed unphysical interactions for short distances $R<0.2$.
For comparison with the results obtained for the finite-sized circular particles, 
it is therefore sufficient to consider separation distances $R\geq 0.2$. 
The major difference is that the point model does not reproduce the strong repulsion 
at the limiting separation $R=0.2$. This is unsurprising as the point model 
introduces a cutoff length similar to the length the strong repulsion acts over. 
For separation distances above this cutoff i.e., for $R>0.3$, 
the observed interactions agree qualitatively and quantitatively as  depicted in
Figures~\ref{fig:INTERACP} and~\ref{fig:ENERGYCIRCLE}. 
A point inclusion counterpart to the  finite-sized model 
with elliptical particles would require
anisotropic functionals of the form \eqref{eq:CF} 
which are not considered here.

The interaction potential between two particles with circular cross-section 
has been intensively studied for more than 20 years (see, e.g., \cite{BarFou03, DomFouGal98, DomFou02,GouBruPin93,KimNeuOst98,MarMis02,
ParLub96,WeiKozHel98}). 
For finite-size particles
Weikl, Kozlov \& Helfrich~\cite{WeiKozHel98} 
found by analytical considerations
that in case of positive membrane tension, i.e., $\sigma>0$, 
the interaction depends on the relative orientation of the two particles: 
Equally oriented particles repel each other at all separation distances, 
whereas for particles with opposite orientation the interaction 
is repulsive at small and attractive at larger distances.
For $\sigma = 0$ the interaction is found to be repulsive at all distances independent of the particles' 
orientation, confirming earlier results for finite-size particles in \cite{GouBruPin93,KimNeuOst98,ParLub96}. 
Similar results have been obtained for point-like particles~\cite{BarFou03,DomFou02,MarMis02}.
These theoretical findings are in  accordance with our numerical computations.

Investigations of membrane-mediated interaction of particles 
with non-circular cross-sections are rare. 
In \cite{ParLub96} and \cite{KimNeuOst00} the effect of the particle shape on the character of interaction is studied. 
The authors of \cite{KimNeuOst00} consider a pair of identical particles 
whose elliptical cross-section is a small perturbation of the unit circle. 
They obtain that introducing the horizontal orientation of the particle as an additional degree of freedom,
 qualitatively changes the asymptotic interaction character from a repulsion to an attraction, which is 
 in agreement with a previous result in \cite{ParLub96} and with our numerical computations.

Note that well-known interaction laws of the form $1/R^4$ for circular and  $1/R^2$
for non-circular particles
are based on large distance asymptotics 
assuming an unbounded asymptotically flat membrane 
with particles separated by distances which are large 
compared to their size, i.e. $r\ll R$. 
Although our numerical experiments 
cover the complementary situation $r \approx R$,
it turns out that our numerical results  quantitatively reproduce the
asymptotical results given in \cite{WeiKozHel98}
even for small separation distances.

\subsection{Point Forces}

\subsubsection{Interaction potential for fixed locations}
We consider Problem \ref{pres_point} with point forces at fixed locations $X=(X_i)\in \overline{\Omega}$
with the solution space $V=H^2(\Omega)\cap H_0^1(\Omega)$. 
Similar to Subsection~\ref{subsec:PCCN}, our numerical approach is based on a splitting of the fourth order problem
\eqref{eq:POFOEQ} into two second order problems for the unknown functions 
$u$ and $w = \kappa\Delta u-\sigma u$.
For  fixed $p \in (2,\infty)$ and $q \in (1,2)$ such that $1/p+1/q=1$,
we consider the variational problem to find
$(u,w) \in H^1_0(\Omega) \times W^{1,q}_0(\Omega)$ such that
\begin{equation} \label{eq:SPLITPF}
\begin{array}{rl}
\displaystyle \int_\Omega \nabla w \cdot \nabla v \; dx \hspace*{-2mm}&= 
\displaystyle  -\sum_{i=1}^{N} \alpha_i \delta_{X_i}v \; dx, \\[5mm]
\displaystyle \int_\Omega \kappa \nabla u \cdot \nabla v + \sigma uv \; dx \hspace*{-2mm} 
&= \displaystyle  -\int_\Omega wv \; dx
\end{array}
\end{equation}
holds for all $v \in W^{1,p}_0(\Omega)$.
Due to the regularity of solutions to Problem \ref{pres_point}, $(u,w)$ solves \eqref{eq:SPLITPF},
 if and only if $u$ solves Problem \ref{pres_point} and $w=\kappa\Delta u -\sigma u$. For details we refer to \cite{hobbsDiss16}. We discretize the system \eqref{eq:SPLITPF} using $P^1$ finite elements.

Recall that  point forces of equal sign cluster to one point (see Corollary~\ref{cor:clustering_two_particle_problem}). 
Thus, we restrict the investigation of membrane-mediated interactions 
to the case $N=2$ and $\alpha_1<0<\alpha_2$.
We choose the same physical length scale as in Section~\ref{subsec:FBCC},
fix the same scaling of the energy by selecting $\kappa=1$,
and consider the domain $\Omega=\{x\in \R^2\st |x| < 1\}$ 
for all subsequent computations.
We also choose the forces $\alpha_1=-10$ and $\alpha_2=10$.
For a motivation of this choice, let us consider the length scale  $\kappa/|\alpha_i|$ and the typical value 
$\kappa\approx 20  \hspace*{1pt} \kappa_B T$
with the Boltzmann constant $\kappa_B$ and 
absolute temperature $T \approx  300 \hspace*{1pt} K$~\cite{EvaTurSen03}.
Then our selection corresponds to the strength $|\alpha_i| \approx 27  \hspace*{1pt} pN$. 
This is reasonable as forces applied by actin polymerisation are of the order 10pN~ \cite{AnaEhr07}.

To study the interaction potential between these two opposite forces, we fix $X_1=(0,0)$, allow $X_2$ to vary along the abscissa and compute the 
resulting approximate minimal energy $\cJ$ as 
a function of the separation distance $R$. 
This is done for a variety of values of $\sigma$
corresponding to interaction lengths $\sqrt{\kappa/\sigma}$ 
ranging from $1 \hspace*{1pt} nm$ to $30 \hspace*{1pt} nm$~\cite{EvaTurSen03}. 

\begin{figure}[ht!]
\includegraphics[angle=270]{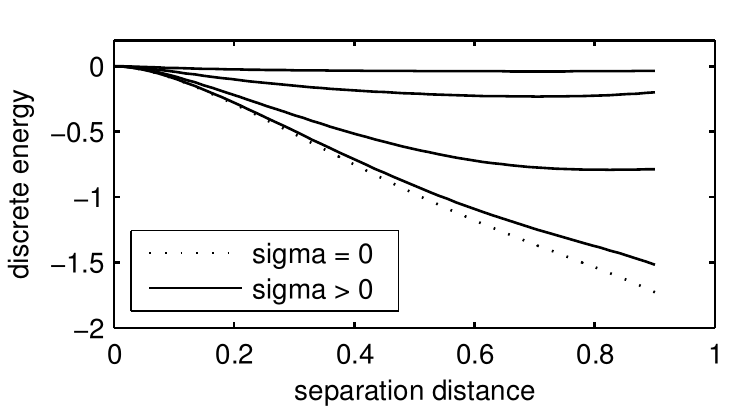}
\caption{Interaction potential for opposite point forces over separation distance for $\sigma=0,1,10,100,1000$ (bottom up).}
\label{opp_pointforce_vary_sigma}
\end{figure}  

The results depicted in Figure \ref{opp_pointforce_vary_sigma} show
that opposite forces repel each other and that this repulsion depends upon the ratio $\kappa/\sigma$. 
Increasing $\sigma$, i.e. decreasing the ratio $\kappa/\sigma$, yields a decrease in the distance 
over which the repulsive interaction plays a role. 
For $\sigma = 1$ the repulsion persists close to the boundary, 
whereas for higher values of $\sigma$ the repulsion has a shorter length scale and is, 
from a certain distance $R^*$ on,  dominated by the inwards force applied at the boundary. 
Note that this inwards force is a consequence of the (artificial) boundary condition $u=0$. 

\subsubsection{Discussion}
The above numerical findings could be related to the theoretical results 
derived in Section~\ref{sec:point_forces_vary}. 
According to Remark~\ref{rem:CLUST}, there are essentially two types of global minimizers
solving Problem \ref{point_force_glob_prob}.
A type 1 global minimizer is characterized by $X^+$, $X^-\in \Omega$  (case (i)) and  type 2 means that 
either $X^+$ or $X^-$ is located on $\partial \Omega$ (cases (ii) and (iii)).
The numerical results  shown in Figure~\ref{opp_pointforce_vary_sigma} indicate
that type 1 or type 2 minimizers occur for sufficiently small  or large ratio $\kappa/\sigma$, respectively.
\begin{figure}
\begin{subfigure}[b]{0.45\linewidth}
\centering
\includegraphics[width=1\textwidth]{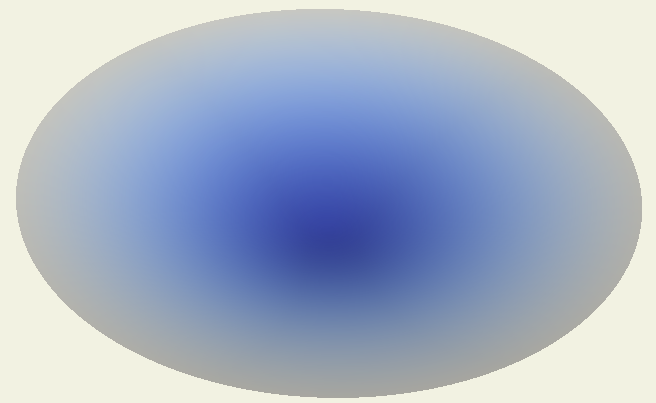}
\caption{type 1 global minimizer}
\label{point_forces_type1}
\end{subfigure}
\begin{subfigure}[b]{0.45\linewidth}
\centering
\includegraphics[width=1\textwidth]{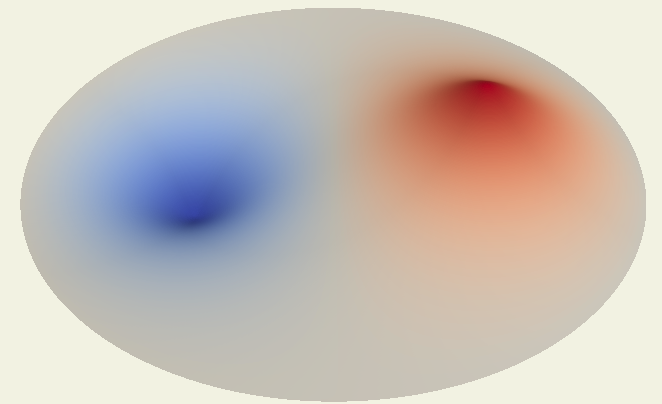}
\caption{type 2 global minimizer}
\label{point_forces_type2}
\end{subfigure}
\caption{Approximate membrane displacement for different types of global energy minimizers.}
\end{figure}
Figure \ref{point_forces_type1} shows a type 1 global minimizer solving
Problem \ref{point_force_glob_prob} for $\kappa=\sigma=1$ while Figure \ref{point_forces_type2} illustrates a type 2 global minimizer
occurring for $\kappa=1$, $\sigma=100$.
Recall that type 2 global minimizers only occur 
due to influence of the domain boundary. Thus, it is not surprising 
that there is no such result in the existing literature
that mostly concentrates on  
unbounded asymptotically flat membranes.
However, dependence of the length scale of repulsive interaction 
between particles which apply forces to the membrane
upon the ratio $\kappa/\sigma$ is well known 
and discussed, e.g.,  in \cite{EvaTurSen03}. 

%

\section{Appendix}\label{sec:APPENDIX}

\subsection{Constrained minimization}

In this section we collect some results on linearly constrained
minimization problems and their approximation by penalization.
Throughout this section, $J$ denotes the quadratic  functional
\begin{align*}
    J(v) = \textstyle \frac{1}{2} a(v,v) - \ell(v), \quad v\in V,
\end{align*} 
defined on a Hilbert space  $V$  with norm $\|\wc\|$. We assume that the bilinear form 
$a(\wc,\wc): V \times V \to \R$ is symmetric and elliptic in the sense that 
\begin{align}\label{eq:ELLAP}
    \gamma \|v\|^2 \leq a(v,v), \qquad a(v,w) \leq \Gamma \|v\| \|w\| \qquad \forall v,w \in V,
\end{align}
holds with positive constants $\gamma, \Gamma \in \R$, and $\ell\in V'$.
We consider  minimization of $J$ on 
the affine subspace $V_b \subset V$, given by
\begin{align} \label{eq:CONSSP}
    V_b = \{v \in V \st Tv \in X_0 + b\}.
\end{align}
Here, $X_0\subset X$ is a complete subspace of a pre-Hilbert space
$X$ with scalar product $(\wc,\wc)_X$
and norm $\|\wc\|_X$, $T:V \to X$ is a bounded linear operator, 
and $b \in X$.  Note that
\begin{equation} \label{eq:PVB}
    V_b = \{v \in V \st P(Tv-b) = 0\},
\end{equation}
where $P:X \to X$ is the orthogonal projection
with $\op{ker} P = X_0$.

\begin{problem}[Constrained minimization]\ \\ \label{eq:A_CONSTR_MIN}%
    Find $u\in V_b$ minimizing the energy $J$ on $V_b$.
\end{problem}
It is well-known that Problem~\ref{eq:A_CONSTR_MIN} is equivalent to  find $u\in V_b$ such that
\begin{align} \label{eq:APVAR}
     a(u,v) = \ell(v) \qquad \forall v \in V_0,
\end{align}
where $V_0 = \{v \in V \st T(v) \in X_0\}$ is obtained for $b=0$.
As  $T$ is continuous and  $X_0$ is closed, the affine subspace $V_b$
is also closed. Hence,  the Lax-Milgram lemma provides the following existence and uniqueness result.
\begin{proposition}\label{prop:A_CONSTR_MIN}
Problem~\ref{eq:A_CONSTR_MIN} has a unique solution $u\in V_b$, if and only if
$V_b \neq \emptyset$.
\end{proposition}

In order to derive a penalty approximation of the constraints $Tv \in X_0+b$ occurring in 
\eqref{eq:CONSSP}, we consider the following  minimization problem.
\begin{problem}[Parametrized minimization problem]\ \\ \label{eq:A_CONSTR_MIN_PARA}%
Find $(u,x)\in V\times X_0$ minimizing $J$ on $\{ (v,y)\in V\times X_0\st Tv=y+b\}$.
\end{problem}
\begin{lemma}\label{lem:A_CONST_EQUIV}
  The minimization Problems  \ref{eq:A_CONSTR_MIN} and~\ref{eq:A_CONSTR_MIN_PARA}
    are equivalent in the sense that
    $(u,x)$ solves Problem~\ref{eq:A_CONSTR_MIN_PARA},
    if and only if $u$ solves Problem~\ref{eq:A_CONSTR_MIN} and $x=Tu-b$.
\end{lemma} 
\begin{proof}
   As $v \in V_b$  $\Leftrightarrow$ $Tv=y+b$ with some $y\in X_0$,
   the sets of admissible $v$ are the  same for both minimization problems.
\end{proof} 
A penalty approximation of  Problem~\ref{eq:A_CONSTR_MIN} now reads  as follows.
\begin{problem}[Penalized parametrized minimization problem]\ \\ \label{eq:A_PENALIZED_MIN_PARA}%
Find $ (u_{\varepsilon},x_\varepsilon)\in V\times X_0$ minimizing the energy
\[
    J(u_\varepsilon) + \tfrac{1}{2\varepsilon} \|T u_\varepsilon -(x_\varepsilon+b)\|_X^2.
\]
on $V\times X_0$ with a given penalty parameter $\varepsilon >0$.  
\end{problem}
The associated variational formulation of
Problem~\ref{eq:A_PENALIZED_MIN_PARA} is given by
\begin{align}\label{eq:A_PENALIZED_VAR_PARA1}
    a(u_\varepsilon,v)
    + \tfrac{1}{\varepsilon} (Tu_\varepsilon,Tv)_X - \tfrac{1}{\varepsilon}(x_\varepsilon,Tv)_X
    &= \tfrac{1}{\varepsilon}(b,Tv)_X + \ell(v) &&\forall v \in V,\\ \label{eq:A_PENALIZED_VAR_PARA2}
    - \tfrac{1}{\varepsilon}(Tu_\varepsilon,y)_X + \tfrac{1}{\varepsilon} (x_\varepsilon,y)_X
    &= -\tfrac{1}{\varepsilon}(b,y)_X &&\forall y \in X_0.
\end{align}
The orthogonal projection  $P: X \to X$ with  $\op{ker} P = X_0$ 
was used to characterize the solution space $V_b$ according to \eqref{eq:PVB}.
Now we can use the projection $P$ to characterize
the solution of Problem~\ref{eq:A_PENALIZED_MIN_PARA} as follows.

\begin{proposition}\label{lem:A_PENALIZED_MIN_EQIV}
    The pair $(u_\varepsilon,x_\varepsilon)$ solves  Problem~\ref{eq:A_PENALIZED_MIN_PARA},
    if and only if $u_\varepsilon \in V$ is  the unique mini\-mizer of the energy
    \[
        J(u_\varepsilon) + \tfrac{1}{2\varepsilon} \|P(T u_\varepsilon -b)\|_X^2
    \]
    on $V$ or, equivalently,
    \begin{align} \label{eq:A_PENALIZED_VAR_PROJ}
        a(u_\varepsilon,v) + \tfrac{1}{\varepsilon} (PTu_\varepsilon,PTv)_X &= \tfrac{1}{\varepsilon}(Pb,PTv)_X + \ell(v)
        &&\forall v \in V
    \end{align}
    and  $x_\varepsilon = (I-P)(Tu_\varepsilon-b)$.
\end{proposition} 
\begin{proof}
    Let $(u_\varepsilon,x_\varepsilon)$ be a solution to Problem~\ref{eq:A_PENALIZED_MIN_PARA}.
    Then~\eqref{eq:A_PENALIZED_VAR_PARA2} implies
    \begin{align} \label{eq:JKL}
        \tfrac{1}{\varepsilon} (x_\varepsilon,y)_X
        = \tfrac{1}{\varepsilon}((I-P)(Tu_\varepsilon-b),y)_X
        \qquad \forall y \in X_0,
    \end{align}
    which, by $(I-P)(Tu_\varepsilon-b) \in X_0$, yields $x_\varepsilon = (I-P)(Tu_\varepsilon-b)$.
    Let $v\in V$. Inserting $y = (I-P)Tv \in X_0$ into \eqref{eq:JKL} we get
    \begin{align}\label{eq:A_PENALIZED_MIN_EQIV_1}
        \tfrac{1}{\varepsilon} (x_\varepsilon,Tv)_X
        = \tfrac{1}{\varepsilon}((I-P)(Tu_\varepsilon-b),Tv)_X.
    \end{align}
    Adding this equation  to~\eqref{eq:A_PENALIZED_VAR_PARA1} provides \eqref{eq:A_PENALIZED_VAR_PROJ}.

    Conversely, exploiting continuity of $T$ and $P$ there is a unique solution 
    $u_\varepsilon$  of~\eqref{eq:A_PENALIZED_VAR_PROJ} by the Lax-Milgram lemma. 
    Let $x_\varepsilon = (I-P)(Tu_\varepsilon-b)$.
    Applying $\tfrac{1}{\varepsilon}(\wc,y)_X$ for arbitrary $y \in X_0$
    to $x_\varepsilon = (I-P)(Tu_\varepsilon-b)$
    gives~\eqref{eq:A_PENALIZED_VAR_PARA2}, which, for $y=(I-P)Tv$,
    implies~\eqref{eq:A_PENALIZED_MIN_EQIV_1}.
    Subtracting this  equation from~\eqref{eq:A_PENALIZED_VAR_PROJ} we get~\eqref{eq:A_PENALIZED_VAR_PARA1}.
\end{proof} 

Now we show convergence for $\varepsilon \to 0$.

\begin{proposition}\label{prop:A_PENALTY_CONV}
    Assume that $V_b \neq \emptyset$ and let $u$, $u_\varepsilon$ denote the solutions of
    Problem~\ref{eq:A_CONSTR_MIN} and \ref{eq:A_PENALIZED_MIN_PARA}, respectively.
    Then $(u_\varepsilon,x_\varepsilon) \to (u,x)$  converges in $V\times X$ as $\varepsilon \to 0$.
    Moreover, $\|PT (u_\varepsilon - u)\|_X^2 \leq  C \varepsilon$  holds with a positive constant $C$,
    depending only on  $\gamma$, $\Gamma$, and $\ell$.
\end{proposition} 
\begin{proof}  Proposition~\ref{lem:A_PENALIZED_MIN_EQIV}
implies that $u_\varepsilon$ solves \eqref{eq:A_PENALIZED_VAR_PROJ}.
    Testing \eqref{eq:A_PENALIZED_VAR_PROJ} with $v=u_\varepsilon-u$,
    adding $a(u,u-u_\varepsilon)$,  using $PTu=Pb$ and coercivity of $a(\wc,\wc)$, we obtain
    \begin{align}\label{eq:A_LEM_PENALIZED_CONV_ESTIMATE}
        \begin{split}
            \gamma\|u_{\varepsilon}-u\|^2 &\leq
            a(u_\varepsilon-u,u_\varepsilon-u) + \tfrac{1}{\varepsilon} \|PT(u_\varepsilon-u)\|_X^2 \\
            &= a(u,u-u_\varepsilon)- \ell(u-u_\varepsilon)
            \leq c\|u_\varepsilon-u\|
        \end{split}
    \end{align}
    with $c=\Gamma \|u\| + \| \ell \|\leq (\Gamma/\gamma + 1)\| \ell \|$.  
    Thus, $\|u_\varepsilon-u\| \leq \tfrac{c}{\gamma}$ 
    and  $\|PT(u_\varepsilon-u)\|_X^2 \leq C \varepsilon$ holds with $C=\tfrac{c^2}{\gamma}$.
    As a consequence, $PTu_\varepsilon \to PTu$ converges in $X$ as $\varepsilon \to 0$.
    
    Now let $\varepsilon_n\to 0$. Then, boundedness of $\|u_{\varepsilon_n}\|$ implies existence
    of  $\bar{u}\in V$ and of a  subsequence $u_{\varepsilon_{n_k}}\rightharpoonup \bar{u}$  converging weakly
    in $V$.
    Since $PT$ is linear and continuous,  $PT$ is also weakly continuous.
    Hence,  strong convergence $PTu_\varepsilon \to PTu$  in $X$ implies $PT\bar{u}=PTu$
    and thus $u-\bar{u} \in V_0$.
    Therefore, exploiting $u_{\varepsilon_{n_k}} \rightharpoonup \bar{u}$ and~\eqref{eq:A_LEM_PENALIZED_CONV_ESTIMATE}
    we obtain
    \begin{align*}
        \gamma \|u_{\varepsilon_{n_k}}-u\|^2
        \leq a(u,u-u_{\varepsilon_{n_k}}) - \ell(u-u_{\varepsilon_{n_k}})
        \rightarrow a(u,u-\bar{u}) - \ell(u-\bar{u})=0.
    \end{align*} 
    for $k\to \infty$. Since this holds for any weakly convergent subsequence, 
    we have shown $u_\varepsilon \to u$ in $V$.
    Now continuity of $P$ and $T$ provides $x_\varepsilon \to x$ in $X$.
\end{proof}

Note, that constraints of the form $Tu=b$ are covered by the choice  $X_0=\{0\}$.

\subsection{Global minimizers}
We now consider a generalization of Problem~\ref{prob:APPFINITEDIM} 
by parametrizing the functional $T$ over a suitable set $\cM$.
More precisely, for a given metric space $\cM$,
and a given mapping $T:\cM \to L(V,X)$, with $L(V,X)$ denoting
the space of all bounded linear mappings from $X$ to $V$,
we define the admissible set 
\begin{align}\label{eq:A_PARAM_SOL_MANIFOLD}
    \cW &= \{(v,y) \in V \times \cM \st v \in V_{b,y}\}, &
    V_{b,y}=\{v\in V \st  P(T(y)v-b)=0 \}
\end{align}
with $P: X \to X$ denoting the orthogonal projection with
$\op{ker} P = X_0$ as in the preceding section.
Furthermore we allow for an additional term $\cV:\cM \to \Rinfty$ to
model certain preferences of $y \in \cM$.
\begin{problem}[Global minimization]\ \\ \label{QPP_Global}%
    Find $(u,x) \in \cW$ minimizing the energy
    $J(u) + \cV(x)$    on $\cW$.
\end{problem}

By definition of $\cW$, any solution $(u,x) \in \cW$ to Problem~\ref{QPP_Global} satisfies
\begin{align*}
    u &= \argmin{v \in V_{b,x}} J(v), &
    x &= \argmin{\stackrel{y \in \cM}{V_{b,y}\neq \emptyset}}
        \Bigl(\Bigl( \min_{v \in V_{b,y}} J(v) \Bigr) + \cV(y)\Bigr).
\end{align*}
Hence Problem~\ref{QPP_Global} is equivalent to minimizing
$\xi + \cV$ over $\cM$, where the functional
$\xi:\cM\to \Rinfty$ is given by
\begin{equation}\label{eq:XIDEF}
    \xi(y)=
        \begin{cases}
            \min_{v\in V_{b,y}} J(v) &\text{if }V_{b,y}\neq \emptyset,\\
            \infty & \text{else}.
        \end{cases}
\end{equation}
Note that $\xi$ is well-defined by Proposition~\ref{prop:A_CONSTR_MIN}.

\begin{proposition}\label{QPP_Global_exist}
    Assume that $\cM$ is compact, $T:\cM \to L(V,X)$ is continuous,
    $\cV:\cM \to \Rinfty$ is lower semi-continuous,
    and that there is $x_0 \in \cM$ with $\cV(x_0) < \infty$ and $V_{b,x_0} \neq \emptyset$.
    Then Problem~\ref{QPP_Global} has a solution.
\end{proposition}
\begin{proof}
    We will show that $\xi$ is lower semi-continuous.
    Then $\xi + \cV: \cM \to \Rinfty$ is also lower semi-continuous on the
    compact set $\cM$ which provides existence of a minimizer $x \in \cM$
    \cite[(12.7.9)]{Dieudonne1970TreatiesAnalysis2}.

    Let $(x_n) \in \cM$ with $\lim_{n \to \infty} x_n = \bar{x} \in \cM$.
    We have to show
    \[
        m = \liminf_{n \to \infty} \xi(x_n) \geq \xi(\bar{x}).
    \]
    Since this is trivial for $m = \infty$ we now assume $m<\infty$.
    Coercivity of $J$ implies $m> - \infty$.
    Furthermore, there is a subsequence, still denoted by  $(x_n)$, such that $\lim_{n \to \infty} \xi(x_n) = m$
    and $\xi(x_n)< \infty$ for all $n \in \N$.
    As a consequence of  $\xi(x_n)< \infty$, we have $V_{b,x_n} \neq \emptyset$ so that 
    \begin{align*}
        u_n = \argmin{v \in V_{b,x_n}}J(v)
    \end{align*}
    is well-defined. Boundedness of $J(u_n) = \xi(x_n)<\infty$ and coercivity of $J$ 
    imply that $u_n$ is bounded.
    Hence there is a $\bar{u} \in V$ and another subsequence, still denoted by $(u_n)$,
    such that
    \begin{align*}
        u_n\rightharpoonup \bar{u} \quad \text{for}\quad n\to \infty.
    \end{align*}
    Now continuity $T(x_n)\to T(\bar{x})$ yields $PT(x_n)u_n\rightharpoonup PT(\bar{x})\bar{u}$.
    Hence,  $\bar{u}\in V_{b,\bar{x}}$ and $(\bar{u}, \bar{x}) \in \cW$.
    As $J$ is convex and continuous, $J$ must be weakly lower semi-continuous
    (see, e.g.~\cite{EkeTem99}[Chapter~I, Corollary~2.2]) which gives
    \[
        \xi(\bar{x}) \leq J(\bar{u})
            \leq \liminf_{n \to \infty} J(u_n)
            = \lim_{n \to \infty} \xi(x_n) = m.
    \]
    This concludes the proof.
\end{proof}

In order to show existence of minimizers for a penalized
version of Problem~\ref{QPP_Global} we will use a general
result on parametrized problems without constraints.
In the following we will identify continuous bilinear forms
$a_x(\wc,\wc):V \times V \to \R$ with their operator
representations $a_x : L(V,V')$.

\begin{lemma}\label{lem:A_parametrized_global_unconstrained}
    Assume that $\cM$ is compact, that the functions
    \begin{align*}
        \cM \ni y &\mapsto a_y(\wc,\wc) \in L(V,V'), &
        \cM \ni y &\mapsto \ell_y(\wc) \in V'
    \end{align*}
    are continuous, that $a_y(\wc,\wc)$ is symmetric 
    and uniformly coercive for all $y \in \cM$,
    that $\cV:\cM \to \Rinfty$ is lower semi-continuous,
    and that there is $x_0 \in \cM$ with $\cV(x_0) < \infty$.
    Then there is $(u,x) \in V \times \cM$ minimizing $J_x(u)+ \cV(x)$ with
    \begin{align}\label{eq:A_parametrized_functional}
        J_x(u)= \tfrac{1}{2} a_x(u,u) - \ell_x(u) 
    \end{align}
    on $V \times \cM$.
\end{lemma}
\begin{proof}
    For any $y \in \cM$ we define $u_y = \inlineargmin{v \in V} J_y(v)$.
    We will show that the function $\eta : \cM \to \R$ defined by
    \begin{align*}
        \eta(y) = \min_{v \in V} J_y(v) = J_y(u_y)
    \end{align*}
    is continuous. Then $\eta + \cV: \cM \to \Rinfty$ is lower semi-continuous on the
    compact set $\cM$  which provides existence of a minimizer $x \in \cM$ of $\eta + \cV$ on $\cM$
    \cite[(12.7.9)]{Dieudonne1970TreatiesAnalysis2} and thus of a minimizer $(u,x)$ of $J_x(u) + \cV(x)$
    on $V \times \cM$ with $u = \inlineargmin{v \in V} J_x(v)$.
    
    First we note that $(v,y) \mapsto J_y(v)$ is continuous on $V\times \cM$, because
    $y \mapsto a_y(\wc,\wc)\in L(V,V')$ and $y \mapsto \ell_y(\wc)\in V'$ are continuous.
    Now let $(y_n) \in \cM$ with $\lim_{n \to \infty}y_n = y \in \cM$.
    Then the first lemma of Strang~\cite[Theorem 4.1.1]{Cia78} implies
    \begin{align*}
        \|u_y - u_{y_n}\|
            &\leq (1 + \|a_y\|) \Bigl(
                \sup_{w \in V} \frac{|a_y(u_y,w)-a_{y_n}(u_y,w)|}{\|w\|}
                + \sup_{w \in V} \frac{|\ell_y(w)-\ell_{y_n}(w)|}{\|w\|}
                \Bigr) \\
                &\leq (1 + \|a_y\|) \bigl(
                    \|a_y-a_{y_n}\| \|u_y\| + \|\ell_y-\ell_{y_n}\|\bigr)
                    \underset{n \to \infty} \longrightarrow 0.
    \end{align*}
    Hence, $u_{y_n} \to u_y$ in $V$ for $n\to \infty$. 
    Continuity of  $(v,y) \mapsto J_y(v)$ on $V\times \cM$ yields 
    \[
        \lim_{n \to \infty} \eta(y_n)=   \lim_{n \to \infty} J_{y_n}(u_{y_n})  = J_{y}(u_{y})=\eta(y).
    \]
\end{proof}

After these preparations, we  consider  a  penalized version of Problem~\ref{QPP_Global}.
\begin{problem}[Penalized global minimization]\ \\ \label{QPP_Global_penalized}%
    Find $(u_\varepsilon,x_\varepsilon) \in V \times \cM$ minimizing the energy
    \begin{align*}
        J(u_\varepsilon) + \tfrac{1}{2\varepsilon} \|P(T(x_\varepsilon) u_\varepsilon -b)\|_X^2 + \cV(x_\varepsilon)
    \end{align*}
    on $V\times \cM$ with a given penalty parameter $\varepsilon >0$.
\end{problem}
\begin{proposition}\label{QPP_Global_penalized_exist}
    Assume that $\cM$ is compact, $T:\cM \to L(V,X)$ is continuous,
    $\cV:\cM \to \Rinfty$ is lower semi-continuous,
    and that there is $x_0 \in \cM$ with $\cV(x_0) < \infty$.
    Then Problem~\ref{QPP_Global_penalized} has a solution.
\end{proposition}
\begin{proof}
    For $(v,y) \in V \times$ the energy functional
    \begin{align*}
        J_y(v) = J(v) + \tfrac{1}{2\varepsilon} \|P(T(y) v -b)\|_X^2
    \end{align*}
    takes the form~\eqref{eq:A_parametrized_functional} with
    \begin{align*}
        a_y(w,v) &= a(w,v) + \tfrac{1}{\varepsilon} (PT(y) w,PT(y) v)_X, &
        \ell_y(v)&= \ell(v) + \tfrac{1}{\varepsilon}(Pb,PT(y)v)_X.
    \end{align*} 
    Now continuity of $T$ implies continuity of $y \mapsto (PT(y))^* PT(y) \in L(V,V')$
    and $y \mapsto (PT(y))^* Pb \in V'$ and thus of $y \mapsto a_y(\wc,\wc)$ and
    $y \mapsto \ell_y(\wc)$.
    Furthermore $a(v,v,) \leq a_y(v,v)$ implies uniform coercivity of $a_y(\wc,\wc)$
    with respect to $y \in \cM$.
    Hence Lemma~\ref{lem:A_parametrized_global_unconstrained} provides the assertion.
\end{proof}

To conclude this subsection we consider a minimization problem
with parametrized source term $\cM\ni y \to \ell_y(\wc)\in V'$.
\begin{problem}[Parametrized  source term]\ \\ \label{prob:PARSOU}%
Find $(u,x)\in V\times \cM$ minimizing the energy
$
J_x(u)={\textstyle \frac{1}{2}} a(u,u)-\ell_x(u)
$
on $ V\times \cM$.
\end{problem}

Existence is an immediate consequence of 
Lemma~\ref{lem:A_parametrized_global_unconstrained} with 
$a_y(\wc,\wc)=a(\wc,\wc)$ and $\cV=0$.

\begin{proposition} \label{pro:PARSOU}
Assume that $\cM$ is compact and that $\cM \ni y \to \ell_y (\wc)\in V'$ is continuous.
Then there is a solution $(u,x)\in V\times \cM$ of Problem~\ref{prob:PARSOU}.
\end{proposition}

\subsection{Equivalent characterization of minimizers}
We first consider Problem~\ref{eq:A_CONSTR_MIN} in the special case
\begin{equation} \label{eq:APPSC}
    X=\R^N, \quad X_0=\{0\}, \quad b=(b_i)\in \R^N, \quad T=(T_i)\in (V')^N.
\end{equation}
\begin{problem}[Finite dimensionally constrained minimization]\ \\ \label{prob:APPFINITEDIM}%
    Find $u\in V$ minimizing the energy $J$ subject to the constraints 
    \begin{equation} \label{eq:APFC}
        T_i u=b_i, \qquad i=1,\dots, N.
    \end{equation}
\end{problem}
While Proposition~\ref{prop:A_CONSTR_MIN} provides existence 
and uniqueness of solution to Problem~\ref{prob:APPFINITEDIM}, 
if and only if $V_b\neq \emptyset$, we will now derive an equivalent 
characterization of the solution in terms of suitable basis functions. 
Utilizing the Lax-Milgram lemma, 
we define  $\phi_0 \in V$ to be the unique solution to
\[
a(\phi_0,v) = \ell(v) \qquad \forall \, v \in V.
\]
We also define $\phi_i\in V$, $i= 0,\dots, N$, such that
\begin{equation} \label{eq:VIEHDEF}
    a(\phi_i,v) = T_iv \qquad  \forall  v \in V.
\end{equation}
and the associated Gramian matrix  $A = (A_{ij})\in \R^{N\times N}$ 
with $A_{ij} = a(\phi_i,\phi_j)$.

\begin{proposition} \label{quad_exist_NEU}
    Assume that Problem~\ref{prob:APPFINITEDIM} is non-degenerate in the sense 
    that the functionals $T_i$, $i=1,\dots, N$, are linearly independent.
    Then there is a unique solution of Problem~\ref{prob:APPFINITEDIM} given by
    \begin{equation}\label{eq:APPREP}
        u = \phi_0 + \sum_{i=1}^N U_i \phi_i
    \end{equation}
    with $U = A^{-1}(b - T\phi_0)$.
\end{proposition}
\begin{proof}
    Linear independence of the functionals $T_i$, $i=1,\dots, N$,
    implies that $T$ is surjective.  
    Hence, $V_b \neq \emptyset$ and  Problem~\ref{prob:APPFINITEDIM} has a unique solution
    by Proposition~\ref{prop:A_CONSTR_MIN}.
    Furthermore the functions $\phi_i$ defined in \eqref{eq:VIEHDEF}
    are linearly independent and the associated Gramian matrix $A$ is regular.

    Let $u$ be given by \eqref{eq:APPREP}. Then 
    $u\in V_b$, as for any $i=1, \dots, N$ we have
    \[
        T_iu = T_i\phi_0 + \sum_{j=1}^N U_j T_i\phi_j 
        = T_i\phi_0+\sum_{j=1}^N U_j A_{ij}  = b_i.
    \]
    In addition, for all $v\in V_0$, we have
    \[ 
        a(u,v) = a(\phi_0,v) + \sum_{i=1}^N U_i a(\phi_i,v)  = 
        \ell(v)+\sum_{i=1}^N U_i T_iv  = \ell(v).
    \]
    Hence, $u$ is the unique solution of Problem~\ref{prob:APPFINITEDIM}.
  \end{proof}

We now consider Problem~\ref{QPP_Global} in the  related  special case 
\begin{equation} \label{eq:APPSCMC}
    X=\R^N,  \;\; X_0=\{0\}, \;\; b=(b_i)\in \R^N, \;\;  \cM\ni x\to T(x)=(T_i(x))\in (V')^N.
\end{equation}
Recall that $\cW$ is defined in \eqref{eq:A_PARAM_SOL_MANIFOLD} and
$V_{b,y}=\{v\in V \st  T_i(y)v-b_i=0, \; i=1,\dots, N \}$.

\begin{problem}[Global finite dimensionally constrained minimization]\ \\ \label{prob:GLOBFINITEDIM}%
    Find $(u,x)\in \cW$ minimizing the energy $J(u) + \cV(x)$ on $\cW$.
\end{problem}

While existence of solutions is considered in Proposition~\ref{QPP_Global_exist}, 
we now derive an equivalent characterization in terms of suitable basis functions. 
For all $y\in \cM$, we introduce  $y$-dependent counterparts $\phi_i(y)\in V$ of $\phi_i$ 
by replacing $T_i$ by $T_i(y)$ in \eqref{eq:VIEHDEF}. The associated Gramian matrix is
denoted by $ A(y)$.
We also define 
\[\cM'=\{y\in \cM\st T_i(y), \; i=1,\dots, N, \text{ are linearly independent}\}.
\]
\begin{lemma}\label{lem:min_con}
    For $y \in \cM'$ the functional $\xi$ defined in~\eqref{eq:XIDEF}
    can be written as
    \begin{align} \label{eq:XIREP}
        \xi(y) = \tfrac{1}{2}(b-T(y)\phi_0)^\top A(y)^{-1} (b-T(y)\phi_0) - \tfrac{1}{2}a(\phi_0,\phi_0).
    \end{align}
 \end{lemma}
\begin{proof}
    First note that the definition of $\phi_0$ implies
    \begin{align}\label{eq:min_con_proof_orthogonality}
        J(v+\phi_0) =
            \tfrac{1}{2}a(v,v) - \tfrac{1}{2}a(\phi_0,\phi_0)
    \end{align} 
    for any $v \in V$.
    Now let $y \in \cM'$.
    Then,  $V_{b,y} \neq \emptyset$, because $T(y)$ is surjective.
    Hence, by Proposition~\ref{prop:A_CONSTR_MIN}, we can define
    $u(y) = \argmin{v \in V_{b,y}} J(v)$.    Inserting the representation
    \begin{align*}
        u(y) = \phi_0 + \sum_{i=1}^N U(y)_i \phi_i(y), \qquad U(y) = A(y)^{-1}(b - T(y)\phi_0)
    \end{align*}
    as obtained from Proposition~\ref{quad_exist_NEU}, 
    we exploit~\eqref{eq:min_con_proof_orthogonality},
    the definition of $A(y)$ and $U(y)$, to obtain
    \begin{align*}
        J(u(y)) &= \tfrac{1}{2} U(y)^\top  A(y) U(y)
        - \tfrac{1}{2}a(\phi_0, \phi_0) \\
        &= 
        \tfrac{1}{2}(b-T(y)\phi_0)^\top A(y)^{-1}(b-T(y)\phi_0) - \tfrac{1}{2}a(\phi_0,\phi_0).
    \end{align*}
\end{proof} 

Now the following characterization of solutions to Problem~\ref{prob:GLOBFINITEDIM} 
is an immediate consequence of
Lemma~\ref{lem:min_con} and Proposition~\ref{quad_exist_NEU}.

\begin{proposition}\label{min_con}
    Assume that Problem~\ref{prob:GLOBFINITEDIM}  is non-degenerate
    in the sense that all solutions $(u,x)$ satisfy $x \in\cM'$.
    Then $(u,x)\in V\times \cM$ solves Problem~\ref{QPP_Global}, if and only if $x\in \cM'$ is a minimizer of 
    $\xi+ \cV$ on $\cM'$ with $\xi$ given in \eqref{eq:XIREP} and 
    \begin{equation}\label{eq:APPREPMC}
        u = \phi_0 + \sum_{i=1}^N U_i(x) \phi_i(x)
    \end{equation}
    holds with $U(x) = A(x)^{-1}(b-T(x)\phi_0)$.
\end{proposition}

\subsection{Regularity of Green's functions}

We provide the following 
regularity result for Green's functions of eighth order problems on unbounded domains
as exploited in Section~\ref{subsec:UNBD}.
\begin{lemma} \label{BF_4o_reg}
    Let $X \in \mathbb{R}^2$, $a_0, \dots,a_4 \geq 0$ with $a_0$, $a_4>0$, 
    and
    \[
        a(u,v) = \int_{\R^2} a_4 \Delta^2 u \Delta^2 v + a_3 \nabla\Delta u \cdot \nabla\Delta v + a_2 \Delta u \Delta v + a_1 \nabla u \cdot \nabla v + a_0 uv\;dx
    \]
    for $u$, $v\in H^4(\mathbb{R}^2)$.
    Then the corresponding Green's function $u \in H^4(\mathbb{R}^2)$ 
    characterized by the variational equality
    \[
        a(u,v)=v(X) \qquad \forall v \in H^4(\mathbb{R}^2)
    \]
    satisfies $u \in H^s(\mathbb{R}^2)$ for any $s \in (0,7)$.
\end{lemma}
\begin{proof}
    Without loss of generality take $X=0$ as Green's functions for
    $X \neq 0$ will be translations of this case. 
    Let $\mathcal{F}:L^2(\mathbb{R}^2) \rightarrow L^2(\mathbb{R}^2)$
    denote the Fourier transform which is the continuous extension of $\mathcal{F}[\varphi]$, 
    $\varphi \in C^\infty_0(\mathbb{R}^2)$, defined by
    \[
        \mathcal{F}[\varphi](\xi)= \int_{\mathbb{R}^2} \varphi(x) e^{-2\pi i \xi \cdot x} \; dx, \qquad  \xi \in \R^2 .
    \]
    Let 
    \[
        f(\xi)= \left[a_4 (4\pi^2|\xi|^2)^4 + a_3 (4\pi^2|\xi|^2)^3 + a_2 (4\pi^2|\xi|^2)^2 + a_1 (4\pi^2|\xi|^2) +  a_0 \right]^{-1},\quad  \xi \in \R^2 .
    \] 
    and observe that $f \in L^2(\mathbb{R}^2)$, because $a_0$, $a_4>0$.
    We can thus define $g=\mathcal{F}^{-1}[f]\in L^2(\mathbb{R}^2)$. 
    Note that
    \[
        (1+r^2)^s|f((r,0))|^2r\leq Cr^{2s-15}
    \]
    holds for suitable $C>0$ and sufficiently large $r$ so that 
    \[
        \int_{\mathbb{R}^2} (1+|\xi|^2)^s|\mathcal{F}[g](\xi)|^2 \;d\xi = 2\pi\int_0^\infty (1 + r^2)^s |f((r,0))|^2 r \;dr < \infty
    \]
    holds for $2s-15 < -1$ or, equivalently, $s<7$.
    Hence $g \in H^s(\mathbb{R}^2)$ for $s <7$ and, in particular, $g \in H^4(\mathbb{R}^2)$.
    Now, for arbitrary $\varphi \in C^\infty_0(\mathbb{R}^2)$ Parseval's formula provides 
    {\small %
    \begin{align*}
            a(g,\varphi)
                &=\int_{\mathbb{R}^2} \biggl(a_4 (4\pi^2|\xi|^2)^4 + a_3 (4\pi^2|\xi|^2)^3 + a_2 (4\pi^2|\xi|^2)^2 + a_1 (4\pi^2|\xi|^2) +  a_0 \biggr)
                    \mathcal{F}[g]\mathcal{F}[\varphi] \, d\xi \\
                &= \int_{\mathbb{R}^2} \mathcal{F}[\varphi] \, d\xi = \varphi(0) = \delta_0(\varphi).
    \end{align*}
    }%
    As $C^\infty_0(\mathbb{R}^2)$ is dense in $H^4(\mathbb{R}^2)$ and $\delta_0:H^4(\mathbb{R}^2) \rightarrow \mathbb{R}$ is continuous, it follows $a(g,v)=v(0)$ $\forall v \in H^4(\mathbb{R}^2)$. Thus $u=g \in H^s(\mathbb{R}^2)$ for $s <7$.
\end{proof}

\bibliography{paper}
\bibliographystyle{siam}

\end{document}